\documentclass[10pt,english,reqno]{amsart} 
\calclayout

\usepackage[mathscr]{euscript}
\newcommand{\fr}{\mathfrak}
\newcommand{\cal}{\mathscr}

\newcommand{\op}{\operatorname}
\newcommand{\tn}{\textnormal}

\newcommand{\xysmall}{\xymatrix@C=1.5em@R=1.5em}

\newcommand{\loo}[1]{(\!(#1)\!)}
\newcommand{\arc}[1]{[\![#1]\!]}

\newcommand{\Spec}{\mathrm{Spec}}

\newcommand{\dR}{\mathrm{dR}}

\newcommand{\QCoh}{\mathrm{QCoh}}

\newcommand{\et}{\textnormal{\'et}}
\newcommand{\eh}{\textnormal{\'eh}}
\newcommand{\bfeh}{\textbf{\'eh}}

\newcommand{\Mod}{\textnormal{-}\mathrm{Mod}}

\newcommand{\Ran}{\mathrm{Ran}}

\newtheorem*{thm*}{Theorem}
\newtheorem{thm}{Theorem}[section]
\newtheorem{thmx}{Theorem}

\newtheorem{prop}[thm]{Proposition}
\newtheorem{lem}[thm]{Lemma}

\newtheorem{cor}[thm]{Corollary}

\theoremstyle{definition}

\newtheorem{eg-eg}[thm]{Example of Example}

\newtheorem{rem}[thm]{Remark}

\usepackage[colorlinks=true]{hyperref}
\usepackage{graphicx}
\usepackage{amssymb}
\usepackage{epstopdf}
\usepackage{enumerate}
\usepackage{tikz}
\usepackage{marginnote}
\usepackage{epigraph}

\usepackage{soul}

\usepackage{mdframed}
\newmdenv[
  topline=false,
  bottomline=false,
  rightline=false,
  skipabove=\topsep,
  skipbelow=\topsep
]{siderules}

\numberwithin{equation}{section}

\usepackage{mathtools}

\usepackage{epigraph}



\usepackage{color}
\usepackage[all]{xy}

\title{Tame twistings and $\Theta$-data}

\author{Yifei Zhao}
\date{\today}
\email{yifei@math.harvard.edu}

\begin{document}
\maketitle

\begin{abstract}
The goal of this paper is to assign an intrinsic meaning to the space of quantum parameters $\tn{Par}_G$ appearing in the geometric Langlands program of Beilinson--Drinfeld. We introduce tame twistings, a variant of twisted differential operators (TDOs) for which regularity of twisted $\cal D$-modules is well-defined. Our main result is that for a proper curve $X$, $\tn{Par}_G$ is  precisely the moduli space of factorization tame twistings on the affine Grassmannian.
\end{abstract}

\setcounter{tocdepth}{1}
\tableofcontents

\section*{Introduction}

\subsection*{Quantum parameters}$ $

\smallskip

Not long since the origin of the geometric Langlands program concerning $\cal D$-modules on the moduli stack of $G$-bundles over a proper complex curve $X$ \cite{beilinson1991quantization}, it has been speculated that the entire program should have a deformation related to the $1$-parameter family of quantum groups $U_q(\fr g)$ deforming the universal enveloping algebra \cite{stoyanovsky2006quantum} \cite{gaitsgory2016quantum}. 

\smallskip

For a simple group $G$, there is a natural candidate for such a deformation. Namely, the stack $\tn{Bun}_G$ has a determinant line bundle $\det_{\fr g}$ and one may consider $\cal D$-modules twisted by any of its complex power $\det_{\fr g}^c$. The quantum Langlands program, therefore, asks for a spectral interpretation of the twisted category $\cal D\Mod^c(\tn{Bun}_G)$ in terms of the dual group $\check G$. As the category $\cal D\Mod^c(\tn{Bun}_G)$ receives a functor from reprensentations of the Kac--Moody Lie algebra at level $c$, it indeed relates to representations of $U_q(\fr g)$ for $q = \exp(2\pi i c)$ \cite{gaitsgory2008twisted}.

\smallskip

This paper is devoted to answering the following (apparently ill-posed) question: \emph{what does the parameter $c$ mean?}

\smallskip

To begin with, the relationship between $\det_{\fr g}$ and the Killing form suggests that for a reductive group $G$, the number $c$ should be replaced by a Weyl-invariant bilinear form $\kappa$ on the Cartan subalgebra $\fr t$. On the other hand, the study of parabolic induction indicates that $\kappa$ is not the only relevant part of quantum parameters---reduction from $G$ to $T$ acquires a shift by a sheaf of twisted differential operators (TDO) on $\tn{Bun}_T$ which can be attributed to an extension of $\cal O_X$-modules of the following form \cite[\S3.3]{gaitsgory2016eisenstein}:
$$
0 \rightarrow \omega_X \rightarrow E \rightarrow \fr t\otimes\cal O_X \rightarrow 0.
$$
Incorporating this ``quantum anomaly'' led to the definition of the parameter space $\tn{Par}_G$ as pairs $(\kappa, E)$ where $\kappa$ is as before, and $E$ is an $\omega_X$-extension of $\fr z_G\otimes\cal O_X$, for $\fr z_G$ being the center of the Lie algebra $\fr g$.

\smallskip

This definition of quantum parameters turns out to be quite convenient. In \cite{zhao2017quantum}, it is observed that each $(\kappa, E)$ gives rise directly to a TDO on $\tn{Bun}_G$ and thus to a category of twisted $\cal D$-modules, bypassing line bundles. However, it has been unclear what the nature of such pairs $(\kappa, E)$ is. The na\"ive guess that they parametrize all TDOs on $\tn{Bun}_G$ is already wrong for a torus: $\tn{Bun}_T$ has infinitely many connected components labeled by the cocharacter lattice $\Lambda_T$.

\subsection*{Factorization twistings}$ $

\smallskip

A more sensible guess is that $\tn{Par}_G$ parametrizes factorization twistings on the Beilinson--Drinfeld affine Grassmannian $\tn{Gr}_{G, \tn{Ran}}$. The object $\tn{Gr}_{G, \tn{Ran}}$ can be viewed as a local avatar of $\tn{Bun}_G$, attached to any collection of points $x^{(i)}$ in the base curve. The formal way to say this is that $\tn{Gr}_{G, \tn{Ran}}$ is a prestack over the \emph{Ran space} of $X$. In fact, the projection $\tn{Gr}_{G, \tn{Ran}}\rightarrow \tn{Ran}$ is a filtered colimit of schematic morphisms, though not smooth ones. The factorization structure on $\tn{Gr}_{G, \tn{Ran}}$ describes how its fibers merge as distinct points collide.

\smallskip

On the other hand, Gaitsgory--Rozenblyum \cite{gaitsgory2011crystals} introduced the notion of a \emph{twisting} as a natural generalization of TDOs to non-smooth schemes, so it makes sense to study twistings on $\tn{Gr}_{G, \tn{Ran}}$ which respect the factorization structure. They are called \emph{factorization twistings}. This discussion does not involve the global geometry of $X$, so one may even drop the assumption that $X$ is proper.

\smallskip

For a torus $T$, a twisting on $\tn{Gr}_{T}$ might appear differently on each connected component $\tn{Gr}_T^{\lambda}$, but factorization forces the distinct components to interact. On the other hand, imposing factorization is natural for the purpose of the Langlands program. In order to make contact with spectral data, $G(\cal O)$-equivariant twisted $\cal D$-modules on $\tn{Gr}_G$ should form a Tannakian category, where the symmetry constraint arises from the factorization structure \cite{mirkovic2007geometric}. This would not be possible if the twisting defining $\cal D$-modules itself lacked factorization.

\smallskip

We do not yet know whether $\tn{Par}_G$ parametrizes factorization twistings aside from the case of a semisimple, simply connected group $G$\footnote{where the answer is affirmative, see below; however, we suspect the answer to be false in general.}. In this paper, we show that when $X$ is proper, $\tn{Par}_G$ instead parametrizes a variant of twistings, called \emph{tame twistings}. The following result appears as Theorem \ref{thm-twisting-classification} in the main text.

\begin{thmx}
\label{thmx-twisting}
For a proper, smooth, connected curve $X$ and a reductive group $G$, the category of factorization tame twistings on $\tn{Gr}_{G, \tn{Ran}}$ is canonically equivalent to $\tn{Par}_G$.
\end{thmx}

Like usual twistings, tame twistings are objects of algebraic geometry and exist over any ground field $k=\bar k$ with $\tn{char}(k) = 0$. Before giving a precise definition, we mention several aspects of this notion that explain how it appears ``in nature.''
\begin{enumerate}[(a)]
	\item A usual twisting on a \emph{smooth} scheme $X$ is a torsor for the complex $\Omega_X^1 \rightarrow \Omega_X^{2,\tn{cl}}$, whereas a tame twisting is a torsor for the subsheaf $\mathring{\Omega}_X^1$ of $\Omega_X^1$ whose sections over $U$ consists of differentials with logarithmic growth along a good compactification $\overline U$ of $U$. 
	
	\smallskip
	
	In particular, the process of inducing twistings from line bundles factors through tame twistings by the map $d\log : \cal O_X^{\times} \rightarrow \mathring{\Omega}_X^1$.
	\smallskip
	\item In contrast to usual twistings, the category of $\cal D$-modules twisted by a tame twisting has a natural notion of \emph{regularity} generalizing the usual notion of regular $\cal D$-modules.
	\smallskip
	\item A tame twisting has an underlying \emph{tame gerbe}. Furthermore, when $k=\mathbb C$, tame gerbes on $X$ form a full subcategory of gerbes on the analytification $X^{\tn{an}}$ banded by the constant group $\mathbb C^{\times}$.
\end{enumerate}

\noindent
These properties suggest that tame twistings naturally arise when we consider twisted $\cal D$-modules in conjunction with complex constructible sheaves---this is, indeed, something one does for the purpose of the geometric Langlands program (see \cite{liu2019semi}, for example). We emphasize that tameness is not a condition of a twisting, but an additional piece of structure.

\smallskip

The properness hypothesis in the statment of Theorem \ref{thmx-twisting} is artificial in the following sense. The actual result we shall prove is that factorization tame twistings are paramterized by a modified groupoid $\mathring{\tn{Par}}_G$, regardless of properness of $X$. It consists of pairs $(\kappa, \mathring E)$ where $\kappa$ is as before and $\mathring E$ is an extension of Zariski sheaves valued in $k$-vector spaces:
$$
0 \rightarrow \mathring{\Omega}_X^1 \rightarrow \mathring E\rightarrow \fr z_G \rightarrow 0.
$$
It just so happens that when $X$ is proper, the datum of  $\mathring E$ is equivalent to that of $E$. In the non-proper case, there are advantages of taking $\mathring{\tn{Par}}_G$ as the definition of quantum parameters as opposed to $\tn{Par}_G$. Besides its closer relationship with analytic objects, $\mathring{\tn{Par}}_G$ has the structure of an algebraic stack with \emph{finite}-dimensional automorphism groups (hence $1$-affine, see \cite{gaitsgory2015sheaves}).

\smallskip

Tame twistings are $k$-linear objects, and the equivalence of Theorem \ref{thmx-twisting} which we shall produce respects $k$-linearity. Thus we automatically obtain an equivalence of $k$-linear stacks.

\smallskip

We also give a partial answer to the classification problem of usual factorization twistings on $\tn{Gr}_{G, \tn{Ran}}$, as it is interesting in its own right. The following result appears as Theorem \ref{thm-usual-twisting-classification}, where the curve $X$ is only assumed to be smooth and connected.

\begin{thmx}
\label{thmx-usual-twisting}
Suppose $G$ is semisimple and simply connected. Then the category of factorization twistings on $\tn{Gr}_{G, \tn{Ran}}$ is canonically equivalent to Weyl-invariant symmetric bilinear forms on $\fr t$.
\end{thmx}

\noindent
In particular, for a semisimple and simply connected group $G$, a usual factorization twisting on $\tn{Gr}_{G, \tn{Ran}}$ is canonically tame. This is not the case for more general $G$.

\subsection*{More on quantum paramters}$ $

\smallskip

From the perspective of the Langlands program, the role played by quantum parameters in the $\cal D$-module context is analogous to the Brylinski--Deligne data. The latter are central extensions $\mathbf E$ of $G$ by the big Zariski sheaf of the second algebraic K-group $\mathbf K_2$ and are used to produce metaplectic coverings of the ad\`elic group $G(\mathbb A_{\mathbf F})$ in the usual Langlands program \cite{brylinski2001central}.

\smallskip

By Gaitsgory \cite{gaitsgory2018parameterization}, the groupoid of Brylinski--Deligne data $\mathbf{CExt}(G, \mathbf K_2)$ admits a functor $\Xi_{\mathbf{Pic}}$ to factorization line bundles on $\tn{Gr}_{G, \Ran}$, which is futhermore an equivalence \cite{tao2019extensions}. We shall explicitly identify the composition:
\begin{align*}
	\mathbf{CExt}(G, \mathbf K_2) \underset{\sim}{\xrightarrow{\Xi_{\mathbf{Pic}}}} & \mathbf{Pic}^{\tn{fact}}(\tn{Gr}_{G, \Ran}) \\
	& \rightarrow \mathring{\mathbf{Tw}}{}^{\tn{fact}}(\tn{Gr}_{G, \Ran}) \underset{\sim}{\xrightarrow{\Psi_{\mathring{\mathbf{Tw}}}}} \mathring{\tn{Par}}_G.
\end{align*}
as a combination of the standard procedure of extracting a quadratic form from a central extension by $\mathbf K_2$ and the functor of ``taking the derivative'' of $\mathbf E$ when restricted to the center of $G$ (Corollay \ref{cor-compatibility-bd-data}). This expresses a kind of compatibility between the classification theorem of Brylinski--Deligne \cite{brylinski2001central} and our classification of factorization tame twistings.

\subsection*{Implementing tameness}$ $

\smallskip

Let us now give a precise definition of $\mathring{\mathbf{Tw}}(X)$ for an arbitrary finite type scheme $X$ over $k$. This turns out to be slightly technical, because we simultaneously want $\mathring{\mathbf{Tw}}$ to have strong descent properties (like the usual twistings) and to retain the explicit description as Zariski $\mathring{\Omega}^1$-torsors over a smooth scheme.

\smallskip

Concretely, we first define $\mathring{\mathbf{Ge}}$ as the $\eh$-sheafification of the classifying ($2$-)stack of the stack of rank--$1$ \emph{regular} local systems, in the sense of $\cal D$-modules. We call $\mathring{\mathbf{Ge}}$ the stack of \emph{tame gerbes}. There is a canonical map from $\mathring{\mathbf{Ge}}$ to the \emph{derived} $\bfeh$-sheafification\footnote{We use bold characters to emphasize topologies defined on derived schemes.} of $\op B^2\mathbb G_m$ and we let $\mathring{\mathbf{Tw}}$ be the fiber of this map. Analogous to their usual counterparts, we have a fiber sequence relating line bundles, tame twistings, and tame gerbes:
$$
\mathbf{Pic}(X) \rightarrow \mathring{\mathbf{Tw}}(X) \rightarrow \mathring{\mathbf{Ge}}(X).
$$
We are forced to work with derived schemes in defining $\mathring{\mathbf{Tw}}$, as even usual twistings satisfy derived $\mathbf h$-descent but fail classical $\tn h$-descent, a fact which ultimately boils down to the derived $\mathbf h$-descent of perfect complexes due to Halpern-Leistner--Preygel \cite{halpern2014mapping}. On the other hand, there is no problem in building $\mathring{\mathbf{Ge}}$ on classical schemes because the resulting stack is nil-invariant.

\smallskip

Instead of the $\tn h$-topology, we choose to work with the weaker $\eh$-topology because we need the restriction of $\mathring{\mathbf{Tw}}$ to smooth schemes to recover $\mathring{\Omega}^1$-torsors. This relies on an $\eh$-to-\'etale comparison theorem for the cohomology of $\mathbb G_m$ due to T.~Geisser \cite{geisser2006arithmetic}. We will also need to calculate the $\eh$-cohomology groups of the sheaf $\mathring{\Omega}^1$. These turn out to be very calculable after establishing the fact that $\mathring{\Omega}^1$ is an $\mathbb A^1$-invariant $\tn h$-sheaf with transfer.

\smallskip

In fact, $\mathring{\Omega}^1$ is just the first piece in a family of sheaves $\mathring{\Omega}^p$, for all $p\ge 0$, which we call ``differential forms of moderate growth.'' They are all $\mathbb A^1$-invariant $\tn h$-sheaves with transfer on the category of smooth schemes. Regarded as Zariski sheaves, they are related by a Gersten resolution:
	$$
	\mathring{\Omega}^p \rightarrow \bigoplus_{x\in X^{(0)}} (i_x)_* \mathring{\Omega}^p(x) \rightarrow \bigoplus_{x\in X^{(1)}} (i_x)_*\mathring{\Omega}^{p-1}(x) \rightarrow \cdots \rightarrow \bigoplus_{x\in X^{(p)}} (i_x)_*k,
	$$
whose existence can either be seen as a consequence of Mazza--Voevoedsky--Weibel \cite{mazza2011lecture} or the Bloch--Ogus theorem \cite{bloch1974gersten} combined with elementary facts from mixed Hodge theory.

\subsection*{From factorization to $\Theta$-data}$ $

\smallskip

We now sketch the proofs of Theorems \ref{thmx-twisting} and \ref{thmx-usual-twisting}.

\smallskip

The first step in our proof of Theorem \ref{thmx-twisting} is to recognize $\mathring{\tn{Par}}_G$ as a kind of ``enhanced $\Theta$-data.'' Let us explain what these are. Recall that Brylinski--Deligne \cite{brylinski2001central} classified central extensions of a torus $T$ by $\mathbf K_2$ over the base $X$ by the following groupoid. It consists of pairs $(q, \cal L^{(\lambda)})$ where $q$ is an integral quadratic form on $\Lambda_T$ and $\cal L^{(\lambda)}$ is a $\Lambda_T$-indexed system of line bundles over $X$. They are equipped with multiplicative structures:
$$
c_{\lambda,\mu} : \cal L^{(\lambda)} \otimes \cal L^{(\mu)} \rightarrow \cal L^{(\lambda + \mu)},
$$
which are associative, but only commutative up to a $\kappa$-twist, for $\kappa$ being the bilinear form assocaited to $q$. The same groupoid showed up in the study of chiral algebras \cite{beilinson2004chiral} and was called \emph{even $\Theta$-data}. Since we do not need $\mathbb Z/2\mathbb Z$-grading, we shall refer to this groupoid simply as \emph{$\Theta$-data} of the lattice $\Lambda_T$. The Brylinski--Deligne classification for a reductive group $G$ involves a $\Theta$-datum for the co-weight lattice as well as a certain isomorphism $\varepsilon$ of two $\Theta$-data for the co-root lattice. We call the groupoid of such gadgets \emph{enhanced $\Theta$-data}.

\smallskip

It is straightforward to see that $\mathring{\tn{Par}}_G$ identifies with enhanced $\Theta$-data when we replace the value group of $q$ by $k$, and the system of line bundles $\cal L^{(\lambda)}$ by a system of tame twistings. Moreover, we shall formalize a general \emph{theory of gerbes} to be an \'etale stack $\mathbf G$ valued in strictly commutative Picard $2$-groupoids, which receives a map (``first Chern class''):
$$
c_1 : \mathbf{Pic} \underset{\mathbb Z}{\otimes} A(-1) \rightarrow \mathbf G,\quad (\cal L, a)\leadsto \cal L^a,
$$
with $A(-1)$ being a certain coefficient group associated to $\mathbf G$. Then there is a sensible notion of enhanced $\Theta$-data for a theory of gerbes $\mathbf G$, denoted by $\Theta_G(\Lambda_T; \mathbf G)$. This paradigm applies to line bundles, twistings (tame or usual), as well as gerbes in various sheaf-theoretic contexts.

\smallskip

Roughly speaking, we will build a functor from various factorization gadgets to their corresponding groupoids of enhanced $\Theta$-data. The canonicity of the construction produces a morphism of fiber sequences of Picard 2-groupoids.
$$
\xysmall{
	\mathbf{Pic}^{\tn{fact}}(\op{Gr}_{G, \tn{Ran}}) \ar[r]\ar[d]^{\Psi_{\mathbf{Pic}}} & \mathring{\mathbf{Tw}}{}^{\tn{fact}}(\op{Gr}_{G,\tn{Ran}}) \ar[r]\ar[d]^{\Psi_{\mathring{\mathbf{Tw}}}} & \mathring{\mathbf{Ge}}{}^{\tn{fact}}(\tn{Gr}_{G, \tn{Ran}}) \ar[d]^{\Psi_{\mathring{\mathbf{Ge}}}} \\
	\Theta_G(\Lambda_T; \mathbf{Pic}) \ar[r] & \Theta_G(\Lambda_T; \mathring{\mathbf{Tw}}) \ar[r] & \Theta_G(\Lambda_T; \mathring{\mathbf {Ge}})
}
$$
Then we will prove that $\Psi_{\mathring{\mathbf{Tw}}}$ is an equivalence on all homotopy groups, thereby deducing Theorem \ref{thmx-twisting}. This will follow from showing that $\Psi_{\mathbf{Pic}}$ and $\Psi_{\mathring{\mathbf{Ge}}}$ are both equivalences and that $\Psi_{\mathring{\mathbf{Tw}}}$ is surjective on $\pi_0$. The joint work with J.~Tao \cite{tao2019extensions} shows that $\Psi_{\mathbf{Pic}}$ is an equivalence, so a significant step of the proof already exists. A direct argument exploiting the $k$-linear structure of tame twistings then shows that $\Psi_{\mathring{\mathbf{Tw}}}$ is essentially surjective.

\smallskip

At this point, it is tempting to use the aforementioned fact that tame gerbes form a full subcategory of analytic $\mathbb C^{\times}$-gerbes and reduce the statement about $\Psi_{\mathring{\mathbf{Ge}}}$ to Reich \cite[Theorem II.7.3]{reich2012twisted}. However, we avoid this input as the proof in \emph{loc.cit.}~relies on several errors and consequently yielded an incorrect classification statement.\footnote{Contrary to the assertion of \cite[Theorem II.7.3]{reich2012twisted}, the fiber sequence has no canonical splitting and its proof used an incorrectly defined splitting (Proposition II.3.6, Proposition II.7.5). Furthermore, two steps in the proof applied cohomological purity of divisors to non-(ind-)smooth schemes (Lemma II.7.6 and Proposition III.2.8).} Instead, we supply a proof using a different strategy.

\smallskip

In fact, we will provide a uniform proof for gerbes in various sheaf-theoretic contexts. We remove the restriction on $\tn{char}(k)$ but fix a sufficiently strong topology $t$ which allows for resolution of singularities. Then we characterize those theories of gerbes which are ``motivic.'' The properties included are purity, $\mathbb A^1$-homotopy invariance, $t$-descent, and a weak form of proper base change. The following result appears as Theorem \ref{thm-gerbe-classification}.

\begin{thmx}
\label{thmx-gerbe}
Let $\mathbf G$ be a motivic $t$-theory of gerbes. Then we have an equivalence of categories:
$$
\Psi_{\mathbf G} : \mathbf G^{\tn{fact}}(\tn{Gr}_{G, \tn{Ran}}) \xrightarrow{\sim} \Theta_G(\Lambda_T; \mathbf G).
$$
\end{thmx}

\smallskip

Besides tame gerbes and analytic $\mathbb C^{\times}$-gerbes, Theorem \ref{thmx-gerbe} also applies to \'etale gerbes valued in suitable torsion abelian groups. The latter has been used in Gaitsgory--Lysenko \cite{gaitsgory2018parameters} to define geometric metaplectic dual data.

\smallskip

Recently, various other sheaf theories have been studied in the context of the affine Grassmannian and the Satake equivalence---there are the perverse $\mathbb F_p$-sheaves of R.~Cass \cite{cass2019perverse}, the stratified mixed Tate motives of Richarz--Scholbach \cite{richarz2019motivic}, among others. We hope that our formulation would be useful for generalizing their results to the metaplectic setting.

\smallskip

In the case $k=\mathbb C$, we summarize the relationship between the various twisting agents in the following diagram on the left. In the special case where $G$ is a \emph{simple}, \emph{simply connected} group, their classifications are depicted in the diagram on the right.
$$
\xymatrix@C=-4em@R=1em{
	& \mathbf{Pic}^{\tn{fact}}(\tn{Gr}_{G, \tn{Ran}}) \ar[d] \\
	& \mathring{\mathbf{Tw}}{}^{\tn{fact}}(\tn{Gr}_{G, \tn{Ran}}) \ar[dl]\ar[dr] & \\
	\mathbf{Tw}^{\tn{fact}}(\tn{Gr}_{G, \tn{Ran}}) & & \mathring{\mathbf{Ge}}{}^{\tn{fact}}(\tn{Gr}_{G, \tn{Ran}}) \ar[dr] \\
	& & & \mathbf{Ge}_{\tn{an}}^{\tn{fact}}(\tn{Gr}_{G, \tn{Ran}})
}
\quad
\xymatrix@C=2em@R=1.3em{
	& \mathbb Z \ar[d] \\
	& \mathbb C \ar[dl]_{\tn{id}}\ar[dr] & \\
	\mathbb C & & \mathbb C/\mathbb Z \ar[dr]^{\exp(2\pi i-)} \\
	& & & \mathbb C^{\times}.
}
$$
The recent (unpublished) works of Chen--Fu and R.-T.~Yang on representations of $U_q(\fr g)$ suggest that there is a paradigm of equivalences of \emph{factorization categories}. Namely, one starts with an object $\cal T$ of $\mathring{\mathbf{Tw}}{}^{\tn{fact}}(\tn{Gr}_{G, \tn{Ran}})$ and tries to relate certain factorzation categories of $\cal T$-twisted crystals with certain factorization categories of $\cal G$-twisted constructible sheaves, for $\cal G$ being its image in $\mathbf{Ge}_{\tn{an}}^{\tn{fact}}(\tn{Gr}_{G, \tn{Ran}})$. Our classification of these gadgets by enhanced $\Theta$-data can hopefully contribute to their line of research.

\smallskip

Thus, Theorem \ref{thmx-gerbe} is in part motivated by a desire to perform ``community service.''

\smallskip

Having Theorem \ref{thmx-gerbe} at our disposal, we apply it to a theory of gerbes which is \emph{not} used to twist any category of sheaves. Namely, we consider the stack which associates to $X$ the groupoid of $\mathbb G_a$-gerbes on $X_{\dR}$. The fact that this theory of gerbes is motivic follows from usual facts about algebraic de Rham cohomology. Finally, Theorem \ref{thmx-usual-twisting} follows from this result combined with the Borel--Weil--Bott theorem on affine Schubert varieties.

\subsection*{Organization of the paper}$ $

\smallskip

The paper is roughly split into two parts. Sections \S\ref{sec-topology}-\ref{sec-ge-and-tw} are devoted to developping the notion of tame gerbes and tame twistings. These require the $\tn{char}(k)=0$ assumption to allow for Hironaka's resolution of singularities. Sections \S\ref{sec-motivic-gerbes}-\ref{sec-classification-proof} formulate and prove the main classification theorems, with applications to various sheaf-theoretic contexts.

\smallskip

We first record some preliminary facts about the $\eh$-topology and its derived analogue in \S\ref{sec-topology}. These do no go beyond the work of Geisser \cite{geisser2006arithmetic}, Friedlander--Voevodsky \cite{friedlander2000bivariant}, and Halpern-Leistner--Preygel \cite{halpern2014mapping}.

\smallskip

In \S\ref{sec-diff-forms}, we study the sheaves $\mathring{\Omega}^p$ systematically, for all $p\ge 0$. Their basic properties follow from mixed Hodge theory. The $\tn h$-descent is proved by comparing $\mathring{\Omega}^p$ with the $\tn h$-sheafification of $\Omega^p$ studied by H\"uber--Jorder \cite{huber2013differential}. Then a series of cohomological comparison results follow from the theorems of Voevodsky and Scholbach \cite{scholbach2012geometric}, so we end up only needing to calculate the Zariski cohomology of $\mathring{\Omega}^p$, where a Gersten resolution supplies the required tools.

\smallskip

We gather these ingredients to define tame gerbes and tame twistings in \S\ref{sec-ge-and-tw}. We prove that tame twistings satisfy various expected properties and can be used to form a twisted category of $\cal D$-modules, which possesses a notion of regularity.

\smallskip

In \S\ref{sec-motivic-gerbes}, we formulate a \emph{motivic $t$-theory of gerbes} for a sufficiently strong topology $t$. Then we verify that \'etale mod-$\ell$ gerbes, complex analytic gerbes, as well as tame gerbes are examples of such motivic theories. By contrast, tame twistings form a theory of gerbes according to our definition, but not a motivic one.

\smallskip

The next \S\ref{sec-fact-theta} contains all the main results of this paper.  We first define \emph{enhanced $\Theta$-data} $\Theta_G(\Lambda_T; \mathbf G)$ attached to a theory of gerbes $\mathbf G$. Then we recall the classification of factorization line bundles by integral enhanced $\Theta$-data, established in \cite{tao2019extensions}. Then we state Theorem \ref{thmx-gerbe} and deduce Theorems \ref{thmx-twisting} and \ref{thmx-usual-twisting} from it. The actual argument is less formal than what we sketched above, because to define the functor $\Psi_{\mathring{\mathbf{Tw}}}$ for an arbitrary reductive group $G$ requires knowledge about its behavior for tori and semisimple, simply connected groups. We prove the compatibility statement between quantum parameters and Brylinski--Deligne data alluded to above, although there seems to be more mathematics on this topic that remains to be explored.

\smallskip

Finally, we prove Theorem \ref{thmx-gerbe} in \S\ref{sec-classification-proof}. Roughly speaking, we use the classification of factorization line bundles to supply enough factorization gerbes, and appeal to the motivic properties of $\mathbf G$ to ensure that there are not too many of them.

\subsection*{Notations}$ $

\smallskip

Throughout the paper, we work over a ground field $k=\bar k$.

\smallskip

By a \emph{scheme} we shall always mean a separated (classical) scheme over $k$, and we denote by $\mathbf{Sch}_{/k}$ the category they form. The notation $\mathbf{Sch}^{\tn{ft}}_{/k}$ will mean (separated) schemes of finite type over $k$. We let $\mathbf{Sm}_{/k}$ denote its full subcategory consisting of smooth schemes.

\smallskip

Our convention on ind-schemes is as follows. We call an \emph{ind-scheme} a presheaf on $\mathbf{Sch}_{/k}$ which can be represented as a filtered colimit $\underset{\nu}{\tn{colim}}\;X^{(\nu)}$ where each $X^{(\nu)}$ belongs to $\mathbf{Sch}_{/k}$, each morphism $X^{(\nu)} \rightarrow X^{(\nu')}$ is a closed immersion, and the index category has cardinality $\le |\aleph_0|$. The category of ind-schemes is denoted by $\mathbf{IndSch}_{/k}$. It has a full subcategory $\mathbf{IndSch}^{\tn{ft}}_{/k}$, which consists of \emph{ind-finite type} ind-schemes, i.e., we can take each $X^{(\nu)}$ to lie in $\mathbf{Sch}^{\tn{ft}}_{/k}$ in a colimit presentation as above.

\smallskip

We will need to consider presheaves on $\mathbf{Sch}^{\tn{ft}}_{/k}$ valued in $2$-groupoids. However, we find it convenient to import the theory of $\infty$-groupoids and use the well-developped theory of algebras and modules in them \cite{lurie2009higher} \cite{lurie2016higher}. We will denote by $\mathbf{Spc}$ the $\infty$-category of $\infty$-groupoids in the sense of Lurie. By a \emph{presheaf} $\cal F$ on $\mathbf{Sch}^{\tn{ft}}_{/k}$, we will mean a $\mathbf{Spc}$-valued presheaf unless otherwise stated. The $\infty$-category they form is denoted by $\tn{PSh}(\mathbf{Sch}^{\tn{ft}}_{/k})$. For a topology $t$ on $\mathbf{Sch}^{\tn{ft}}_{/k}$, we denote by $\tn{Shv}_t(\mathbf{Sch}^{\tn{ft}}_{/k})$ the full subcategory of $t$-sheaves. Given $\cal F \in \tn{PSh}(\mathbf{Sch}^{\tn{ft}}_{/k})$, its $t$-sheafification is denote by $\cal F_t$.

\smallskip

Although most of this paper stays within classical algebraic geometry, for the definition of a tame twisting we will need derived schemes. Thus we let $\mathbf{DSch}_{/k}$ denote the $\infty$-category of (separated) derived schemes over $k$, locally modeled on simplicial commutative $k$-algebras. The full subcategory $\mathbf{DSch}^{\tn{ft}}_{/k}$ denotes finite type derived schemes, i.e., $X\in\mathbf{DSch}_{/k}$ whose underlying classical scheme is of finite type and $\cal O_X$ is a coherent $\pi_0\cal O_X$-module (in particular eventually coconnective). Tautologically, we have inclusions:
$$
\mathbf{Sm}_{/k} \subset \mathbf{Sch}^{\tn{ft}}_{/k} \subset \mathbf{DSch}^{\tn{ft}}_{/k},
$$
where neither functor preserves fiber products.

\smallskip

In fact, we only need derived schemes when $\tn{char}(k) = 0$, so one can take the equivalent theory modeled on connective commutative DG algebras over $k$, as is done in \cite{gaitsgory2017study}. The theory of ind-coherent sheaves as well as left and right crystals have been developed in this context \cite{gaitsgory2011ind} \cite{gaitsgory2011crystals}.

\smallskip

By a \emph{reductive} group $G$, we always refer to a connected reductive group defined over $k$. We will use $G_{\tn{der}}$ to denote its derived subgroup, and $\widetilde G_{\tn{der}}$ its universal cover. Thus $\widetilde G_{\tn{der}}$ is a semisimple, simply connected group. The letter $T$ denotes a maximal torus of $G$, and $B$ denotes a Borel with nilpotent radical $N$.

\smallskip

We use ``covariant notations'' for the root data of $G$. More precisely, $\Lambda_T := \tn{Hom}(\mathbb G_m, T)$ is the co-character lattice, whereas $\check{\Lambda}_T := \tn{Hom}(T, \mathbb G_m)$ is the character lattice. Let $\Lambda_T^r$ (resp.~$\check{\Lambda}_T^r$) denote the sublattice spanned by co-roots (resp.~roots). Then the algebraic fundamental group of $G$ is the quotient $\Lambda_T/\Lambda_T^r$. We use $\Phi$ and $\check{\Phi}$ to denote the co-root and root systems, and $\Delta$ and $\check{\Delta}$ to denote the choice of simple co-roots and roots determined by $B$.

\smallskip

The objects associated to $G_{\tn{der}}$ and $\widetilde G_{\tn{der}}$ are decorated in the same manner. For example, $\Lambda_{\widetilde T_{\tn{der}}}$ is the co-character lattice of the maximal torus $\widetilde T_{\tn{der}}\subset\widetilde G_{\tn{der}}$ corresponding to $T$. In fact, $\Lambda_{\widetilde T_{\tn{der}}}$ canonically identifies with $\Lambda_T^r$.

\subsection*{Acknowledgments} I thank Dennis Gaitsgory both for suggesting this problem in 2016 and for the numerous helpful conversations that followed. In fact, the notion of tame twistings emerged from one of these conversations.

\smallskip

I am grateful to James Tao for the collaboration \cite{tao2019extensions} as the classification theorems in the current paper can be seen as an outgrowth of \emph{loc.cit.}. Many of his ideas are thus present here.

\smallskip

I thank Ruotao Yang for pointing out an error in an earlier draft, and to Sasha Beilinson for relating $\mathring{\Omega}^p$ to Bloch--Ogus theory. I also thank Dori Bejleri, Lin Chen, Elden Elmanto, and Yuchen Fu for helpful conversations related to this work.

\bigskip

\section{Some topologies}
\label{sec-topology}

In this section, we recall the definition of the $\eh$-topology and introduce its analogue for derived schemes. The results which will be used in the sequal are the two comparison lemmas between $\eh$ and \'etale cohomology (Lemma \ref{lem-eh-to-etale-gm} and \ref{lem-eh-to-etale}) and interactions between the classical and derived $\eh$-topology in \S\ref{sec-derived-eh-topology}.

\subsection{Classical $\eh$-topology}

\subsubsection{} Recall the $\tn h$-topology on $\mathbf{Sch}_{/k}^{\tn{ft}}$ introduced by V.~Voevodsky \cite[\S 3]{voevodsky1996homology}. Its coverings are generated by universal topological epimorphisms. In fact, a presheaf $\cal F$ on $\mathbf{Sch}^{\tn{ft}}_{/k}$ is an $\tn h$-sheaf if and only if it satisfies descent with respect to Nisnevich (or \'etale) covers and proper surjections\footnote{D.~Gaitsgory has kindly pointed out that Nisnevich can be weakened to Zariski, thank to a theorem of Goodwillie--Lichtenbaum \cite[Theorem 4.1]{goodwillie2001cohomological}.}. By de Jong's alteration, every scheme $X\in\mathbf{Sch}^{\tn{ft}}_{/k}$ is $\tn h$-locally smooth.

\subsubsection{} In this paper, we will extensively use the $\eh$-topology on $\mathbf{Sch}^{\tn{ft}}_{/k}$ introduced by Geisser \cite{geisser2006arithmetic}. It is generated by \'etale coverings and abstract blow-up squares.

\smallskip

The following diagram summarizes its relationship to several other topologies on $\mathbf{Sch}^{\tn{ft}}_{/k}$, where $\preceq$ denotes the ``coarser than'' relation.
$$
\xysmall{
	\tn{cdh} \ar@{}[r]|{\preceq} & \eh \ar@{}[r]|{\preceq} & \tn{h} \\
	\tn{Nis} \ar@{}[r]|{\preceq} \ar@{}[u]|{\reflectbox{\rotatebox[origin=c]{90}{$\preceq$}}} & \et \ar@{}[u]|{\reflectbox{\rotatebox[origin=c]{90}{$\preceq$}}}
}
$$
In fact, the $\eh$ topology bears the same relationship to the \'etale topology as the $\tn{cdh}$ topology (c.f.~Voevodsky \cite{voevodsky1996homology}) does to the Nisnevich topology.

\subsubsection{} Let us recall the definition of $\eh$. A Cartesian square in $\mathbf{Sch}^{\tn{ft}}_{/k}$:
\begin{equation}
\label{eq-blowup}
\xysmall{
	E \ar[r]\ar[d] & Y \ar[d]^p \\
	Z \ar[r]^i & X
}
\end{equation}
is an \emph{abstract blow-up square} if $i$ is a closed immersion, $p$ is a proper morphism and induces an isomorphism $Y\backslash E\xrightarrow{\sim} X\backslash Z$. Let $t_0$ denote the coarsest topology on $\mathbf{Sch}^{\tn{ft}}_{/k}$ including the empty sieve of $\emptyset$ and the sieve generated by $\{p,i\}$ for every abstract blow-up square \eqref{eq-blowup} as coverings.

\subsubsection{} Abstract blow-up squares are obviously stable under pullback and given an abstract blow-up square \eqref{eq-blowup}, the induced square:
$$
\xysmall{
E \ar[d]\ar[r] & Y \ar[d]^{\Delta_p} \\
E\underset{Z}{\times}E \ar[r]^-{(i,i)} & Y\underset{X}{\times} Y
}
$$
is again an abstract blow-up square \cite[Lemma 2.14]{voevodsky2010unstable}. Thus the conditions of \cite[Theorem 3.2.5]{asok2017affine} are satisfied and one sees that a $\mathbf{Spc}$-valued presheaf $\cal F$ on $\mathbf{Sch}^{\tn{ft}}_{/k}$ is a $t_0$-sheaf if and only if $\cal F(\emptyset)$ is contractible and for every abstract blow-up square \eqref{eq-blowup}, the induced square is homotopy Cartesian:
	$$
	\xysmall{
		\cal F(E) & \cal F(Y) \ar[l] \\
		\cal F(Z) \ar[u] & \cal F(X) \ar[u]\ar[l]
	}
	$$

\subsubsection{} The $\eh$-topology on $\mathbf{Sch}^{\tn{ft}}_{/k}$ is defined as the coarsest topology contaning the \'etale topology and $t_0$. In the remainder of this section, we shall assume:

\smallskip
\emph{---The ground field $k$ has $\tn{char}(k) = 0$.}

\smallskip
\noindent
Then by Hironaka's resolution of singularities, every $X\in\mathbf{Sch}^{\tn{ft}}_{/k}$ is $\eh$-locally smooth.

\subsubsection{} We note that the \'etale covering sieves together with $t_0$ define a \emph{quasi-topology} on $\mathbf{Sch}^{\tn{ft}}_{/k}$, i.e., if $S$ is a covering sieve on $X$, then for every morphism $f : Y\rightarrow X$, the pullback $f^*S$ is again a covering sieve. The presheaves on $\mathbf{Sch}^{\tn{ft}}_{/k}$ satisfying descent with respect to this quasi-topology are precisely \'etale sheaves which turn every abstract blow-up square into a homotopy Cartesian square. According to \cite[Corollary C.2]{hoyois2015quadratic}, this condition precisely characterizes the $\eh$-sheaves in $\tn{PSh}(\mathbf{Sch}^{\tn{ft}}_{/k})$.

\subsubsection{} The following Lemma describes a ``normal form'' of \'eh covers of a smooth scheme.

\begin{lem}
\label{lem-eh-normal-form}
Let $X\in\mathbf{Sm}_{/k}$. Every $\eh$-cover of $X$ has a refinement of the form $\{U_i \rightarrow X'\rightarrow X\}$ where $\{U_i \rightarrow X'\}$ is an \'etale cover and $X'\rightarrow X$ is a composition of blow-ups along smooth centers.
\end{lem}
\begin{proof}
This is \cite[Corollary 2.6]{geisser2006arithmetic}.
\end{proof}

\subsection{Lemmas of Geisser and Friedlander--Voevodsky}
\label{sec-comparison-lemmas}

\subsubsection{} We note two results comparing cohomology groups calculated in $\eh$-versus-\'etale topologies. These results apply to sheaves valued in \emph{abelian groups}, so we temporarily assume the convention that presheaves are valued in sets instead of higher groupoids.

\subsubsection{} Let us consider the inclusing of sites:
$$
\rho : \mathbf{Sm}_{/k} \rightarrow \mathbf{Sch}^{\tn{ft}}_{/k}.
$$
The $\eh$-topology on $\mathbf{Sch}^{\tn{ft}}_{/k}$ induces an \emph{$\eh$-topology on $\mathbf{Sm}_{/k}$} in the sense of \cite[Expos\'e III, \S 3.1]{bourbaki2006theorie}, i.e., it is the finest topology for which presheaf restriction along $\rho$ takes sheaves to sheaves. Furthermore, since every $X\in\mathbf{Sch}^{\tn{ft}}_{/k}$ is $\eh$-locally smooth, restriction defines an equivalence $\tn{Shv}_{\eh}(\mathbf{Sch}^{\tn{ft}}_{/k}) \xrightarrow{\sim} \tn{Shv}_{\eh}(\mathbf{Sm}_{/k})$ (Th\'eor\`eme 4.1 of \emph{loc.cit.}). We can summarize the situation in the following commutative diagram:
$$
\xysmall{
	\tn{Shv}_{\eh}(\mathbf{Sch}^{\tn{ft}}_{/k}) \ar[r]^-{\sim}\ar@{^{(}->}[d] & \tn{Shv}_{\eh}(\mathbf{Sm}_{/k}) \ar@{^{(}->}[d] \\
	\tn{PSh}(\mathbf{Sch}^{\tn{ft}}_{/k}) \ar[r]^-{\tn{Res}} & \tn{PSh}(\mathbf{Sm}_{/k})
}
$$

\subsubsection{} Passing to left adjoints, we obtain a commutative diagram:
$$
\xysmall{
	\tn{Shv}_{\eh}(\mathbf{Sch}^{\tn{ft}}_{/k})  & \tn{Shv}_{\eh}(\mathbf{Sm}_{/k}) \ar[l]_-{\sim} \\
	\tn{PSh}(\mathbf{Sch}^{\tn{ft}}_{/k}) \ar[u]_L & \tn{PSh}(\mathbf{Sm}_{/k}) \ar[u]_L \ar[l]_-{\tn{LKE}}
}
$$
In particular, the functor of left Kan extension along $\rho$ followed by $\eh$-sheafification\footnote{This composition is denoted by $\rho^*_d$ in \cite{geisser2006arithmetic} (for $d=\infty$) and by $\cal F\leadsto \cal F_{\tn{cdh}}$ in \cite{friedlander2000bivariant} for its $\tn{cdh}$ version.} identifies with $\eh$-sheafification within the presheaf category on $\mathbf{Sm}_{/k}$:
$$
L : \tn{PSh}(\mathbf{Sm}_{/k}) \rightarrow \tn{Shv}_{\eh}(\mathbf{Sm}_{/k}).
$$

\smallskip

Analogously, starting with an \'etale sheaf on $\mathbf{Sm}_{/k}$ (or any topology weaker than $\eh$), left Kan extension along $\rho$ followed by $\eh$-sheafification identifies with the functor:
\begin{equation}
\label{eq-pullback-exact}
L : \tn{Shv}_{\et}(\mathbf{Sm}_{/k}) \rightarrow \tn{Shv}_{\eh}(\mathbf{Sm}_{/k}),
\end{equation}
which is, in particular, exact.

\subsubsection{} Let $\mathbb G_{m,\eh}$ be the \'eh-sheaf on $\mathbf{Sch}^{\tn{ft}}_{/k}$ associated to $\mathbb G_m$. The following Lemma is a special case of a theorem of Geisser \cite{geisser2006arithmetic}.

\begin{lem}
\label{lem-eh-to-etale-gm}
Suppose $X\in\mathbf{Sm}_{/k}$. Then the canonical map is an isomorphism for all $i\ge 0$:
$$
\op H^i_{\et}(X; \mathbb G_m) \xrightarrow{\sim} \op H^i_{\eh}(X; \mathbb G_{m, \eh}).
$$
\end{lem}
\begin{proof}
Geisser \cite[Theorem 4.3]{geisser2006arithmetic} proves the comparison result for all motivic complexes $\mathbb Z(n)$. On the other hand, $\mathbb G_{m,\eh}[-1]$ is quasi-isomorphic to $\mathbb Z(1)$ as a complex of $\eh$-sheaves on $\mathbf{Sch}_{/k}^{\tn{ft}}$, as follows from the analogous fact for complexes in $\tn{Shv}_{\et}(\mathbf{Sm}_{/k})$ and the exactness of \eqref{eq-pullback-exact} (\cite[Lemma 4.1]{geisser2006arithmetic}).
\end{proof}

\subsubsection{}
\label{sec-corr} We now turn to a comparison result due to Friedlander--Voevodsky. Let $\mathbf{Sm}^{\tn{Cor}}_{/k}$ denote the category whose objects are the same as $\mathbf{Sm}_{/k}$, but a morphism $X\dashrightarrow Y$ is given by a $k$-linear combination of algebraic cycles $W\subset X\times Y$ which are finite over $X$. The graph construction gives a functor $\mathbf{Sm}_{/k} \rightarrow \mathbf{Sm}^{\tn{Cor}}_{/k}$, and a presheaf of abelian groups on $\mathbf{Sm}_{/k}$ has a \emph{transfer structure} if it comes equipped with an extension to $\mathbf{Sm}^{\tn{Cor}}_{/k}$. On the other hand, a presheaf $\cal F$ on $\mathbf{Sm}_{/k}$ is said to be $\mathbb A^1$-invariant, if the canonical map:
$$
\cal F(X) \rightarrow \cal F(X\times\mathbb A^1)
$$
is an isomorphism for all $X\in\mathbf{Sm}_{/k}$.

\subsubsection{} The following Lemma is the \'etale version of \cite[Theorem 5.5(1)]{friedlander2000bivariant}, whereas \emph{loc.cit.}~compares Nisnevich and $\tn{cdh}$ cohomology of an $\mathbb A^1$-invariant presheaf with transfer. Since the proofs are nearly identical, we only indicate the modifications needed.

\begin{lem}
\label{lem-eh-to-etale}
Let $\cal F$ be an $\mathbb A^1$-invariant $\eh$-sheaf with transfer on $\mathbf{Sm}_{/k}$ valued in $\mathbb Q$-vector spaces. Then for $X\in\mathbf{Sm}_{/k}$, the following canonical map is an isomorphism for all $i\ge 0$:
$$
\op H^i_{\et}(X; \cal F) \xrightarrow{\sim} \op H^i_{\eh}(X; \cal F).
$$
\end{lem}

\noindent
The assumption on rational coefficients guarantees that the forgetful functor:
\begin{equation}
\label{eq-et-to-nis}
\tn{oblv} : \tn{Shv}_{\et}(\mathbf{Sm}_{/k}; \mathbb Q) \rightarrow \tn{Shv}_{\tn{Nis}}(\mathbf{Sm}_{/k}; \mathbb Q)
\end{equation}
is exact, c.f.~\cite[Proposition 5.27]{voevodsky2000cohomological}.

\begin{proof}
Arguing as in \cite[Theorem 5.5(1)]{friedlander2000bivariant}, the Lemma reduces to the following statement: given an \emph{\'etale} sheaf $\cal F_1$ of abelian groups on $\mathbf{Sm}_{/k}$ such that the $\eh$ sheafification $(\cal F_1)_{\eh} = 0$, then for any $\mathbb A^1$-invariant pretheory\footnote{We remind the reader that presheaves with transfers are pretheories (\cite[Proposition 3.1.11]{voevodsky2000triangulated}).} $\cal G$ satisfying \'etale descent, one has:
\begin{equation}
\label{eq-ext-vanishing-eh}
\tn{Ext}^i(\cal F_1, \cal G) = 0,\quad\text{for all }i\ge 0.
\end{equation}
Analogous to \cite[Lemma 5.4]{friedlander2000bivariant}, the proof consists of two steps:
\begin{enumerate}[(a)]
	\item Establish \eqref{eq-ext-vanishing-eh} for $\cal F_1 = \tn{Coker}(\mathbb Z_{\et}(U') \rightarrow \mathbb Z_{\et}(U))$, where $U' \rightarrow U$ is a composition of $n$ blow-ups with smooth centers. An induction argument reduces to $n=1$, where the result follows from the Nisnevich version \cite[Lemma 5.3]{friedlander2000bivariant} together with the exactness of \eqref{eq-et-to-nis}.
	\smallskip
	\item Reduction to case (a). Indeed, since $\cal F_1$ is already an \'etale sheaf. Lemma \ref{lem-eh-normal-form} shows that to each section $a\in\cal F_1(U)$, one can find a sequence of blow-ups with smooth centers $p : U' \rightarrow U$ such that $p^*a = 0$. Thus the same argument as in \cite[Lemma 5.4]{friedlander2000bivariant} applies. \qedhere
\end{enumerate}
\end{proof}

\subsection{Derived $\bfeh$-topology}
\label{sec-derived-eh-topology}

\subsubsection{} We introduce a variant of the $\eh$-topology for derived schemes, based on the modified version of abstract blow-up square introduced by Halpern-Leistner--Preygel \cite{halpern2014mapping}. We call a homotopy Cartesian square of derived prestacks:
	\begin{equation}
	\label{eq-derived-blowup}
	\xysmall{
		\cal E \ar[r]\ar[d] & Y \ar[d]^p \\
		\cal Z \ar[r]^i & X
	}
	\end{equation}
a \emph{derived} abstract blow-up square if $X,Y\in\mathbf{DSch}^{\tn{ft}}_{/k}$, $i$ is the formal completion along a closed subset in the topological space $|X|$, and $p$ is proper and induces an isomorphism $Y\backslash \cal E\xrightarrow{\sim} X\backslash \cal Z$. We note that $\cal Z$ and $\cal E$ are thus objects of $\tn{Ind}(\mathbf{DSch}^{\tn{ft}}_{/k})$ (\cite[Proposition 2.1.2]{halpern2014mapping}).

\subsubsection{} Let $\mathbf t_0$ denote the coarsest topology on $\mathbf{DSch}^{\tn{ft}}_{/k}$ such that the empty sieve covers $\emptyset$ and for every derived abstract blow-up square \eqref{eq-derived-blowup}, the sieve generated by $\{p, i\}$ is a covering sieve of $X$.

\smallskip

To give an alternative description, let $\mathbf S$ denote the set of morphisms from the geometric realization $|\tn{\v C}(\fr U)| \rightarrow X$ in $\tn{PSh}(\mathbf{DSch}^{\tn{ft}}_{/k})$, where $\tn{\v C}(\fr U)$ is the \v Cech nerve associated to $\fr U=\{p, i\}$ for any derived abstract blow-up square. Then $\mathbf F\in\tn{PSh}(\mathbf{DSch}^{\tn{ft}}_{/k})$ is a $\mathbf t_0$-sheaf if and only if it is $\mathbf S$-local. Indeed, the presheaf $|\tn{\v C}(\fr U)|$ is equivalent to the sieve generated by $\fr U$, so the result again follows from \cite[Corollary C.2]{hoyois2015quadratic}.

\subsubsection{} We note that derived abstract blow-up squares verify the ($\infty$-categorical version of the) conditions of \cite[Theorem 3.2.5]{asok2017affine}. More precisely:
\begin{enumerate}[(a)]
	\item every derived abstract blow-up square is homotopy Cartesian;
	\item derived abstract blow-up squares are stable under base change in $\mathbf{DSch}^{\tn{ft}}_{/k}$;
	\item for every \eqref{eq-derived-blowup}, $i$ is a monomorphism of presheaves;
	\item given \eqref{eq-derived-blowup}, the induced square is still a derived abstract blow-up:
	$$
	\xysmall{
		\cal E \ar[r]\ar[d] & Y \ar[d]^{\Delta_p} \\
		\cal E\underset{\cal Z}{\times}\cal E \ar[r]^-{(i,i)} & Y\underset{X}{\times} Y
	}
	$$
\end{enumerate}

\noindent
Thus we have the following analogus of \cite[Theorem 3.2.5]{asok2017affine}.

\begin{lem}
Let $\mathbf F$ be a presheaf on $\mathbf{DSch}^{\tn{ft}}_{/k}$. Then it is a $\mathbf t_0$-sheaf if and only if $\mathbf F(\emptyset)$ is contractible and for every derived abstract blow-up square \eqref{eq-derived-blowup}, the induced square is homotopy Cartesian:
	\begin{equation}
	\label{eq-derived-blowup-value}
	\xysmall{
		\op{Hom}(\cal E, \mathbf F) & \mathbf F(Y) \ar[l] \\
		\op{Hom}(\cal Z, \mathbf F) \ar[u] & \mathbf F(X) \ar[l]\ar[u]
	}
	\end{equation}
\end{lem}
\begin{proof}
The proof of \emph{loc.cit.}~applies verbatim.
\end{proof}

We remark that Condition (c) would fail if $\cal Z$ was a closed subscheme of $X$ instead of a formal completion.

\subsubsection{} We define $\bfeh$ to be the coarsest topology on $\mathbf{DSch}^{\tn{ft}}_{/k}$ containing the \'etale topology, the topology generated by surjective closed immersions, and $\mathbf t_0$. Thus, a $\mathbf {Spc}$-valued presheaf $\mathbf F$ on $\mathbf{DSch}^{\tn{ft}}_{/k}$ is an $\bfeh$-sheaf if and only if it satisfies:
\begin{enumerate}[(a)]
	\item $\mathbf F$ is an \'etale sheaf;
	\item $\mathbf F$ satisfies descent along surjective closed immersions;
	\item $\mathbf F$ turns every derived abstract blow-up square into a homotopy Cartesian square.
\end{enumerate}

Given a derived abstract blow-up square \eqref{eq-derived-blowup}, the sieve generated by $\{p, i\}$ can be refined by a proper surjective cover (for instance, taking any closed subscheme $Z$ of $X$ with the same underlying set as $\cal Z$, we obtain a proper surjection $Z\sqcup Y\rightarrow X$). Therefore $\bfeh$ is coarser than the derived $\mathbf h$-topology (studied in \cite{halpern2014mapping}). We obtain relations analogous to the classical situation:
$$
\tn{\'etale} \preceq \bfeh \preceq \mathbf h.
$$
However, we caution the reader that the restriction of an $\bfeh$-sheaf to the full subcategory $\mathbf{Sch}^{\tn{ft}}_{/k}$ is not necessarily an $\eh$-sheaf in the classical sense.

\subsubsection{} We record some facts which will be used later.

\begin{lem}
\label{lem-perf-eh-descent}
The presheaf $\tn{Perf}$ is an $\mathbf h$-sheaf on $\mathbf{DSch}^{\tn{ft}}_{/k}$.
\end{lem}
\begin{proof}
This is \cite[Theorem 3.3.1]{halpern2014mapping}.
\end{proof}

\begin{lem}
\label{lem-nil-extend}
Let $\cal F$ (resp.~$\mathbf F$) be an $\eh$-sheaf on $\mathbf{Sch}^{\tn{ft}}_{/k}$ (resp.~\emph{$\bfeh$}-sheaf on $\mathbf{DSch}^{\tn{ft}}_{/k}$).
\begin{enumerate}[(a)]
	\item The tautological extension of $\cal F$ to $\mathbf{DSch}^{\tn{ft}}_{/k}$ is an \emph{$\bfeh$}-sheaf:
$$
(\mathbf{DSch}^{\tn{ft}}_{/k})^{\tn{op}} \rightarrow \mathbf{Spc},\quad X\leadsto \cal F(\pi_0X)
$$
	\item If $\mathbf F$ is nil-invariant, then its restriction to $\mathbf{Sch}^{\tn{ft}}_{/k}$ is an $\eh$-sheaf.
\end{enumerate}
\end{lem}
\noindent
In particular, given a nil-invariant presheaf $\mathbf F$ on $\mathbf{DSch}^{\tn{ft}}_{/k}$, satisfying $\bfeh$ descent is equivalent to its restriction to $\mathbf{Sch}^{\tn{ft}}_{/k}$ satisfying $\eh$ descent.
\begin{proof}
The \'etale descent is clear is both statements. To prove (a), we note that $\cal F$ is nil-invariant so its extension has descent along surjective closed immersions. Let us now be given a derived abstract blow-up square \eqref{eq-derived-blowup} where $\cal Z$ is the formal completion of $Z \subset |X|$. We represent $\cal Z$ as a filtered colimit of $Z_{\alpha}$, where each $Z_{\alpha}$ is a closed subscheme of $X$ with underlying set $Z$. Then $\cal E$ identifies with $ \underset{\alpha}{\tn{colim}}\, E_{\alpha}$ for $E_{\alpha} := Z_{\alpha} \underset{X}{\times} Y$. The square \eqref{eq-derived-blowup-value} is equivalent to:
$$
\xysmall{
	\lim_{\alpha} \cal F(\pi_0E_{\alpha}) & \cal F(\pi_0Y) \ar[l] \\
	\lim_{\alpha} \cal F(\pi_0Z_{\alpha}) \ar[u] & \cal F(\pi_0X) \ar[l]\ar[u]
}
$$
which is a limit of homotopy Cartesian diagrams. To prove (b), let us be given an abstract blow-up square \eqref{eq-blowup}. Let $\cal Z$ (resp.~$\cal E$) be the completion of $Z$ inside $X$ (resp.~$E$ inside $Y$). Then we obtain a derived abstract blow-up square, so the following square is homotopy Cartesian:
$$
\xysmall{
	\tn{Hom}(\cal E, \mathbf F) & \mathbf F(Y) \ar[l] \\
	\tn{Hom}(\cal Z, \mathbf F) \ar[u] & \mathbf F(X) \ar[u]\ar[l]
}
$$
Since $\mathbf F$ is nil-invariant, the left vertical map identifies with $\mathbf F(Z)\rightarrow \mathbf F(E)$.
\end{proof}

\begin{lem}
\label{lem-nil-sheafify}
Suppose $\mathbf F$ is an $n$-truncated presheaf on $\mathbf{DSch}^{\tn{ft}}_{/k}$ for some $n\ge 0$, i.e. ~$\pi_i\mathbf F(X) = 0$ for all $i > n$ and $X\in\mathbf{DSch}^{\tn{ft}}_{/k}$. Then \emph{$\mathbf F_{\bfeh}$} is nil-invariant.
\end{lem}
\begin{proof}
Any $\bfeh$-hypersheaf is nil-invariant since the constant simplicial system $X_{\tn{red}}$ is an $\bfeh$-hypercover of $X\in\mathbf{DSch}^{\tn{ft}}_{/k}$. The $n$-truncation hypothesis implies that the $\bfeh$-sheafification and hypersheafification agree.
\end{proof}

We let $\tn{PSh}^{\tn{nil}, \le n}(\mathbf{DSch}^{\tn{ft}}_{/k})$ denote the $\infty$-category of nil-invariant, $n$-truncated presheaves on $\mathbf{DSch}^{\tn{ft}}_{/k}$. Combining Lemma \ref{lem-nil-extend} and Lemma \ref{lem-nil-sheafify}, we have commutative diagrams:
\begin{equation}
\label{eq-eh-derived-comparison}
\xysmall{
	\tn{PSh}^{\tn{nil}, \le n}(\mathbf{DSch}^{\tn{ft}}_{/k}) \ar[r]^-{\sim} \ar@<-0.5ex>[d]_L & \tn{PSh}^{\tn{nil}, \le n}(\mathbf{Sch}^{\tn{ft}}_{/k}) \ar@<-0.5ex>[d]_L \\
	\tn{Shv}_{\bfeh}^{\tn{nil}, \le n}(\mathbf{DSch}^{\tn{ft}}_{/k}) \ar@<-0.5ex>[u]\ar[r]^-{\sim} & \tn{Shv}_{\eh}^{\tn{nil}, \le n}(\mathbf{Sch}^{\tn{ft}}_{/k}) \ar@<-0.5ex>[u]
}
\end{equation}
In other words, for $n$-truncated nil-invariant presheaves, the $\bfeh$ and $\eh$-topologies give rise to the same sheaf theory with the same functorialities.

\bigskip

\section{Differential forms of moderate growth}
\label{sec-diff-forms}

In this section, the ground field $k$ is assumed algebraically closed with $\tn{char}(k)=0$.

\smallskip

Its purpose is to introduce another ingredient in the construction of tame twistings, namely ``differential forms of moderate growth.'' We introduce the sheaves $\mathring{\Omega}^p$ for $p\ge 0$ on the category of (classical) finite type schemes $\mathbf{Sch}^{\tn{ft}}_{/k}$, study their descent properties, and finally calculate their cohomology groups over a smooth curve in \S\ref{sec-curve-coh}.

\subsection{Point of departure}

\subsubsection{} An effective Cartier divisor $D$ in a smooth scheme $X$ is said to be of \emph{normal crossing} if, \'etale locally on $X$, $D$ is defined by the vanising of $x_1\cdots x_k$ ($k\le n$) where $x_1,\cdots, x_n$ is a system of coordinates on $X$. Although globally, $D$ may not be a union of smooth divisors, the normalization $\nu : \widetilde D\rightarrow D$ always produces a smooth $\widetilde D$. In the situation of a normal crossing divisor with complement $\mathring X$:
$$
\mathring X \xrightarrow{j} X \xleftarrow{i} D,
$$
one may define a locally free $\cal O_X$-module $\Omega_X^p(\log D)$ for each $p\ge 0$. We refer the reader to \cite[\S II.3]{deligne2006equations} for its basic properties.

\subsubsection{} Let $X\in\mathbf{Sm}_{/k}$. A \emph{good compactification} of $X$ is an open immersion $X\hookrightarrow\overline X$, where $\overline X$ is proper, smooth, and $D:=\overline X\backslash X$ is a normal crossing divisor. Hironaka's desingularization shows that a good compactification always exists. The complex $\Omega_{\overline X}^{\bullet}(\log D)$ equipped with the Hodge filtration (i.e., stupid truncation) yields a spectral sequence:
\begin{equation}
\label{eq-hodge-dR}
{}_F\tn E_1^{p, q} = \op H^q(\overline X; \Omega_{\overline X}^p(\log D)) \implies \mathbb H^{p+q}(\overline X; \Omega^{\bullet}_{\overline X}(\log D)),
\end{equation}
which degenerates at $\tn E_1$ (\cite[Corollaire 3.2.13(ii)]{deligne1971theorie}). Since $\mathbb H^p(\overline X; \Omega_{\overline X}^{\bullet}(\log D))$ and the Hodge filtration it carries are canonically independent of the good compactification (\cite[Th\'eor\`eme 3.2.5(ii)]{deligne1971theorie}), so must be its $p$th graded piece
$$
{}_F\op{Gr}^p\mathbb H^p(\overline X; \Omega^{\bullet}_{\overline X}(\log D))\xrightarrow{\sim} \op H^0(\overline X; \Omega_{\overline X}^p(\log D)).
$$

\subsubsection{} We are thus led to the following definition. For $p\ge 0$, define $\mathring{\Omega}^p$ as the subpresheaf of $\Omega^p$ on $\mathbf{Sm}_{/k}$, consisting of those differential forms $\omega\in\Omega^p(X)$ which extend to $\op H^0(\overline X; \Omega^p_{\overline X}(\log D))$ for a good compactification $X\hookrightarrow \overline X$. The above observation implies that $\mathring{\Omega}^p$ is a well-defined presheaf on $\mathbf{Sm}_{/k}$. We extend $\mathring{\Omega}^p$ to $\mathbf{Sch}^{\tn{ft}}_{/k}$ by the procedure of right Kan extension:
$$
\mathring{\Omega}^p(X) := \underset{\substack{Y\rightarrow X \\ Y\in\mathbf{Sm}_{/k}}}{\lim} \mathring{\Omega}^p(Y).
$$

\subsubsection{} Let us note some quick consequences of the definition:

\begin{lem}
\label{lem-basic-properties}
The presheaves $\mathring{\Omega}^p$ ($p\ge 0$) satisfy:
\begin{enumerate}[(a)]
	\item $\mathring{\Omega}^0$ is canonically the constant sheaf $\underline k$;
	\smallskip
	\item $\mathring{\Omega}^p$ takes values in finite-dimensional $k$-vector spaces;
	\smallskip
	\item For $X\in\mathbf{Sm}_{/k}$, the subspace $\mathring{\Omega}^p(X)\subset \Omega^p(X)$ belongs to closed $p$-forms;
	\smallskip
	\item $\mathring{\Omega}^p$ is a sheaf in the Zariski topology on $\mathbf{Sm}_{/k}$.
\end{enumerate}
\end{lem}
\begin{proof}
(a) is immediate. (b) follows from the smooth case by taking a smooth hypercover. (c) is a consequence of the degeneration of \eqref{eq-hodge-dR} at $\tn E_1$ (\cite[Corollaire 3.2.14]{deligne1971theorie}). For (d), it is clear that $\mathring{\Omega}^p$ is a separated presheaf. To check gluing, we cover $X \in \mathbf{Sm}_{/k}$ by opens $U$ and $V$, the Mayer-Vietoris sequence on de Rham cohomology:
$$
\mathbb H^p(X) \rightarrow \mathbb H^p(U) \oplus \mathbb H^p(V) \rightarrow \mathbb H^p(U\cap V)
$$
is exact and strictly compatible with the Hodge filtration (\cite[Th\'eor\`eme 1.2.10(iii)]{deligne1971theorie}), so it remains exact after applying ${}_F\tn{Gr}^p$ (\cite[Proposition 1.1.11(ii)]{deligne1971theorie}).
\end{proof}

\subsection{$\tn h$-descent}

\subsubsection{} In this section, we shall prove:

\begin{prop}
\label{prop-h-descent}
For all $p\ge 0$, the presheaf $\mathring{\Omega}^p$ on $\mathbf{Sch}^{\tn{ft}}_{/k}$ satisfies $\tn h$-descent.
\end{prop}

\noindent
Instead of giving a direct argument, we compare $\mathring{\Omega}^p$ to the $\tn h$-sheafification $\Omega^p_h$ of the usual differential $p$-forms, studied by Huber--J\"order \cite{huber2013differential}. Their theorem is that $\Omega_{\tn h}^p$ identifies with the right Kan extension of $\Omega^p$ from $\mathbf{Sm}_{/k}$:
$$
\Omega^p_{\tn h}(X) \xrightarrow{\sim} \underset{\substack{Y\rightarrow X \\ Y\in\mathbf{Sm}_{/k}}}{\lim} \Omega^p(Y).
$$
This implies that $\mathring{\Omega}^p$ can be regarded as a subpresheaf of $\Omega_{\tn h}^p$, characterized by the property that a section $\omega\in\Omega_{\tn h}^p(X)$ belongs to $\mathring{\Omega}^p(X)$ if and only if its pullback to any smooth scheme $Y\rightarrow X$ belongs to $\mathring{\Omega}^p(Y)$.

\subsubsection{} Therefore, in order to prove Proposition \ref{prop-h-descent}, we only need to show that for $\pi : \widetilde X\rightarrow X$ an $\tn h$-covering in $\mathbf{Sch}^{\tn{ft}}_{/k}$, if $\omega\in\Omega^p_{\tn h}(X)$ has the property that $\pi^*\omega$ belongs to $\mathring{\Omega}^p(\widetilde X)$, then $\omega\in\mathring{\Omega}^p(X)$. By mapping a smooth scheme $Y$ to $X$ and considering a further smooth $\tn h$-cover of $Y\underset{X}{\times}\widetilde X$, we may assume that $\widetilde X$ and $X$ are both smooth. Fitting $\widetilde X\rightarrow X$ into a map between good compactifications, the Proposition follows from the Lemma below.

\begin{lem}
\label{lem-log-local}
Suppose there is a commutative diagram in $\mathbf{Sm}_{/k}$:
$$
\xysmall{
	\mathring Y \ar@{^{(}->}[r]\ar[d] & Y \ar[d]^{\pi} \\
	\mathring X \ar@{^{(}->}[r] & X
}
$$
where $\mathring X\hookrightarrow X$ (resp.~$\mathring Y\hookrightarrow Y$) is an open immersion whose boundary is a normal crossing divisor $D$ (resp.~$E$). Assume furthermore that $\pi$ is a proper surjection. Then given any $\omega\in\Omega^p(\mathring X)$, it extends to $\Omega^p_X(\log D)$ if and only if $\pi^*\omega$ extends to $\Omega_Y^p(\log E)$.
\end{lem}
\begin{proof}
The ``only if'' direction is clear as $\pi^{-1}D$ is set-theoretically contained in $E$. Let us argue the converse. The question is \'etale local on $X$. Since $\Omega^p_X(\log D)$ is locally free, it suffices to show that $\omega$ extends to $\Omega_X^p(\log D)$ away from codimension $\ge 2$. Thus we will choose coordinates $x_1,\cdots, x_n\in\cal O_X$ such that $D$ is defined by $x_1 = 0$ and $\Omega_X^1$ is free on $dx_1,\cdots, dx_n$.

\smallskip

We will also replace $Y$ by its formal neighborhood around some $y\in Y$ contained in the smooth locus of an irreducible component $E_1$ of $E$ which dominates $D$. Since the normalization $\widetilde E_1\rightarrow D$ is a proper surjection and $\widetilde E_1$ is connected and smooth, we see that $\Omega^p(D) \rightarrow \Omega^p(E_1)$ is injective. In other words, we shall assume:
\begin{enumerate}[(a)]
	\item $Y = \Spec(k\arc{y_1, \cdots, y_m})$, $E_1$ is defined by $y_1 = 0$, and $\mathring Y = Y\backslash E_1$ is the preimage of $\mathring X$;
	\smallskip
	\item The map $\Omega^p(D) \rightarrow \Omega^p(E_1)$ is injective.
\end{enumerate}
Thus $\pi^*x_1 = u y_1^e$ for some $e\ge 1$ and $u\in\cal O_Y^{\times}$. Hensel's lemma finds an $e$th root of $u$, so after an automorphism on $Y$ fixing $E_1$, we may further assume:
\begin{enumerate}[(c)]
	\item $\pi^*x_1 = y_1^e$.
\end{enumerate}

\smallskip

Let us now consider a meromorphic form $\omega \in \Omega^p(X)[x_1^{-1}]$ such that $\pi^*\omega \in \Omega^p(Y)[y_1^{-1}]$ is logarithmic along $E_1$. Write
$$
\omega = \omega_1 + \omega_2 \wedge \frac{dx_1}{x_1},
$$
where $\omega_1, \omega_2 \in \Omega^p(X)[x_1^{-1}]$ do not feature $dx_1$. In what follows we assume $\omega_1, \omega_2$ are both nonzero (the case where either is zero being similar but simpler). Write $\omega_1 = x^{d_1}_1\tilde{\omega}_1$ and $\omega_2 = x_1^{d_2}\tilde{\omega}_2$ where $\tilde{\omega}_1$ and $\tilde{\omega}_2$ are holomorphic and not divisible by $x_1$. Then:
\begin{align}
\pi^*\omega &= (y_1^e)^{d_1}\pi^*\tilde{\omega}_1 + (y_1^e)^{d_2}\pi^*\tilde{\omega}_2 \wedge e\frac{dy_1}{y_1}\notag \\
&= (y_1^e)^{d_1}(\eta^{(1)}_1 + \eta^{(2)}_1\wedge dy_1) + (y_1^e)^{d_2}\eta^{(1)}_2 \wedge e\frac{dy_1}{y_1} \label{eq-omega-decomp}
\end{align}
where $\pi^*\tilde{\omega}_1 = \eta^{(1)}_1 + \eta^{(2)}_1\wedge dy_1$ is its decomposition into parts where $\eta_1^{(1)}$, $\eta_1^{(2)}$ do \emph{not} feature $dy_1$ (and analogously for $\pi^*\omega_2$). Then assumption (b) implies that $\pi^*\tilde{\omega}_1, \pi^*\tilde{\omega}_2$ are \emph{nonzero} after pulling back to $E_1$. Thus $\eta_1^{(1)}$ and $\eta_2^{(1)}$ are \emph{not} divisible by $y_1$. Now, analyzing the part of the expression \eqref{eq-omega-decomp} \emph{not} featuring $dy_1$, we see that $d_1\ge 0$. Hence the first term is holomorphic, so the second term is necessarily logarithmic along $y_1$. Since $\eta_2^{(1)}$ is not divisible by $y_1$, we see that $d_2\ge 0$ as well.
\end{proof}

\qed(Proposition \ref{prop-h-descent})

\subsubsection{} A particular consequence of the $\tn h$-descent of $\mathring{\Omega}^p$ is a canonical transfer structure on the restriction of $\mathring{\Omega}^p$ to $\mathbf{Sm}_{/k}$. We recall the category of correspondences $\mathbf{Sm}^{\tn{Cor}}_{/k}$ mentioned in \S\ref{sec-corr}. According to J.~Scholbach \cite[Lemma 2.1]{scholbach2012geometric}, the representable presheaf $\mathbb Z_{\tn{tr}}(X)$ on $\mathbf{Sm}_{/k}^{\tn{Cor}}$ for any $X\in\mathbf{Sm}_{/k}$ has the property that its $\tn h$-sheafification identifies with that of $\mathbb Z(X)$ on $\mathbf{Sm}_{/k}$:
$$
\mathbb Z_{\tn h}(X) \xrightarrow{\sim} (\mathbb Z_{\tn{tr}}(X)|_{\mathbf{Sm}_{/k}})_{\tn h}.
$$
Consequently, for any $\tn h$-sheaf of abelian groups $\cal F$ on $\mathbf{Sm}_{/k}$ there is an isomorphism:
$$
\cal F(X) \xrightarrow{\sim} \op{Hom}_{\tn{PSh}(\mathbf{Sm}_{/k})}(\mathbb Z_{\tn{tr}}(X), \cal F),
$$
so $\cal F$ acquires a canonical transfer structure.

\begin{lem}
\label{lem-transfer}
The restriction of $\mathring{\Omega}^p$ ($p\ge 0$) to $\mathbf{Sm}_{/k}$ is an $\mathbb A^1$-invariant sheaf with a canonical transfer structure.
\end{lem}
\begin{proof}
The $\mathbb A^1$-invariance is a direct consequence of the identification of $\mathring{\Omega}^p(X)$ with the $p$th graded piece of $\mathbb H^p(\overline X, \Omega_{\overline X}^{\bullet}(\log D))$ with respect to the Hodge filtration. The canonical transfer structure has just been noted above.
\end{proof}

By construction, the transfer structure on $\mathring{\Omega}^p$ is compatible with that of $\Omega^p$. For an explicit formula of the latter, we refer the reader to the trace construction of Lecomte--Wach \cite{lecomte2009complexe}. In particular, the morphism $d\log : \mathbb G_m \rightarrow \mathring{\Omega}^1$ commutes with transfer.

\subsection{Cohomological properties}

\subsubsection{} Suppose $\cal F$ is a presheaf on $\mathbf{Sm}_{/k}$ valued in abelian groups. Following Voevodsky \cite[\S3.1]{voevodsky2000cohomological}, we define $\cal F_{-1}$ to be the presheaf:
$$
\cal F_{-1} : X\leadsto \tn{Coker}(\cal F(X\times\mathbb A^1) \rightarrow \cal F(X\times(\mathbb A^1\backslash\mathbf 0))).
$$
The presheaf $\cal F_{-n}$ is then defined iteratively.

\begin{lem}
\label{lem-omega-minus}
There holds:
\begin{enumerate}[(a)]
	\item The sheaf $(\mathring{\Omega}^0)_{-1}$ is identically zero;
	\item For any $p\ge 1$, there is a canonical isomorphism $(\mathring{\Omega}^p)_{-1} \xrightarrow{\sim} \mathring{\Omega}^{p-1}$.
\end{enumerate}
\end{lem}
\begin{proof}
Part (a) is tautological. Part (b) follows either from the Hodge-theoretic interpretation of $\mathring{\Omega}^p$ or a direct calculation making use of the product formula for logarithmic forms \cite[\S II, Proposition 3.2(iii)]{deligne2006equations}.
\end{proof}

\subsubsection{} For notational convenience, we extend $\mathring{\Omega}^p$ to smooth local schemes (i.e., localizations of smooth schemes at a point) by the formula:
$$
\mathring{\Omega}^p(\eta) := \underset{U_{\alpha}}{\tn{colim}}\,\mathring{\Omega}^p(U_{\alpha}),
$$
where $U_{\alpha}$ is a cofiltered limit presentation of $\eta$ with each $U_{\alpha}$ smooth, affine and each $U_{\alpha}\rightarrow U_{\beta}$ an open immersion. The following Theorem summarizes the cohomological properties of $\mathring{\Omega}^p$:

\begin{thm}
\label{thm-coh-properties}
Let $p\ge 0$ and $\tau$ be one of the following Grothendieck topologies on $\mathbf{Sm}_{/k}$: Zariski, Nisnevich, \'etale, \tn{cdh}, \eh, \tn{qfh}, \tn{h}. There holds:
\begin{enumerate}[(a)]
	\item For all $n\ge 0$, the presheaf $X \leadsto \op H^n_{\tau}(X; \mathring{\Omega}^p)$ on $\mathbf{Sm}_{/k}$ is an $\mathbb A^1$-invariant presheaf with transfer, and is canonically independent of the choice of $\tau$;
	\smallskip
	\item For $X\in\mathbf{Sm}_{/k}$, the Zariski sheaf $\mathring{\Omega}^p_X$ is quasi-isomorphic to the following complex concentrated in degrees $[0, p]$:
	$$
	\bigoplus_{x\in X^{(0)}} (i_x)_* \mathring{\Omega}^p(x) \rightarrow \bigoplus_{x\in X^{(1)}} (i_x)_*\mathring{\Omega}^{p-1}(x) \rightarrow \cdots \rightarrow \bigoplus_{x\in X^{(p)}} (i_x)_*k.
	$$
\end{enumerate}
\end{thm}
\noindent
Here, $X^{(n)}$ denotes the set of codimension-$n$ points of $X$.
\begin{proof}
Statement (a) is valid for any $\mathbb A^1$-invariant $\tn h$-sheaf of $\mathbb Q$-vector spaces, by Scholbach \cite[Theorem 2.11]{scholbach2012geometric}; the only choice of $\tau$ not covered in \emph{loc.cit.}~is the $\eh$-topology, which follows from Lemma \ref{lem-eh-to-etale}. For statement (b), Mazza--Voevodsky--Weibel \cite[Theorem 24.11]{mazza2011lecture} shows that an $\mathbb A^1$-invariant pretheory $\cal F$ satisfying Zariski descent admits a Gersten resolution with terms given by $\bigoplus_{x\in X^{(n)}} (i_x)_*F_{-n}(x)$. We are done by the calculation of $(\mathring{\Omega}^p)_{-n}$ in Lemma \ref{lem-omega-minus}.
\end{proof}

\begin{rem}
A.~Beilinson has kindly pointed out that the Gersten resolution in (b) also follows directly from applying ${}_F\tn{Gr}^p$ to the Gersten resolution of algebraic de Rham cohomology obtained from the Bloch--Ogus theorem.
\end{rem}

\subsubsection{Example}
\label{sec-curve-coh} We calculate the cohomology of $\mathring{\Omega}^1$ on a smooth curve $X$. Since the cohomology groups will be independent of the chosen Grothendieck topology (Theorem \ref{thm-coh-properties}(a)), we may as well calculate them in the Zariski topology using the Gersten resolution (Theorem \ref{thm-coh-properties}(b)). The answer is as follows:
\begin{enumerate}[(a)]
	\item if $X$ is affine, then $\op H^1(X; \mathring{\Omega}^1) = 0$;
	\smallskip
	\item if $X$ is proper, then the canonical map $\op R\Gamma(X; \mathring{\Omega}^1) \rightarrow \op R\Gamma(X; \Omega^1)$ is an isomorphism.
\end{enumerate}
\noindent
Indeed, the affine case amounts to the problem of contructing $\omega$ with prescribed poles and follows from $\op H^1(\overline X; \Omega^1(E)) = 0$ for the boundary divisor $E:=\overline X\backslash X$ in a smooth completion $\overline X$.

\smallskip

For the proper case, the nontrivial part is cohomology in degree $1$. We reduce to $X$ connected (with generic point $\eta$) and remove one closed point $\mathring X:= X\backslash x$. The sum-of-residue formula and the vanishing of $\op H^1(\mathring X, \mathring{\Omega}^1)$ shows that the cokernel of $d$ is indeed identified with $k$:
$$
\xysmall{
	& & \mathring{\Omega}(\eta) \ar[d]_d \ar[dr] & \\
	0 \ar[r] & k \ar[r] & \bigoplus_{x\in X^{(1)}}k \ar[r] & \bigoplus_{x\in\mathring X^{(1)}} k \ar[r] & 0.
}
$$

\subsubsection{Tangential remarks} We conclude this section with some remarks concerning the interaction between $\mathring{\Omega}^p$ and algebraic cycles. These facts will not play a role in this paper.

\smallskip

Let $\mathbf K_p^{\tn M}$ denote the Zariski sheaf of the $p$th Milnor $\tn K$-theory group on $\mathbf{Sm}_{/k}$. For a field $F$, $\mathbf K_p^{\tn M}(F)$ is the $p$th graded piece of the tensor algebra $T^{\otimes}(F^{\times})$ modulo $u\otimes v$ for $u + v = 1$. More generally, $\mathbf K_p^{\tn M}$ is given by a Gersten resolution. When $X$ is furthermore projective, $\op H^p(X, \mathbf K_p^{\tn M})$ identifies with the Chow group $\tn{CH}^p(X)$ of codimension-$p$ cycles \cite[Th\'eor\`eme 5]{soule1985operations}. In particular, the construction:
$$
d\log : \mathbf K_p^{\tn M}(\eta) \rightarrow \mathring{\Omega}^p(\eta),\quad f_1\otimes\cdots\otimes f_n\leadsto \frac{df_1}{f_1}\wedge\cdots\wedge\frac{df_n}{f_n}
$$
for points $\eta$ on $X\in\mathbf{Sm}_{/k}$ defines a morphism of Zariski sheaves on $\mathbf{Sm}_{/k}$:
\begin{equation}
\label{eq-milnor-to-tame}
d\log : \mathbf K_p^{\tn M}\underset{\mathbb Z}{\otimes} k \rightarrow \mathring{\Omega}^p.
\end{equation}

We obtain the following factorization of the algebraic de Rham cycle class map:
$$
\xysmall{
	\tn{CH}^p(X) \underset{\mathbb Z}{\otimes} k \ar[r]^-{\sim}\ar[drr]_{\tn{cl}} & \op H^p(X; \mathbf K^{\tn M}_p\underset{\mathbb Z}{\otimes} k) \ar[r]^-{d\log} & \op H^p(X; \mathring{\Omega}^p) \ar[d]^{\tn{can}} \\
	 & & \op H^p(X; \Omega^p)
}
$$
Indeed, its factorization through $d\log : \op H^p(X; \mathbf K^{\tn M}_p\underset{\mathbb Z}{\otimes} k) \rightarrow \op H^p(X; \Omega^p)$ is already observed in \cite{esnault1994remarks} and the further factorization through $\op H^p(X, \mathring{\Omega}^p)$ is tautological. The Gersten resolution of $\mathring{\Omega}^p$ (Theorem \ref{thm-coh-properties}(b)) implies that the composition $\tn{CH}^p(X) \underset{\mathbb Z}{\otimes} k \rightarrow \op H^p(X; \mathring{\Omega}^p)$ is surjective. Thus the image of $\op H^p(X; \mathring{\Omega}^p)$ in $\op H^p(X; \Omega^p)$ is precisely the span of cycle classes.

\bigskip

\section{Tame gerbes and twistings}
\label{sec-ge-and-tw}

We continue to assume $k=\bar k$ with $\tn{char}(k)=0$.

\smallskip

The purpose of this section is to define tame gerbes and tame twistings. They will be constructed as derived $\bfeh$-stacks valued in strict (i.e., strictly commutative) Picard groupoids. We also compare tame gerbes with analytic $\mathbb C^{\times}$-gerbes when the ground field is $\mathbb C$ (\S\ref{sec-analytic-comparison}). This section contains mostly definitions and very few statements that require proofs.

\subsection{Picard $n$-groupoids}
\label{sec-picard-groupoids}

\subsubsection{} In this paper, we refer to commutative group objects of $\mathbf{Spc}$ as \emph{Picard groupoids}. More precisely, Picard groupoids $\mathbf A$ form the full subcategory of $\mathbb E_{\infty}$-spaces characterized by the property of being \emph{grouplike}, i.e., $\pi_0\mathbf A$ is a group under the commutative multiplication. A Picard groupoid $\mathbf A \in \tn{ComGrp}(\mathbf{Spc})$ with $\pi_i\mathbf A=0$ for $i > 1$ is thus a Picard groupoid in the classical sense (c.f.~\cite[Expos\'e XVIII]{bourbaki2006theorie}).

\smallskip

The $\infty$-category $\tn{ComGrp}(\mathbf{Spc})$ is also equivalent to that of connective spectra:
$$
\tn{ComGrp}(\mathbf{Spc}) \xrightarrow{\sim} \mathbf{Sptr}_{\ge 0}.
$$
We note that the forgetful functor from $\tn{ComGrp}(\mathbf{Spc})$ to $\mathbf{Spc}$, which passes to $\Omega^{\infty}$ on the level of spectra, preserves limits and filtered colimits.

\subsubsection{} We will also need to consider the more restricted notion of \emph{strict Picard groupoids}. These are the $\op H\mathbb Z$-module objects in $\tn{ComGrp}(\mathbf{Spc})$. The Dold--Kan correspondence identifies the following $\infty$-categories:
\begin{enumerate}[(a)]
	\item Nonpositively graded cochain complexes of abelian groups $\mathbb Z\Mod^{\le 0}$;
	\smallskip
	\item $\op H\mathbb Z$-module objects in $\tn{ComGrp}(\mathbf{Spc})$.
\end{enumerate}
\noindent
Under this correspondence, the $\op H^{-i}$ of a cochain complex identifies with $\pi_i$ of the $\tn H\mathbb Z$-module, for all $i\ge 0$. We will denote this $\infty$-category by $\tn{ComGrp}^{\tn{st}}(\mathbf{Spc})$, often passing without mention the Dold--Kan correspondence.

\begin{rem}
For every $\mathbf A\in\tn{ComGrp}^{\tn{st}}(\mathbf{Spc})$, the commutativity constraint $\mathbf A\otimes\mathbf A\xrightarrow{\sim} \mathbf A\otimes\mathbf A$ is homotopy equivalent to $\tn{id}_{\mathbf A\otimes\mathbf A}$. For $\mathbf A\in\tn{ComGrp}(\mathbf{Spc})$ with $\pi_i\mathbf A = 0$ for $i > 1$, being strict is a condition but this will no longer be the case in general.
\end{rem}

We shall call a (resp.~strict) Picard groupoid $\mathbf A$ with $\pi_i\mathbf A = 0$ for $i > n$ a (resp.~\emph{strict}) \emph{Picard $n$-groupoid}. One of the main objects we shall be concerned with---gerbes---form a strict Picard $2$-groupoid.

\subsubsection{}
\label{sec-complex-to-space}
Let us note the sheaf-theoretic analogue of the above discussion. For $X\in\mathbf{Sch}_{/k}^{\tn{ft}}$, there is a functor from the $\infty$-category of complexes of \'etale sheaves of abelian groups on $X$ to the $\infty$-category of $\tn{ComGrp}^{\tn{st}}(\mathbf{Spc})$-valued \'etale sheaves:
\begin{equation}
\label{eq-complex-to-space}
\cal F^{\bullet}\leadsto\mathbf F,\quad \mathbf F(U) := \tn{DK}(\tau^{\le 0} \op R\Gamma(U, \cal F^{\bullet})).
\end{equation}
Here $\tau^{\le 0}$ denotes cohomological trunction and $\tn{DK}$ is the Dold--Kan correspondence. The fact that $\mathbf F$ is again a sheaf follows from the preservation of limits under $\tau^{\le 0}$ and $\tn{DK}$. We say that the \'etale sheaf of strict Picard groupoids $\mathbf F$ is \emph{represented by} the complex $\cal F^{\bullet}$.

\begin{lem}
Under the functor $\cal F^{\bullet}\leadsto \mathbf F$ \eqref{eq-complex-to-space}, there holds:
\begin{enumerate}[(a)]
	\item For any $x\in X$, we have an isomorphism of stalks $\mathbf F_x \xrightarrow{\sim} \tn{DK}(\tau^{\le 0} \cal F^{\bullet}_x)$;
	\smallskip
	\item Suppose $f : X\rightarrow Y$ is a morphism in $\mathbf{Sch}^{\tn{ft}}_{/k}$, then $f_*\mathbf F$ identifies with the $\tn{ComGrp}^{\tn{st}}(\mathbf{Spc})$-valued sheaf associated to $\op Rf_*\cal F^{\bullet}$.
\end{enumerate}
\end{lem}
\begin{proof}
Part (a) follows from the identification of $\cal F^{\bullet}_x$ with $\underset{U}{\tn{colim}} \op R\Gamma(U, \cal F^{\bullet})$, where $U$ ranges over \'etale neighborhoods of $x$, and the commutation of $\tau^{\le 0}$ with filtered colimits. Part (b) follows from the fact that for every \'etale $V\rightarrow Y$, the complex $\op R\Gamma(V, \op Rf_*\cal F^{\bullet})$ identifies with $\op R\Gamma(V \underset{Y}{\times}X, \cal F^{\bullet})$.
\end{proof}

\subsection{Local systems}

\subsubsection{} Let $X\in\mathbf{Sch}^{\tn{ft}}_{/k}$. The de Rham prestack $X_{\dR}$ is the prestack whose value on $S\in\mathbf{Sch}^{\tn{ft}}_{/k}$ is given by $\op{Maps}(S_{\tn{red}}, X)$. By a \emph{rank--$1$ local system} on $X$, we will mean a line bundle on $X_{\dR}$. Denote by $\mathbf{Loc}_1$ the prestack which associates to $X\in\mathbf{Sch}_{/k}^{\tn{ft}}$ the strict Picard ($1$-)groupoid of rank--$1$ local systems on $X$.

\begin{lem}
The prestack $\mathbf{Loc}_1$ satisfies $\tn h$-descent.
\end{lem}
\begin{proof}
The $\infty$-prestack $\tn{Crys}^l$ which associates $\QCoh(X_{\dR})$ to $X\in\mathbf{DSch}^{\tn{ft}}_{/k}$ satisfies (derived) $\mathbf h$-descent \cite[Proposition 3.2.2, Proposition 2.4.4]{gaitsgory2011crystals}. Since $\tn{Crys}^l$ is nil-invariant, its restriction to $\mathbf{Sch}^{\tn{ft}}_{/k}$ satisfies (usual) $\tn h$-descent. We observe that $\mathbf{Loc}_1(X)$ is the full subcategory of $\tn{Crys}^l(X)$ consisting of invertible objects lying in the heart of the $t$-structure as an object of $\QCoh(X)$. Thus $\mathbf{Loc}_1$ inherits $\tn h$-descent from $\tn{Crys}^l$.
\end{proof}

Every object in $\mathbf{Loc}_1(X)$ can be viewed as a line bundle $\cal L$ on $X$ equipped with an isomorphism $\op{pr}_1^*\cal L\xrightarrow{\sim}\op{pr}_2^*\cal L$ on the completion of the diagonal in $X\times X$, satisfying a cocycle condition \cite[Proposition 3.4.3]{gaitsgory2011crystals}. When $X$ is smooth, this is equivalent to a connection $\nabla : \cal L\rightarrow \cal L\otimes\Omega^1_X$, but not in general.

\subsubsection{}
It is clear that over $\mathbf{Sm}_{/k}$, the strict Picard groupoid $\mathbf{Loc}_1$ is represented by the complex of \'etale sheaves concentrated in degrees $[-1, 0]$ (in the sense of \S\ref{sec-complex-to-space}):
$$
d\log : \cal O^{\times} \rightarrow \Omega^{1, \op{cl}}.
$$
We recall the subsheaf $\mathring{\Omega}_X^1\hookrightarrow \Omega_X^{1, \tn{cl}}$ of differential forms of moderate growth from \S\ref{sec-diff-forms}.

\begin{lem}
\label{lem-classical-reg}
Let $X\in\mathbf{Sm}_{/k}$, the following conditions are equivalent for any $\sigma\in\mathbf{Loc}_1(X)$.
\begin{enumerate}[(a)]
	\item $\sigma$ belongs to the subcomplex $d\log : \cal O_X^{\times} \rightarrow \mathring{\Omega}_X^1$;
	\item $\sigma$ is regular as a $\cal D_X$-module.
\end{enumerate}
\end{lem}
\noindent
Being \emph{regular} as a $\cal D_X$-module means for any smooth curve $f : C\rightarrow X$, the pullback $f^*\sigma$ acquires a connection with at most logarithmic poles at points of $\overline C\backslash C$.
\begin{proof}
The implication (a) $\implies$ (b) is clear. Conversely, suppose $\sigma$ is regular. To check that it belongs to the subcomplex $\cal O_X^{\times} \xrightarrow{d\log} \mathring{\Omega}_X^1$, it suffices to do so locally on $X$, so we may assume that the underlying line bundle of $\sigma$ is trivial. Thus the connection $1$-form is given by $d + \omega$ for some $\omega\in \Omega^{1,\tn{cl}}(X)$. We need to argue $\omega\in\mathring{\Omega}(X)$. Consider a good compactification $X\hookrightarrow\overline X$. The line bundle extends trivially to $\overline X$. The Lemma thus becomes the implicaion (ii) $\implies$ (iv) in \cite[\S II, Th\'eor\`eme 4.1]{deligne2006equations}.
\end{proof}

\subsubsection{} Let $X\in\mathbf{Sch}_{/k}^{\tn{ft}}$. Then a local system $\sigma\in\mathbf{Loc}_1(X)$ is said to be \emph{tame} if for all morphisms $f : Y\rightarrow X$ with $Y$ smooth, the pullback $f^*\sigma$ satisfies the conditions of Lemma \ref{lem-classical-reg}. We let $\mathring{\mathbf{Loc}}_1$ denote the prestack of tame rank--$1$ local systems on $\mathbf{Sch}_{/k}^{\tn{ft}}$.

\begin{lem}
\label{lem-loc-h-descent}
The prestack $\mathring{\mathbf{Loc}}_1$ satisfies $\tn h$-descent.
\end{lem}
\begin{proof}
Since $\mathring{\mathbf{Loc}}_1$ is a full subfunctor of $\mathbf{Loc}_1$, we only need to prove the following: for an $\tn h$-cover $\pi : \widetilde X\rightarrow X$, if $\sigma\in\mathbf{Loc}_1(X)$ has the property that $\pi^*\sigma \in \mathbf{Loc}_1(\widetilde X)$ is tame, then so is $\sigma$. By definition, we may assume $\widetilde X\rightarrow X$ is a dominant morphism of smooth curves, and the result is straightforward (in fact, a special case of Lemma \ref{lem-log-local}).
\end{proof}

\subsection{Gerbes}
\label{sec-gerbes}

\subsubsection{} We define $\mathring{\mathbf{Ge}}$ as the $\eh$-sheafification of the classifying prestack of $\mathring{\mathbf{Loc}}_1$:
$$
\mathring{\mathbf{Ge}} := \op B_{\eh}\mathring{\mathbf{Loc}}_1.
$$
Informally, a tame gerbe $\cal G$ on a scheme $X$ can be described by \v Cech data as follows. For some $\eh$-cover $\widetilde X\rightarrow X$, we are a given transition tame local system $\sigma$ on the double overlap $\widetilde X\underset{X}{\times}\widetilde X$. On triple overlaps, we are supplied with isomorphisms relating distinct pullbacks of $\sigma$. These isomorphisms must satisfy a cocycle condition on quadruple overlaps.

\smallskip

For $X\in\mathbf{Sch}^{\tn{ft}}_{/k}$, we call $\mathring{\mathbf{Ge}}(X):=\op{Maps}(X, \mathring{\mathbf{Ge}})$ the category of \emph{tame gerbes} on $X$. It has the structure of a strict Picard $2$-groupoid. Lemma \ref{lem-loc-h-descent} guarantees that the loop prestack $\tn{pt} \underset{\mathring{\mathbf{Ge}}}{\times}\op{pt}$ identifies with $\mathring{\mathbf{Loc}}_1$.

\subsubsection{} The following result shows that tame gerbes on a smooth scheme can be defined using the weaker \'etale topology.

\begin{lem}
\label{lem-gerbe-etale-h}
Suppose $X\in\mathbf{Sm}_{/k}$. Then the following canonical map is an isomorphism:
$$
\op{Maps}(X, \op B_{\et}\mathring{\mathbf{Loc}}_1) \xrightarrow{\sim} \mathring{\mathbf{Ge}}(X).
$$
\end{lem}
\noindent
In particular, $\mathring{\mathbf{Ge}}$ is represented by the complex $d\log : \cal O_X^{\times} \rightarrow \mathring{\Omega}^1_X$ in degrees $[-2,-1]$.
\begin{proof}
Let $\mathbf F_{\eh}$ denote the fiber of $\eh$-sheaves $\mathring{\mathbf{Loc}}_1 \rightarrow \op B_{\eh}\mathbb G_m$ on $\mathbf{Sm}_{/k}$. Evaluating at $X \in \mathbf{Sm}_{/k}$ produces a fiber sequence:
$$
\mathbf F_{\eh}(X) \rightarrow \mathring{\mathbf{Loc}}_1(X) \rightarrow \op{Maps}(X, \op B_{\eh}\mathbb G_m).
$$
The comparison Lemma \ref{lem-eh-to-etale-gm} shows that $\op{Maps}(X, \op B_{\eh}\mathbb G_m)$ identifies with $\op{Maps}(X, \op B_{\et}\mathbb G_m)$. Thus Lemma \ref{lem-classical-reg} implies that $\mathbf F_{\eh}$ identifies with $\mathring{\Omega}^1$. On the other hand, $\mathring{\mathbf{Loc}}_1 \rightarrow \op B_{\eh}\mathbb G_m$ is a surjection of $\eh$-sheaves, so $\mathring{\mathbf{Loc}}_1$ is an $\eh$ $\mathring{\Omega}^1$-torsor over $\tn B_{\eh}\mathbb G_m$. This gives us another fiber sequence:
$$
\mathring{\mathbf{Loc}}_1 \rightarrow \op B_{\eh}\mathbb G_m \rightarrow \op B_{\eh}\mathring{\Omega}^1.
$$
Delooping and taking sections over $X\in\mathbf{Sm}_{/k}$, we obtain a fiber sequence:
$$
\op{Maps}(X, \op B_{\eh}\mathring{\mathbf{Loc}}_1) \rightarrow \op{Maps}(X, \op B^2_{\eh}\mathbb G_m) \rightarrow \op{Maps}(X, \op B^2_{\eh} \mathring{\Omega}^1).
$$
Thus the desired result follows from the comparison Lemma \ref{lem-eh-to-etale-gm} for $\mathbb G_m$ and Theorem \ref{thm-coh-properties}(a) for $\mathring{\Omega}^1$.
\end{proof}

\subsubsection{}
\label{sec-div-class-map}
Note that there is a morphism of sheaves of strict Picard groupoids on $\mathbf{Sch}^{\tn{ft}}_{/k}$:
\begin{equation}
\label{eq-funct-loc}
\mathbb G_m \underset{\mathbb Z}{\otimes} k/\mathbb Z \rightarrow \mathring{\mathbf{Loc}},\quad (f, a) \leadsto f^a.
\end{equation}
Indeed, given $f\in\cal O_X^{\times}$ and $a\in k/\mathbb Z$, we will construct a tame local system $f^{a}$ on each smooth $Y$ mapping to $X$ in a compatible way. This process will construct an object of $\mathring{\mathbf{Loc}}(X)$ by Lemma \ref{lem-loc-h-descent}. We choose a lift $\bar{a}\in k$ of $a$. The local system $f^{a}$ on $Y$ is set to be
$$
f^{\bar{a}} := (\cal O_Y, d + \bar{a}d\log f).
$$
Indeed, another choice of the lift $\bar{a}'$ must differ from $\bar a$ by an integer $n$, and the local systems $f^{\bar{a}}$ and $f^{\bar{a}'}$ are canonically isomorphic via multiplication by $f^n\in\cal O_Y^{\times}$. This shows that $f^{a}\in\mathring{\mathbf{Loc}}(Y)$ is well-defined. It is obviously compatible with change of $Y$.

\subsubsection{} From \eqref{eq-funct-loc}, we obtain a morphism of sheaves of strict Picard $2$-groupoids on $\mathbf{Sch}^{\tn{ft}}_{/k}$:
\begin{equation}
\label{eq-div-class-map-tame}
\mathbf{Pic} \underset{\mathbb Z}{\otimes} k/\mathbb Z \rightarrow \mathring{\mathbf{Ge}},\quad (\cal L, a) \leadsto \cal L^a.
\end{equation}
We call \eqref{eq-div-class-map-tame} the \emph{divisor class map} for tame gerbes.

\subsubsection{}
\label{sec-analytic-comparison}
When the ground field $k=\mathbb C$, there is a Riemann--Hilbert correspondence relating tame gerbes to analytic $\mathbb C^{\times}$-gerbes. Given a scheme $X\in\mathbf{Sch}^{\tn{ft}}_{/k}$, we let $X^{\tn{an}}$ denote its analytification. Let $\mathbf{An}_{/\mathbb C}^{\tn{ft}}$ denote the category of separated analytic spaces of finite type over $\mathbb C$. We write $\textbf{Tors}_{\mathbb C^{\times}}$ (resp.~$\textbf{Ge}_{\mathbb C^{\times}}$) for the presheaf of strict Picard $1$-groupoid of analytic $\mathbb C^{\times}$-torsors (resp.~$2$-groupoid of $\mathbb C^{\times}$-gerbes) on $\mathbf{An}_{\mathbb C}^{\tn{ft}}$.

\begin{lem}
\label{lem-analytic-comparison}
Let $X\in\mathbf{Sch}^{\tn{ft}}_{/\mathbb C}$. Then,
\begin{enumerate}[(a)]
	\item there is an equivalence of Picard $1$-groupoids
	$$
	\mathring{\mathbf{Loc}}_1(X)\xrightarrow{\sim}\mathbf{Tors}_{\mathbb C^{\times}}(X^{\tn{an}});
	$$
	\item there is a fully faithful functor of strict Picard $2$-groupoids whose image consists of those $\mathbb C^{\times}$-gerbes trivialized over $\widetilde X^{\tn{an}} \rightarrow X^{\tn{an}}$ for an $\eh$-cover $\widetilde X\rightarrow X$:
	$$
	\mathring{\mathbf{Ge}}(X) \hookrightarrow \mathbf{Ge}_{\mathbb C^{\times}}(X^{\tn{an}}).
	$$
\end{enumerate}
\end{lem}
\begin{proof}
(a) Recall that $\mathring{\mathbf{Loc}}_1$ satisfies $\tn h$-descent (Lemma \ref{lem-loc-h-descent}). On the other hand, the association $X\leadsto \tn R\Gamma(X^{\tn{an}}; \mathbb C^{\times})$ is also an $\tn h$-sheaf by cohomological descent of proper surjections of topological spaces \cite[Theorem 7.7]{conrad2003cohomological}. Since $\mathbf{Tors}_{\mathbb C^{\times}}(X^{\tn{an}})$ is the groupoid corresponding to $\tau^{\le 0}(\tn R\Gamma(X^{\tn{an}}; \mathbb C^{\times})[1])$, the association $X \leadsto \mathbf{Tors}_{\mathbb C^{\times}}(X^{\tn{an}})$ is also an $\tn h$-sheaf.\footnote{More generally, for an $\tn h$-hypercovering $\widetilde X^{\bullet} \rightarrow X$, the geometric realization of $(\widetilde X^{\bullet})^{\tn{an}}$ has homotopy type equivalent to $X^{\tn{an}}$, by a theorem of Blanc \cite[Proposition 3.21]{blanc2016topological}.} Thus the problem reduces to the case of smooth $X$. There, $\mathring{\mathbf{Loc}}_1(X)$ is the category of invertible objects inside regular, holonomic $\cal D$-modules on $X$, which lie in the heart when considered as objects of $\QCoh(X)$.

\smallskip

The Riemann--Hilbert correspondence is symmetric monoidal with respect to the $!$-monoidal structure on the constructible derived category $\tn{Shv}_c(X^{\tn{an}})$. In particular, it preserves invertible objects. On the other hand, the invertible objects in $\tn{Shv}_c(X^{\tn{an}})$ with respect to $!$ and $*$-monoidal structures agree via tensoring with the dualizing complex. Thus, we see that $\mathring{\mathbf{Loc}}_1(X)$ identifies with $*$-invertible objects in $\tn{Shv}(X^{\tn{an}})$ lying in the heart. The latter category identifies with $\mathbf{Tors}_{\mathbb C^{\times}}(X^{\tn{an}})$.

\smallskip

(b) The analytification functor $\mathbf{Sch}^{\tn{ft}}_{/\mathbb C} \rightarrow \mathbf{An}^{\tn{ft}}_{/\mathbb C}$ defines a map
$$
i_* : \tn{PSh}(\mathbf{An}^{\tn{ft}}_{/\mathbb C}) \rightarrow \tn{PSh}(\mathbf{Sch}^{\tn{ft}}_{/\mathbb C}).
$$
By the observation above, $i_*\mathbf{Tors}_{\mathbb C^{\times}}$ and $i_*\mathbf{Ge}_{\mathbb C^{\times}}$ are $\tn h$-sheaves on $\mathbf{Sch}^{\tn{ft}}_{/\mathbb C}$, so in particular are $\eh$-sheaves. On the other hand, part (a) gives an equivalence:
$$
\mathring{\mathbf{Loc}}_1 \xrightarrow{\sim} i_*\mathbf{Tors}_{\mathbb C^{\times}}.
$$
By delooping, we obtain a sequence of functors:
$$
\op B_{\eh}\mathring{\mathbf{Loc}}_1 \xrightarrow{\sim} (i_*\op B\mathbf{Tors}_{\mathbb C^{\times}})_{\eh} \hookrightarrow (i_*\op B_{\tn{an}}\mathbf{Tors}_{\mathbb C^{\times}})_{\eh} \xrightarrow{\sim} i_*\mathbf{Ge}_{\mathbb C^{\times}}.
$$
The middle functor is fully faithful and its image consists of $\eh$-locally trivial objects.
\end{proof}

\subsection{Twistings}
\label{sec-twistings}
\subsubsection{} The definition of tame twistings require us to work with the $\infty$-category $\mathbf{DSch}^{\tn{ft}}_{/k}$. We first extend $\mathring{\mathbf{Loc}}_1$ and $\mathring{\mathbf{Ge}}$ to $\mathbf{DSch}^{\tn{ft}}_{/k}$ by evaluation on the underlying classical scheme. By the commutative diagram \eqref{eq-eh-derived-comparison}, we see that $\mathring{\mathbf{Ge}}$ is the $\bfeh$-sheafification of $\op B\mathring{\mathbf{Loc}}_1$, regarded as a presheaf on $\mathbf{DSch}^{\tn{ft}}_{/k}$. Next, we consider the $\bfeh$-sheafification $\op B^2_{\bfeh}\mathbb G_m$. Define $\mathring{\mathbf{Tw}}$ as the fiber:
$$
\mathring{\mathbf{Tw}} := \op{Fib}(\mathring{\mathbf{Ge}} \rightarrow \op B^2_{\bfeh}\mathbb G_m).
$$
Thus $\mathring{\mathbf{Tw}}$ is an $\bfeh$-sheaf of strict Picard groupoids on $\mathbf{DSch}^{\tn{ft}}_{/k}$ whose sections are called \emph{tame twistings}. Furthermore, since $\mathring{\mathbf{Tw}}$ identifies with $\op B_{\bfeh}$ applied to:
$$
\tn{Fib}(\mathring{\mathbf{Loc}_1} \rightarrow \op B_{\et}\mathbb G_m) \hookrightarrow \tn{Fib}(\mathbf{Loc}_1 \rightarrow \op B_{\et}\mathbb G_m),
$$
which admits a $k$-linear structure \cite[\S6]{gaitsgory2011crystals}, we see that $\mathring{\mathbf{Tw}}$ is in fact valued in $\op Hk$-module objects in $\tn{ComGrp}(\mathbf{Spc})$. Furthermore, the fiber of the canonical map $\mathring{\mathbf{Tw}} \rightarrow \mathring{\mathbf{Ge}}$ identifies with $\op B_{\bfeh}\mathbb G_m$, but the tautological map $\op B_{\et}\mathbb G_m \rightarrow \op B_{\bfeh}\mathbb G_m$ is an equivalence by the $\bfeh$-descent of line bundles (Lemma \ref{lem-perf-eh-descent}). We thus obtain a fiber sequence:
\begin{equation}
\label{eq-pic-tw-ge}
\mathbf{Pic} \rightarrow \mathring{\mathbf{Tw}} \rightarrow \mathring{\mathbf{Ge}}.
\end{equation}

\subsubsection{} Extension by scalar defines the \emph{divisor class map} of tame twistings:
\begin{equation}
\label{eq-div-class-map-tw}
\mathbf{Pic} \underset{\mathbb Z}{\otimes} k \rightarrow \mathring{\mathbf{Tw}},\quad (\cal L, a) \leadsto \cal L^a.
\end{equation}
This map can also be constructed in a way analogous to \S\ref{sec-div-class-map} by first building a map:
$$
\mathbb G_m \underset{\mathbb Z}{\otimes} k \rightarrow \op{Fib}(\mathring{\mathbf{Loc}}_1 \rightarrow \op B_{\et}\mathbb G_m),\quad (f, a)\leadsto f^a
$$
using the $d\log$ construction over smooth schemes. Consequently, \eqref{eq-div-class-map-tw} is compatible with the divisor class map of tame gerbes \eqref{eq-div-class-map-tame} in the sense that the following diagram canonically commutes:
\begin{equation}
\label{eq-div-class-map-compatible}
\xymatrix@C=1.5em@R=0.3em{
	& \mathbf{Pic} \underset{\mathbb Z}{\otimes} k \ar[r]\ar[dd] & \mathbf{Pic} \underset{\mathbb Z}{\otimes} k/\mathbb Z \ar[dd] \\
	\mathbf{Pic} \ar[ur]\ar[dr] \\
	& \mathring{\mathbf{Tw}} \ar[r] & \mathring{\mathbf{Ge}}
}
\end{equation}

\subsubsection{} We now give an explicit description of tame twistings over a smooth scheme.

\begin{lem}
\label{lem-twisting-coh-interpretation}
Suppose $X\in\mathbf{Sm}_{/k}$. There is an equivalence:
$$
\op{DK}(\tau^{\le 0} \op R\Gamma_{\et}(X, \mathring{\Omega}^1[1])) \xrightarrow{\sim}\mathring{\mathbf{Tw}}(X).
$$
\end{lem}
\begin{proof}
Write provisionally $\mathring{\mathbf{Tw}}_{\et}$ for the sheaf on $\mathbf{Sm}_{/k}$ defined by $\tn{Fib}(\op B_{\et}\mathring{\mathbf{Loc}}_1 \rightarrow \op B^2_{\et}\mathbb G_m)$. Then we have a canonical map $\mathring{\mathbf{Tw}}_{\et} \rightarrow \mathring{\mathbf{Tw}}$ making the following diagram commute:
$$
\xysmall{
	\mathbf{Pic} \ar[r]\ar[d]^{\cong} & \mathring{\mathbf{Tw}}_{\et} \ar[r]^-{\alpha}\ar[d]^{\gamma_1} & \op B_{\et}\mathring{\mathbf{Loc}}_1 \ar[d]^{\gamma_2} \\
	\mathbf{Pic} \ar[r] & \mathring{\mathbf{Tw}} \ar[r] & \mathring{\mathbf{Ge}}
}
$$
The comparison Lemma \ref{lem-gerbe-etale-h} for tame gerbes shows that $\gamma_2$ is an equivalence. Since $\alpha$ is an \'etale local surjection, we see that $\gamma_1$ must also be an equivalence. The fact that $\mathring{\mathbf{Tw}}_{\et}$ is represented by the complex $\mathring{\Omega}^1[1]$ is a direct consequence of Lemma \ref{lem-classical-reg}.
\end{proof}

\subsubsection{}
\label{sec-map-to-usual-tw}
We now produce a morphism from $\mathring{\mathbf{Tw}}$ to the usual presheaf of twistings $\mathbf{Tw}$ defined in \cite{gaitsgory2011crystals}. Recall that the value of $\mathbf{Tw}$ on $X\in\mathbf{DSch}^{\tn{ft}}_{/k}$ can be given equivalently as:
\begin{align*}
	\mathbf{Tw}(X) := & \tn{Fib}(\op{Maps}(X_{\dR}, \op B^2_{\et}\mathbb G_m) \rightarrow \op{Maps}(X, \op B^2_{\et}\mathbb G_m)) \\
	& \xrightarrow{\sim} \tn{Fib}(\op{Maps}(X_{\dR}, \op B^2_{\et}\mathbb G_a) \rightarrow \op{Maps}(X, \op B^2_{\et}\mathbb G_a)).
\end{align*}

\begin{lem}
\label{lem-tw-descent}
The presheaf $\mathbf{Tw}$ on $\mathbf{DSch}^{\tn{ft}}_{/k}$ satisfies \emph{$\mathbf h$}-descent.
\end{lem}
\begin{proof}
Since the formation of de Rham prestack commutes with limits and given any $\mathbf h$-cover $\widetilde X\rightarrow X$ in $\mathbf{DSch}^{\tn{ft}}_{/k}$, the induced map $\widetilde X_{\dR} \rightarrow X_{\dR}$ is surjective in the $\mathbf h$-topology, it suffices to show that $\op B^2_{\et} \mathbb G_a$ satisfies $\mathbf h$-descent. On the other hand, $\tn{Maps}(X, \op B^2_{\et}\mathbb G_a)$ identifies with $\tau^{\le 0}\tn{Hom}(\cal O_X, \cal O_X[2])$ calculated in $\tn{Perf}(X)$, so the result follows from Lemma \ref{lem-perf-eh-descent}.
\end{proof}

Let us now construct the promised morphism:
\begin{equation}
\label{eq-tw-de-rham}
\mathring{\mathbf{Tw}} \rightarrow \mathbf{Tw}.
\end{equation}
We let $\mathbf{Ge}_{\dR}$ denote the \'etale stack which associates to $X\in\mathbf{DSch}^{\tn{ft}}_{/k}$ the strict Picard groupoid $\op{Maps}(X_{\dR}, \op B^2_{\et}\mathbb G_m)$. Taking the fibers along the vertical maps in the following commutative diagram:
$$
\xysmall{
	\op B_{\et}\mathring{\mathbf{Loc}}_1 \ar[r]\ar[d] & \op B_{\et}\mathbf{Loc}_1 \ar[r] & \mathbf{Ge}_{\dR} \ar[d] \\
	\op B^2_{\et}\mathbb G_m \ar[rr]^-{\sim} & & \op B^2_{\et}\mathbb G_m
}
$$
one obtains a morphism from $\op B_{\et}\op{Fib}(\mathring{\mathbf{Loc}}_1 \rightarrow \op B_{\et}\mathbb G_m)$ to $\mathbf{Tw}$. One then obtains \eqref{eq-tw-de-rham} by noting that $\mathbf{Tw}$ satisfies derived $\bfeh$-descent (Lemma \ref{lem-tw-descent}).

\subsubsection{}
Finally, we note that tame twistings can be used to produce a twisted category of $\cal D$-modules equipped with a forgetful functor to ind-coherent sheaves (as studied in \cite{gaitsgory2011ind}). Note that any object $\cal L \in \mathring{\mathbf{Loc}}(X)$ acts as automorphism on $\tn{Crys}^r(X)$:
\begin{equation}
\label{eq-crys-twisted}
\cal M \leadsto \cal M\otimes \cal L,
\end{equation}
and if the object in $\mathbf{Pic}(X)$ induced by $\cal L$ is trivialized, the underlying ind-coherent sheaves of $\cal M$ and $\cal M\otimes\cal L$ become canonically isomorphic.

\smallskip

Since both $\tn{Crys}^r$ and $\tn{IndCoh}$ are $\bfeh$-sheaves on $\mathbf{DSch}^{\tn{ft}}_{/k}$ \cite[Theorem 8.2.2]{gaitsgory2011ind}, the procedure of \cite[\S1.7.2]{gaitsgory2018parameters} defines for every $\cal T\in\mathring{\mathbf{Tw}}(X)$ a twisted category $\tn{Crys}^r_{\cal T}(X)$ equipped with a forgetful functor:
$$
\tn{oblv} : \tn{Crys}^r_{\cal T}(X) \rightarrow \tn{IndCoh}(X).
$$
This construction agrees with the usual twisted category defined by the twisting attached to $\cal T$ under the map \eqref{eq-tw-de-rham}. On the other hand, the full subcategory $\mathring{\tn{Crys}}{}^r(X) \subset \tn{Crys}^r(X)$ of regular $\cal D$-modules form an $\bfeh$-subsheaf. Since \eqref{eq-crys-twisted} preserves regularity (thank to tameness of $\cal L$), the same construction produces a full subcategory:
$$
\mathring{\tn{Crys}}{}_{\cal T}^r(X) \hookrightarrow \tn{Crys}^r_{\cal T}(X).
$$
In other words, the notion of \emph{regularity} makes sense for a crystal twisted by a tame twisting (or even a tame gerbe.)

\bigskip

\section{Motivic theory of gerbes}
\label{sec-motivic-gerbes}

In this section, we assume $k=\bar k$ but we remove the restriction on $\tn{char}(k)$.

\smallskip

We define the notion of a ``motivic theory of gerbes'' and note some consequences of the definition. Then we verify that \'etale, analytic, and tame gerbes are examples of such. We also include the example of ``additive'' de Rham gerbes which will be used in studying usual factorization twistings on the affine Grassmannian.

\subsection{Definitions}

\subsubsection{} Let $\mathbf G$ be an \'etale stack on $\mathbf{Sch}^{\tn{ft}}_{/k}$ valued in \emph{strict} Picard $2$-groupoids (c.f.~\S\ref{sec-picard-groupoids}). We write $A(-1)$ for the fiber of the restriction map $\mathbf G(\mathbb A^1) \rightarrow \mathbf G(\mathbb A^1\backslash\{\mathbf 0\})$ and think of it as a ``Tate twist'' of some coefficient group $A$ (although we do not define $A$). Note that \emph{a priori} $A(-1)$ is a strict Picard $2$-groupoid as opposed to an abelian group. We define a \emph{theory of gerbes} to be such $\mathbf G$, equipped with a map of stacks of strict Picard groupoids:
\begin{equation}
\label{eq-div-class-map-general}
\mathbf{Pic} \underset{\mathbb Z}{\otimes} A(-1) \rightarrow \mathbf G,\quad (\cal L, \lambda)\leadsto \cal L^{\lambda},
\end{equation}
which we shall call a \emph{divisor class map}. We will often refer to $\mathbf G$ as a theory of gerbes, the datum of \eqref{eq-div-class-map-general} being tacitly included. For $X\in\mathbf{Sch}^{\tn{ft}}_{/k}$, the notation $\mathbf G_X$ means the restriction of $\mathbf G$ to the small \'etale site of $X$.

\subsubsection{}
\label{sec-setup-topology} Let us fix a topology $t$ on $\mathbf{Sch}^{\tn{ft}}_{/k}$ which is finer than the \'etale topology and such that every $X\in\mathbf{Sch}^{\tn{ft}}_{/k}$ is $t$-locally smooth. Examples of $t$ include the $\eh$-topology when $\tn{char}(k) = 0$ and the $h$-topology in the general case. We call a theory of gerbes $\mathbf G$ a \emph{$t$-theory of gerbes} if $\mathbf G$ furthermore satisfies $t$-descent.

\subsubsection{} Here is a list of properties that we shall consider for a theory of gerbes $\mathbf G$.

\begin{enumerate}
	\item[(RP1)] $A(-1)$ is discrete, and for any $X\in\mathbf{Sm}_{/k}$ and $i : Z\hookrightarrow X$ a smooth divisor, the map of \'etale stacks induced from the divisor class map is an equivalence:
	$$
	\underline A(-1) \xrightarrow{\sim} \tn{Fib}(\mathbf G_X \rightarrow j_*\mathbf G_{X\backslash Z}),\quad a\leadsto \cal O_X(Z)^a.
	$$
	
	\smallskip
	
	\item[(RP2)] For any $X\in\mathbf{Sm}_{/k}$ and $i : Z\hookrightarrow X$ a closed subscheme of pure codimension $\ge 2$, the morphism is an equivalence:
	$$
	\mathbf G_X \xrightarrow{\sim} j_*\mathbf G_{X\backslash Z}.
	$$
	
	\smallskip
	
	\item[(A)] For any $X\in\mathbf{Sm}_{/k}$, the pullback morphism is an equivalence:
	$$
	\mathbf G(X) \xrightarrow{\sim} \mathbf G(X\times\mathbb A^1).
	$$
	
	\smallskip
	
	\item[(B)] For any proper morphism $p : Y\rightarrow X$ in $\mathbf{Sch}^{\tn{ft}}_{/k}$ and every $k$-point $x\in X$, the \'etale stalk $(p_*\mathbf G_Y)_x$ maps fully faithfully to the fiber $\mathbf G(Y\underset{X}{\times}\{x\})$. 
\end{enumerate}

\noindent
The names of these properties are \emph{relative purity in codimension $1$} (RP1), \emph{relative purity in codimension $\ge 2$} (RP2), \emph{$\mathbb A^1$-invariance} (A), and \emph{weak proper base change} (B). We call a $t$-theory of gerbes $\mathbf G$ satisfying all the above properties a \emph{motivic $t$-theory of gerbes}.

\subsubsection{} We note that although property (B) refers only to $k$-points, the assumption $k=\bar k$ guarantees that we have enough of them.

\begin{lem}
\label{lem-stalks-k-points}
Let $\mathbf F$ be an \'etale sheaf on $X\in\mathbf{Sch}^{\tn{ft}}_{/k}$ valued in strict Picard $n$-groupoids. If the stalk $\mathbf F_x = 0$ for all $k$-points $x\in X$. Then $\mathbf F=0$.
\end{lem}

\noindent
Thus a morphism $\mathbf F\rightarrow\mathbf G$ is an isomorphism if and only if its stalks at all $k$-points are.

\begin{proof}
It suffices to show $\pi_i\mathbf F$, the sheafification of $U\leadsto\mathbf F(U)$, vanishes. Since $(\pi_i\mathbf F)_{\bar\eta} = \pi_i(\mathbf F_{\bar{\eta}})$ for every geometric point $\bar{\eta} \rightarrow X$, the problem reduces to the case where $\mathbf F$ is valued in abelian groups. The problem then reduces to the fact that the \'etale neighborhood of any geometric point $\bar{\eta} \rightarrow \eta\in X$ contains a $k$-point in the closure of $\eta$.
\end{proof}

\subsection{Immediate consequences}

\subsubsection{} Relative purity in codimension $1$ can be generalized to the situation of multiple divisors.

\begin{lem}
\label{lem-purity-multiple}
Let $\mathbf G$ be a theory of gerbes satisfying \tn{(RP1)}. Then for any $X\in\mathbf{Sm}_{/k}$ together with a closed immersion $i : Z\hookrightarrow X$ where $Z$ is a finite union of smooth divisors $i_{\alpha} : Z_{\alpha}\hookrightarrow X$. Then the following map is an equivalence:
$$
\bigoplus_{\alpha} (i_{\alpha})_*\underline A(-1) \xrightarrow{\sim} \tn{Fib}(\mathbf G_X\rightarrow j_*\mathbf G_{X\backslash Z}), \quad (a_{\alpha})\leadsto \bigotimes_{\alpha}\cal O_X(Z_{\alpha})^a
$$
\end{lem}

\noindent
The conclusion is, of course, trivial if $\mathbf G$ also satisfies (RP2).

\begin{proof}
For notational simplicity, we only prove the case $Z = Z_1\cup Z_2$. Factor the open immersion $j : X\backslash Z\hookrightarrow X$ as such:
$$
X\backslash Z \xrightarrow{j_2} X\backslash Z_1 \xrightarrow{j_1} X,
$$
where the complement of $j_2$ is the locally closed subscheme $\mathring Z_2:=Z_2\backslash Z_1$. Applying relative purity to the open immersion $j_2$, we obtain a fiber sequence:
$$
(i_{\mathring Z_2})_*\underline A(-1) \rightarrow \mathbf G_{X\backslash Z_1} \rightarrow (j_2)_*\mathbf G_{X\backslash Z}.
$$
Applying $(j_1)_*$ to this fiber sequence. Using the fact that $\underline A(-1)$ is a constant sheaf so its pushforward under $j_1\circ i_{\mathring{Z}_2}$ identifies with $(i_2)_*\underline A(-1)$, we find a fiber sequence:
\begin{equation}
\label{eq-fiber-seq-1}
(i_2)_*\underline A(-1) \rightarrow (j_1)_*\mathbf G_{X\backslash Z_1} \rightarrow j_*\mathbf G_{X\backslash Z}.
\end{equation}
On the other hand, relative purity applied to the open immersion $j_1$ yields:
\begin{equation}
\label{eq-fiber-seq-2}
(i_1)_*\underline A(-1) \rightarrow \mathbf G_X\rightarrow (j_1)_*\mathbf G_{X\backslash Z_1}.
\end{equation}
Combining \eqref{eq-fiber-seq-1} and \eqref{eq-fiber-seq-2}, we see that the fiber of $\mathbf G_X\rightarrow j_*\mathbf G_{X\backslash Z}$ is an extension of $(i_2)_*\underline A(-1)$ by $(i_1)_*\underline A(-1)$. The symmetry of the situation implies that this extension canonically splits.
\end{proof}

\subsubsection{}
\label{sec-contraction}
We now explain that property (A) can be enhanced in the presence of $t$-descent. Namely, $\mathbf G$ is trivial on ``$\mathbb A^1$-contractible'' ind-schemes of ind-finite type. Note that by our convention, $X\in\mathbf{IndSch}^{\tn{ft}}_{/k}$ has the property that $X\rightarrow X\times X$ is a schematic closed immersion.

\smallskip

Given $X\in\mathbf{IndSch}^{\tn{ft}}_{/k}$ equipped with a $\mathbb G_m$-action, the action is called \emph{contracting} if it extends to an action of the multiplicative monoid $\mathbb A^1$. Such an extension is unique if it exists. Indeed, given two action maps $\xysmall{\mathbb A^1\times X\ar@<0.5ex>[r]^-{\tn{act}_1}\ar@<-0.5ex>[r]_-{\tn{act}_2} & X}$, the locus on which they agree maps to $\mathbb A^1\times X$ via a schematic closed immersion. Therefore, if the locus contains $\mathbb G_m\times X$, it is all of $\mathbb A^1\times X$.

\smallskip

Let $X^0 \hookrightarrow X$ be the fixed-point locus of a contracting $\mathbb G_m$-action. Then $X^0$ is again an ind-scheme of ind-finite type. We have a commutative diagram:
$$
\xysmall{
	\{\mathbf 0\} \times X \ar[r]^-q\ar[d] & X^0 \ar[d]^i \\
	\mathbb A^1 \times X \ar[r]^-{\tn{act}} & X
}
$$
Furthermore, the composition $X^0 \xrightarrow{i} \{\mathbf 0\}\times X \xrightarrow{q} X^0$ is the identity map. This is because $\mathbb G_m$ acts trivially on $X^0$, so it extends uniquely to the trivial $\mathbb A^1$-action.

\begin{lem}
\label{lem-contraction}
Suppose $\mathbf G$ is a motivic $t$-theory of gerbes satisfying $\tn{(A)}$.
\begin{enumerate}[(a)]
	\item For any $X\in\mathbf{IndSch}^{\tn{ft}}_{/k}$, the pullback morphism is an equivalence:
	$$
	\mathbf G(X) \xrightarrow{\sim} \mathbf G(X\times\mathbb A^1).
	$$
	\item Suppose $X\in\mathbf{IndSch}^{\tn{ft}}_{/k}$ is equipped with a contracting $\mathbb G_m$-action. Then restriction to the fixed-point locus is an equivalence:
	$$
	i^* : \mathbf G(X) \xrightarrow{\sim} \mathbf G(X^0).
	$$
\end{enumerate}
\end{lem}
\begin{proof}
For part (a), we first prove the result for $X\in\mathbf{Sch}^{\tn{ft}}_{/k}$. Indeed, take a $t$-hypercovering of $X$ consisting of smooth schemes $\widetilde X^{\bullet}$, the pullback $\widetilde X^{\bullet}\times\mathbb A^1$ is a $t$-hypercovering of $X\times \mathbb A^1$, so we win by $t$-descent. For the general case, we represent $X$ by $\underset{\nu}{\tn{colim}}\,X^{(\nu)}$ with $X^{(\nu)}\in\mathbf{Sch}^{\tn{ft}}_{/k}$. Then $X\times\mathbb A^1$ agrees with $\underset{\nu}{\tn{colim}}\,(X^{(\nu)}\times\mathbb A^1)$, so the result follows from the schematic case.

\smallskip

For part (b), we note that $\mathbb A^1$-invariance gives a canonical isomorphism of functors:
$$
\op{pr}^* \xrightarrow{\sim} \op{act}^* : \mathbf G(X) \rightarrow \mathbf G(\mathbb A^1\times X).
$$
Composing with the pullback to $\{\mathbf 0\} \times X$, we find that the identity functor on $\mathbf G(X)$ is equivalent to $q^*\circ i^*$. On the other hand, $i^*\circ q^*$ is the identity functor on $\mathbf G(X^0)$ as observed above, so the result follows.
\end{proof}

\subsubsection{} We now show that property (B) implies a K\"unneth type formula when some rigidity is assumed of one of the factors. For any $X\in \mathbf{Sch}^{\tn{ft}}_{/k}$, write $\mathbf G(X/\tn{pt})$ as the cofiber of $\mathbf G(\tn{pt}) \rightarrow \mathbf G(X)$ calculated in the $\infty$-category of strict Picard groupoids. Any choice of a $k$-point $x\in X$ identifies $\mathbf G(X/\tn{pt})$ with the fiber $\mathbf G(X;x)$ of $x^* : \mathbf G(X)\rightarrow \mathbf G(\tn{pt})$, i.e., gerbes rigidified at $x$. In particular, $\mathbf G(X/\tn{pt})$ is still a $2$-groupoid.

\begin{lem}
\label{lem-product-decomposition}
Let $\mathbf G$ be a theory of gerbes satisfying $\tn{(B)}$. Let $X_1, X_2\in\mathbf{Sch}^{\tn{ft}}_{/k}$ be connected schemes and furthermore suppose:
\begin{enumerate}[(a)]
	\item $X_1$ is proper, and
	\item $\mathbf G(X_1/\tn{pt})$ is discrete.
\end{enumerate}
Then the external product defines an equivalence:
\begin{equation}
\label{eq-kunneth-formula}
\boxtimes : \mathbf G(X_1/\tn{pt}) \times \mathbf G(X_2/\tn{pt}) \xrightarrow{\sim} \mathbf G(X_1\times X_2/\tn{pt}).
\end{equation}
\end{lem}
\begin{proof}
We let $\underline{\mathbf G(X_1)}$ be the \'etale sheafification of the constant presheaf with value $\mathbf G(X_1)$ on $X_2$ (and similarly for $\underline{\mathbf G(\tn{pt})}$). Let $p : X_1\times X_2\rightarrow X_2$ denote the projection map. External product defines a morphism:
\begin{equation}
\label{eq-external-from-pushout}
\boxtimes : \underline{\mathbf G(X_1)} \underset{\underline{\mathbf G(\tn{pt})}}{\sqcup} \mathbf G_{X_2} \rightarrow p_*\mathbf G_{X_1\times X_2}.
\end{equation}
Here, the push-out is calculated in the $\infty$-category of \'etale sheaves valued in strict Picard groupoids. We claim that \eqref{eq-external-from-pushout} is an equivalence. Indeed, it suffices to check that the stalks at every $k$-point $x_2 \in X_2$ agree (Lemma \ref{lem-stalks-k-points}).

\smallskip

Consider the stalk $\mathbf G_{X_2, x_2}$ of $\mathbf G_{X_2}$ at $x_2$. We first note that $\mathbf G(\tn{pt})\rightarrow\mathbf G_{X_2, x_2}$ is an equivalence since the restriction $\mathbf G_{X_2,x_2}\rightarrow\mathbf G(x_2)$ is fully faithful (Property (B)). Thus the composition:
$$
\mathbf G(X_1) \underset{\mathbf G(\tn{pt})}{\sqcup} \mathbf G_{X_2, x_2} \rightarrow (p_*\mathbf G_{X_1\times X_2})_{x_2} \rightarrow \mathbf G(X_1 \times\{x_2\})
$$
is an equivalence. Since the second map is fully faithful (Property (B)), the first map is an equivalence. This proves that \eqref{eq-external-from-pushout} is indeed an equivalence.

\smallskip

To prove that \eqref{eq-kunneth-formula} is an equivalence, we can fix points $x_1\in X_1$ and $x_2\in X_2$ and instead prove that the external product is an equivalence for rigidified gerbes:
\begin{equation}
\label{eq-kunneth-formula-rigidified}
\boxtimes : \mathbf G(X_1; x_1) \times \mathbf G(X_2; x_2) \rightarrow \mathbf G(X_1\times X_2; (x_1,x_2)).
\end{equation}
The splitting of $\mathbf G(X_1)$ as the bi-product\footnote{Recall: $\mathbf G$ takes values in connective $\tn H\mathbb Z$-module spectra, where bi-products make sense.} $\mathbf G(X_1; x_1) \times \mathbf G(\tn{pt})$ implies that $\underline{\mathbf G(X_1)} \underset{\underline{\mathbf G(\tn{pt})}}{\sqcup} \mathbf G_{X_2}$ is isomorphic to $\underline{\mathbf G(X_1; x_1)} \times \mathbf G_{X_2}$. Since $\mathbf G(X_1; x_1)$ is discrete and $X_2$ is connected, the global section of \eqref{eq-external-from-pushout} yields an equivalence:
$$
\mathbf G(X_1; x_1) \times \mathbf G(X_2) \xrightarrow{\sim} \mathbf G(X_1\times X_2).
$$
Adding the rigidification at $x_2$, respectively $(x_1, x_2)$, implies the equivalence \eqref{eq-kunneth-formula-rigidified}.
\end{proof}

\subsection{\'Etale context}

\subsubsection{} In this subsection, we fix a torsion abelian group $A$ the order of whose elements are indivisible by $p:=\tn{char}(k)$. We shall describe a motivic $\tn h$-theory of gerbes with coefficients in $A$. In practice, this gerbe theory can be used to twist the DG category of constructible \'etale $\overline{\mathbb Q}_{\ell}$-sheaves and $A$ will be a subgroup of $\overline{\mathbb Q}_{\ell}^{\times}$ (well chosen so that $A$ has no $p$-torsion). In the context of metaplectic Langlands program, this gerbe theory has been considered by Gaitsgory--Lysenko \cite{gaitsgory2018parameters}.

\subsubsection{} We define the sheaf $\mathbf{Ge}_{\et}$ of strict Picard $2$-groupoids on $\mathbf{Sch}^{\tn{ft}}_{/k}$ by:
$$
\mathbf{Ge}_{\et}(X) := \op{Maps}(X, \op B^2_{\et}A).
$$
Thus the fiber of $\mathbf{Ge}_{\et}(\mathbb A^1) \rightarrow \mathbf{Ge}_{\et}(\mathbb A^1\backslash\{\mathbf 0\})$ identifies with the usual Tate twist:
$$
A(-1) \xrightarrow{\sim} \underset{n\mid n'}{\tn{colim}}\;\tn{Hom}(\mu_n(k), A),
$$
where for $n\mid n'$, the transition map $\mu_{n'}(k) \rightarrow \mu_n(k)$ is given by raising to $(n'/n)$th power. As $A$ has no $p$-torsion, we may take $n$ to be indivisible by $p$ in this colimit. Since $k = \bar k$, the map $\underline{\mu_n(k)} \rightarrow \mu_n$ is an isomorphism of \'etale sheaves on $\mathbf{Sch}_{/k}^{\tn{ft}}$. Therefore $A(-1)$ is also the colimit of Hom-groups of \'etale sheaves $\underset{n\mid n'}{\tn{colim}}\;\tn{Hom}(\mu_n, \underline A)$.

\subsubsection{} The divisor class map:
$$
\mathbf{Pic} \underset{\mathbb Z}{\otimes} A(-1) \rightarrow \mathbf{Ge}_{\et},\quad (\cal L, a)\leadsto \cal L^a
$$
can be constructed as follows (c.f.~\cite[\S1.4]{gaitsgory2018parameters}). The Kummer exact sequence gives rise to  a map $\theta_n : \mathbf{Pic} \rightarrow \op B_{\et}^2\mu_n$ for each $n$ indivisible by $p$, such that for $n\mid n'$ the following diagram commutes:
$$
\xysmall{
& \op B^2_{\et}\mu_{n'} \ar[dd]^{(-)^{n'/n}} \\
\mathbf{Pic} \ar[ur]^-{\theta_{n'}}\ar[dr]^-{\theta_n} &  \\
& \op B^2_{\et}\mu_{n}
}
$$
Therefore, a pair $(\cal L, a)$ gives rise to a section of $\op B^2_{\et}A$, to be denoted by $\cal L^a$.

\subsubsection{} The properties (RP1), (RP2), (A), and (B) are all standard facts of \'etale cohomology. Finally, we claim that $\mathbf{Ge}_{\et}$ satisfies $\tn h$-descent. This follows from the fact that proper coverings satisfy cohomological descent for \'etale sheaves of $A$-modules \cite[Theorem 7.7]{conrad2003cohomological}.

\smallskip

Alternatively, by a theorem of Suslin--Voevodsky \cite{suslin1996singular}, $\underline A$ is a sheaf in the $\tn h$-topology and one has canonical isomorphisms:
$$
\op H^i_{\et}(X; A) \xrightarrow{\sim} \op H^i_{\tn h}(X; A),\quad\text{for all }i\ge 0.
$$
In particular, this shows that \'etale $A$-gerbes agree with $A$-gerbes in the $\tn h$-topology. In conclusion, $\mathbf{Ge}_{\et}$ is a motivic $\tn h$-theory of gerbes.

\subsection{Analytic context}
\label{sec-an-context}

\subsubsection{} We now fix $k=\mathbb C$. Let $\mathbf{Ge}_{\tn{an}}$ denote the presheaf of strict Picard $2$-groupoids on $\mathbf{Sch}^{\tn{ft}}_{/k}$ which associates $\mathbb C^{\times}$-gerbes over the analytification:
$$
\mathbf{Ge}_{\tn{an}}(X) := \mathbf{Ge}_{\mathbb C^{\times}}(X^{\tn{an}}).
$$
Equivalently, $\mathbf{Ge}_{\tn{an}}(X)$ is the space of maps from the homotopy type of $X^{\tn{an}}$ to the Eilenberg--MacLane space $K(2; \mathbb C^{\times})$. We have already noted (in \S\ref{sec-analytic-comparison}) that $\mathbf{Ge}_{\tn{an}}$ is an $\tn h$-sheaf on $\mathbf{Sch}^{\tn{ft}}_{/k}$. Its coefficient group $A(-1)$ identifies with $\mathbb C^{\times}$.

\subsubsection{} The properties (RP1), (RP2), and (A) are standard facts. To verify the weak proper base change property (B), we shall show that the restriction map:
\begin{equation}
\label{eq-restriction-analytic}
\underset{U}{\tn{colim}} \op H^i(Y^{\tn{an}}\underset{X^{\tn{an}}}{\times} U^{\tn{an}}; \mathbb C^{\times}) \rightarrow \op H^i(Y^{\tn{an}} \underset{X^{\tn{an}}}{\times}\{x\}; \mathbb C^{\times}),\quad i\ge 0,
\end{equation}
where $U$ ranges over \'etale neighborhoods of $x\in X$, is in fact an isomorphism. Note that there is an exact sequence of abelian groups:
$$
0 \rightarrow \mathbb C^{\times}_{\tn{tors}} \rightarrow \mathbb C^{\times} \rightarrow \mathbb C/\mathbb Q \rightarrow 0,
$$
where $\mathbb C^{\times}_{\tn{tors}}$ denotes the torsion subgroup of $\mathbb C^{\times}$. By Artin's comparison theorem, the map \eqref{eq-restriction-analytic} is an isomorphism for coefficients in $\mathbb C^{\times}_{\tn{tors}}$ and $\mathbb Q_{\ell}$ for any prime $\ell$. The same statement must also be true for coefficients in $\mathbb Q$ as the operation $-\underset{\mathbb Q}{\otimes}\mathbb Q_{\ell}$ is conservative. Thus it remains true for $\mathbb C/\mathbb Q$ as it is a direct sum of copies of $\mathbb Q$. This implies the result for coefficients in $\mathbb C^{\times}$. We conclude that $\mathbf{Ge}_{\tn{an}}$ is a motivic $\tn h$-theory of gerbes.

\subsection{De Rham context}
\label{sec-tame-gerbe-motivic}

\subsubsection{} Fix $k = \bar k$ with $\tn{char}(k) = 0$. The na\"ive theory of de Rham gerbes sending $X\in\mathbf{Sch}^{\tn{ft}}_{/k}$ to $\op{Maps}(X_{\dR}, \op B^2_{\et}\mathbb G_m)$ is not motivic; it fails, for instance, $\mathbb A^1$-invariance. This can be seen as the \emph{raison d'\^{e}tre} of the theory of tame gerbes.

\subsubsection{} We shall verify that the sheaf of tame gerbes $\mathring{\mathbf{Ge}}$ defined in \S\ref{sec-gerbes} is a motivic $\eh$-theory of gerbes. In fact, the properties (RP1), (A), and (B) follow immediately from the analytic comparison Lemma \ref{lem-analytic-comparison} and the corresponding properties of $\mathbf{Ge}_{\tn{an}}$. Indeed, take $k$ to be $\mathbb C$ and (RP1) is verified because $\mathring{\mathbf{Ge}}$ is a full subfunctor of $\mathbf{Ge}_{\tn{an}}$. To see (A), we consider the commtutive square when $k=\mathbb C$:
$$
\xysmall{
	\mathring{\mathbf{Ge}}(X) \ar@{^{(}->}[r]\ar[d] & \mathbf{Ge}_{\tn{an}}(X) \ar[d]^{\cong} \\
	\mathring{\mathbf{Ge}}(X\times\mathbb A^1) \ar@{^{(}->}[r] & \mathbf{Ge}_{\tn{an}}(X\times\mathbb A^1)
}
$$
Thus $\mathring{\mathbf{Ge}}(X)\rightarrow\mathring{\mathbf{Ge}}(X\times\mathbb A^1)$ is fully faithful. It is essentially surjective since there is a retraction $\mathring{\mathbf{Ge}}(X\times\mathbb A^1) \rightarrow \mathring{\mathbf{Ge}}(X)$ and two objects $\cal G_1, \cal G_2 \in \mathring{\mathbf{Ge}}(X\times\mathbb A^1)$ are identified once they are identified in $\mathbf{Ge}_{\tn{an}}(X\times\mathbb A^1)$. To prove (B), we observe that the commutative diagram below consists of fully faithful embeddings:
$$
\xysmall{
	(p_*\mathring{\mathbf{Ge}}_Y)_x \ar[r]\ar@{^{(}->}[d] & \mathring{\mathbf{Ge}}(Y \underset{X}{\times}\{x\}) \ar@{^{(}->}[d] \\
	(p_*\mathbf{Ge}_{\tn{an}, Y})_x \ar[r]^-{\sim} & \mathbf{Ge}_{\tn{an}}(Y\underset{X}{\times}\{x\})
}
$$
Below, we shall present algebraic proofs of (RP1), (RP2), and (A). They are based on the calculation of cohomology of $\mathring{\Omega}^1$ and several known facts about the Brauer group. Unfortunately, we have not found an algebraic proof of (B).

\subsubsection{$\tn{(RP1)}$} Over $X\in\mathbf{Sm}_{/k}$, the \'etale sheaf $\mathring{\mathbf{Ge}}_X$ is represented by the complex
\begin{equation}
\label{eq-tame-gerbe-complex}
\cal G_X := \tn{Cofib}(\cal O_X^{\times} \rightarrow \mathring{\Omega}_X^1)[1].
\end{equation}
It suffices calculate the (derived) restriction $\tau^{\le 0}i^!\cal G$. Note that $i^!$ is a left-exact functor on \'etale sheaves, so \eqref{eq-tame-gerbe-complex} gives rise to a long exact sequence:
\begin{align}
0 \rightarrow \op H^{-2}i^!\cal G \rightarrow &\op H^0i^! \cal O_X^{\times} \rightarrow \op H^0i^! \mathring{\Omega}_X^1 \rightarrow \op H^{-1}i^! \cal G \notag \\
&\rightarrow \op H^1i^!\cal O_X^{\times} \xrightarrow{\beta} \op H^1i^!\mathring{\Omega}_X^1 \rightarrow \op H^0i^!\cal G \rightarrow \op H^2i^!\cal O_X^{\times}. \label{eq-long-restr-seq}
\end{align}
We make the following observations based on the tautological triangle for an \'etale sheaf $\cal F$:
$$
i_* i^!\cal F \rightarrow \cal F \rightarrow \op Rj_*(\cal F\big|_{X\backslash Z}).
$$
\begin{enumerate}[(a)]
	\item $\op H^0i^!\cal O_X^{\times}=0$ and $\op H^0i^!\mathring{\Omega}_X^1 = 0$;
	\smallskip
	\item $\op H^1i^!\cal O_X^{\times} \xrightarrow{\sim} \underline{\mathbb Z}$, since this group identifies as the cokernel of $\cal O_X^{\times} \rightarrow j_*\cal O_{X\backslash Z}^{\times}$; the analogous consideration gives $\op H^1i^!\mathring{\Omega}_X^1 \xrightarrow{\sim} \underline k$, and the morphism $\beta$ passes to the tautological inclusion $\underline{\mathbb Z}\rightarrow \underline k$.
	\smallskip
	\item $\op H^2i^!\cal O_X^{\times} = 0$, since this group identifies with $\op R^1j_*\cal O_{X\backslash Z}^{\times}$, which vanishes because every line bundle on $X\backslash Z$ extends across $Z$.
\end{enumerate}
Combining the above observations, we obtain $\op H^{-2}i^!\cal G = 0$, $\op H^{-1}i^!\cal G=0$, and $\op H^0i^!\cal G\xrightarrow{\sim} k/\mathbb Z$. It is straghtforward to see that this isomorphism agrees with \eqref{eq-div-class-map-tame}.

\subsubsection{$\tn{(RP2)}$}
The descent property of $\mathring{\mathbf{Ge}}$ allows to assume $X$ is affine. We again use the complex $\cal G_X$ \eqref{eq-tame-gerbe-complex}, and the result reduces to the following calculations of cohomology groups:
\begin{enumerate}[(a)]
	\item $\op H^i_{\et}(X; \cal O^{\times}_X) \xrightarrow{\sim}  \op H^i_{\et}(X\backslash Z; \cal O^{\times}_X)$ for $i=0,1,2$. The nontrivial part is $i=2$ which follows from purity of the Brauer group for smooth schemes over a field (see Gabber \cite[\S2]{gabber1993injectivity});
	\smallskip
	\item $\op H^i_{\et}(X; \mathring{\Omega}^1_X) \xrightarrow{\sim} \op H^i_{\et}(X\backslash Z; \mathring{\Omega}_X^1)$ for $i=0,1,2$. This follows from the \'etale-to-Zariski comparison and the Gersten resolution (Theorem \ref{thm-coh-properties}).
\end{enumerate}

\subsubsection{$\tn{(A)}$} Proceeding as above, it suffices to establish $\mathbb A^1$-invariance of the following groups:
\begin{enumerate}[(a)]
	\item $\op H^i_{\et}(X; \cal O_X^{\times})$ for $i=0,1,2$. The case for $i=0$ is immediate. For $i=1$, this is the $\mathbb A^1$-invariance of the Picard group over a regular base. For $i=2$, one first identifies $\op H^2_{\et}(X; \cal O_X^{\times})$ with the Brauer group using Gabber's theorem \cite{de2003result}, and then appeals to the theorem of Auslander--Goldman \cite[Proposition 7.7]{auslander1960brauer} (this requires $\op{char}(k)=0$.)
	
	\smallskip
	
	\item $\op H^i_{\et}(X; \mathring{\Omega}^1_X)$ for $i=0,1$. These have been established in Theorem \ref{thm-coh-properties}.
\end{enumerate}

\subsubsection{}
\label{sec-twisting-as-gerbe} Finally, we remark that the additional player in the de Rham context---tame twistings---is a theory of gerbes by construction (c.f.~\S\ref{sec-twistings}). Its coefficient group $A(-1)$ identifies with $k$. However, $\mathring{\mathbf{Tw}}$ does not satisfy $\eh$-descent since it is not nil-invariant. On the other hand, $\mathring{\mathbf{Tw}}$ verifies properties (RP1), (RP2), and (A). Indeed, by the fiber sequence \eqref{eq-pic-tw-ge} and its compatibility with the divisor class maps \eqref{eq-div-class-map-compatible}, these properties follow from the corresponding ones for $\mathbf{Pic}$ and $\mathring{\mathbf{Ge}}$.\footnote{In fact, the previous discussion already includes a direct proof of these facts for $\mathring{\mathbf{Tw}}$.} It is worth pointing out that $\mathbf{Pic}$ is also theory of gerbes according to our definition, with $A(-1) = \mathbb Z$. It satisfies properties (RP1), (RP2), and (A).

\subsection{Additive de Rham context}
\label{sec-add-dr-context}

\subsubsection{} We remark on another theory of gerbes supplied by algebraic de Rham cohomology valued in $\mathbb G_a$. These gerbes are not used to form any twisted category of sheaves.

\smallskip

We remain in the setting where $k = \bar k$ with $\tn{char}(k) = 0$.

\subsubsection{} We define $\mathbf{Ge}_{\dR}^+$ as a presheaf of strict Picard $2$-groupoids on $\mathbf{Sch}^{\tn{ft}}_{/k}$ by:
$$
\mathbf{Ge}_{\dR}^+(X) := \tn{Maps}(X_{\dR}, \op B^2_{\tn{Zar}}\mathbb G_a).
$$
Therefore, $\mathbf{Ge}_{\dR}^+(X)$ is calculated by the truncated complex $\tau^{\le 0}\tn R\Gamma_{\tn{Zar}}(X_{\dR}, \cal O[2])$. The $\mathbf h$-descent of perfect complexes (Lemma \ref{lem-perf-eh-descent}) implies that $\mathbf{Ge}_{\dR}^+$ is an $\tn h$-stack. Indeed, for every $\tn h$-cover $\widetilde X\rightarrow X$, the \v Cech complex of $\widetilde X_{\dR} \rightarrow X_{\dR}$ is canonically the same whether formed as classical or derived prestacks.\footnote{In particular, we can replace the Zariski topology in the definition of $\mathbf{Ge}^+_{\tn{dR}}$ by the \'etale topology.}

\subsubsection{} The value group $A(-1)$ canonically identifies with $k$. The divisor class map:
$$
\mathbf{Pic} \underset{\mathbb Z}{\otimes} k \rightarrow \mathbf{Ge}_{\dR}^+, \quad (\cal L, a) \leadsto \cal L^a
$$
is the ``first Chern class'' construction. Over a smooth scheme $X$, it is induced from $d\log : \cal O_X^{\times} \rightarrow \tau^{\le 2}\Omega_X^{\bullet}$. For general $X\in\mathbf{Sch}^{\tn{ft}}_{/k}$, there is a morphism from $\mathbf{Pic}(X)$ to usual twistings $\mathbf{Tw}(X)$ (c.f.~\S\ref{sec-map-to-usual-tw}) which has an underlying $\mathbb G_a$-gerbe on $X_{\dR}$. One then extends the construction by $k$-linearity.

\subsubsection{} For $k = \mathbb C$, the theory of gerbes $\mathbf{Ge}_{\dR}^+$ is equivalent to analytic $\mathbb C$-gerbes, up to a Tate twist of the divisor class map. More precisely, we let $\mathbf{Ge}_{\tn{an}}^+$ denote the presheaf on $\mathbf{Sch}^{\tn{ft}}_{/k}$ which associates to $X$ the strict Picard $2$-groupoid of $\mathbb C$-gerbes on $X^{\tn{an}}$. In other words, $\mathbf{Ge}^+_{\tn{an}}(X)$ is calculated by the truncated complex $\tau^{\le 0} \tn C^{\bullet}(X^{\tn{an}}; \mathbb C[2])$ of topological cochains valued in $\mathbb C$. The same argument as for $\mathbb C^{\times}$ shows that $\mathbf{Ge}^+_{\tn{an}}$ is an $\tn h$-sheaf.

\smallskip

The coefficient group $A(-1)$ is easily seen to be $\mathbb C$. There is a topological Chern class map:
$$
\mathbf{Pic} \underset{\mathbb Z}{\otimes} \mathbb C \rightarrow \mathbf{Ge}_{\tn{an}}^+,\quad (\cal L, a)\leadsto \cal L^a,
$$
where the image of $\mathbf{Pic}$ lies in $\tau^{\le 0}\tn C^{\bullet}(X^{\tn{an}}; \mathbb Z[2])$.

\subsubsection{} Applying Grothendieck's comparison theorem in the smooth case and using $\tn h$-descent, we find an equivalence making the following diagram commute.
$$
\xysmall{
\mathbf{Pic}\underset{\mathbb Z}{\otimes}\mathbb C \ar[r]\ar[d]^{\tn{id}} & \mathbf{Ge}^+_{\dR} \ar[d]^{\cong} \\
\mathbf{Pic}\underset{\mathbb Z}{\otimes}\mathbb C \ar[r]^-{2\pi i\cdot } & \mathbf{Ge}^+_{\tn{an}}
}
$$
i.e., the divisor class map for $\mathbf{Ge}^+_{\tn{an}}$ has to be multiplied by a factor of $2\pi i$.

\subsubsection{} The theories of gerbes $\mathbf{Ge}_{\dR}^+$ and $\mathbf{Ge}^+_{\tn{an}}$ satisfy the properties (RP1), (RP2), (A), and (B). One can either prove these properties directly for algebraic de Rham cohomology, or use the argument in \S\ref{sec-an-context} for $\mathbf{Ge}^+_{\tn{an}}$ and transfer the results to $\mathbf{Ge}_{\dR}^+$. In summary, they are both motivic $\tn h$-theories of gerbes.

\bigskip

\section{Factorization structure and $\Theta$-data}
\label{sec-fact-theta}

In this section, we assume $k=\bar k$. We further fix a smooth, connected curve $X$ over $k$.

\smallskip

After a review of factorization structures and the affine Grassmannian $\tn{Gr}_{G, \Ran}$, our first goal will be to define the combinatorial gadget of ``enhanced $\Theta$-data'' (\S\ref{sec-enh-theta-data}.) Then we state the classification of factorization gerbes on $\tn{Gr}_{G, \Ran}$ (for any motivic theory of gerbes) and deduce from it the classification of factorization tame twistings (Theorem \ref{thm-twisting-classification}). This fulfills the task of assigning an intrinsic meaning to quantum parameters.

\smallskip

Finally, we address the question of factorization (usual) twistings on $\tn{Gr}_{G,\Ran}$ and classify them for semisimple, simply connected $G$.

\subsection{Factorization gerbes}

\subsubsection{} Let $\tn{Ran}$ denote the prestack on $\mathbf{Sch}^{\tn{ft}}_{/k}$ whose $S$-points are finite sets of maps $x^{(i)} : S\rightarrow X$. Write $\mathbf{fSet}^{\tn{surj}}$ for the category of finite nonempty sets $I$ together with surjective maps $I\twoheadrightarrow J$. The the canonical map $\underset{I \in \mathbf{fSet}^{\tn{surj}}}{\tn{colim}} X^I \rightarrow \tn{Ran}$ is an equivalence of prestacks.

\subsubsection{} For $n\ge 1$, we let $\tn{Ran}^{\times n}_{\tn{disj}}$ denote the open sub-prestack of $\tn{Ran}^{\times n}$ consisting of points $\{x^{(i)}\}_{i\in I_k, 1\le k\le n}$ such that $x^{(i)}$ and $x^{(j)}$ are disjoint as long as $i,j$ belong to $I_k$ and $I_{k'}$ for $k\neq k'$. There is a morphism of ``disjoint union'':
$$
\sqcup_{(n)} : \tn{Ran}^{\times n}_{\tn{disj}} \rightarrow \tn{Ran}.
$$
We shall only be concerned with classical (i.e., non-derived) factorization prestacks valued in \emph{sets}. Let us recall that a \emph{factorization prestack} over $X$ is a prestack $\cal Y$ over $\tn{Ran}$ equipped with the additional data, called a \emph{factorization isomorphism} over $\tn{Ran}^{\times 2}_{\tn{disj}}$:
$$
f_{(2)} : \sqcup_{(2)}^* \cal Y \xrightarrow{\sim} (\cal Y \times \cal Y)_{\tn{disj}}.
$$

\smallskip
The isomorphism $f_{(2)}$ is required to satisfy a coherence condition over $\tn{Ran}^{\times 3}_{\tn{disj}}$ expressing that the three ways on can form an isomorphism $\sqcup^*_{(3)}\cal Y \xrightarrow{\sim} (\cal Y\times\cal Y\times\cal Y)_{\tn{disj}}$ out of $f_{(2)}$ are identical. A convenient way to express this is as follows. We assume to be given:
$$
f_{(3)} : \sqcup_{(3)}^*\cal Y \xrightarrow{\sim} (\cal Y \times \cal Y \times\cal Y)_{\tn{disj}},
$$
such that for each surjection $\varphi : \{1,2,3\} \rightarrow \{1,2\}$, the map $\sqcup_{\varphi} : \tn{Ran}^{\times 3}_{\tn{disj}} \rightarrow \tn{Ran}^{\times 2}_{\tn{disj}}$ of taking unions along each of $\varphi$ makes the following diagram commute.
$$
\xysmall{
	\sqcup^*_{(3)}\cal Y \ar[r]_-{\sim}^-{f_{(3)}}\ar[d]_{\cong} & (\cal Y\times \cal Y\times\cal Y)_{\tn{disj}} \\
	\sqcup^*_{\varphi}\sqcup^*_{(2)}\cal Y \ar[r]_-{\sim}^-{\sqcup_{\varphi}^*f_{(2)}} & \sqcup^*_{\varphi}(\cal Y\times\cal Y)_{\tn{disj}} \ar[u]^{\cong}_{f_{(2), \varphi}}
}
$$
Here, $f_{(2), \varphi}$ means applying $f_{(2)}$ on the factor corresponding to the element of $\{1,2\}$ with two preimages. Clearly, $f_{(3)}$ is not an additional piece of structure.

\subsubsection{} Let us be given a presheaf $\mathbf F$ on $\mathbf{Sch}^{\tn{ft}}_{/k}$ valued in strict Picard $2$-groupoids. We extend $\mathbf F$ to prestacks by the process of right Kan extension:
$$
\mathbf F(\cal Y) = \lim_{\substack{X\rightarrow\cal Y \\ X\in\mathbf{Sch}^{\tn{ft}}_{/k}}} \mathbf F(X).
$$
Suppose $\cal Y$ is a factorization prestack over $X$. Then a \emph{factorization section} $\cal S\in \mathbf F^{\tn{fact}}(\cal Y)$ is a section $\cal S \in \mathbf F(\cal Y)$ equipped with \emph{factorization isomorphisms}:
$$
\sqcup^*_{(n)} \cal S \xrightarrow{\sim} \cal S^{\boxtimes n}\text{ in }\mathbf F(\sqcup^*_{(n)}\cal Y \xrightarrow{\sim} (\cal Y^{\times n})_{\tn{disj}}),
$$
for $n=2,3$. Furthermore, for each surjection $\varphi : \{1,2,3\}\rightarrow\{1,2\}$, we are supplied a $2$-isomorphism witnessing the commutativity of the following diagram:
$$
\vcenter{
\xysmall{
	\sqcup^*_{(3)}\cal S \ar[r]^-{\sim}\ar[d]_{\cong} & \cal S \boxtimes\cal S \boxtimes\cal S \ar@{=>}[dl] \\
	\sqcup_{\varphi}^*\sqcup_{(2)}^*\cal S \ar[r]^-{\sim} & \sqcup_{\varphi}^*(\cal S\boxtimes\cal S) \ar[u]^{\cong}
}}
\quad\text{in}\quad
\mathbf F\left(\vcenter{
\xysmall{
	\sqcup^*_{(3)}\cal Y \ar[r]^-{\sim}\ar[d]_{\cong} & (\cal Y\times \cal Y\times\cal Y)_{\tn{disj}} \\
	\sqcup^*_{\varphi}\sqcup^*_{(2)}\cal Y \ar[r]^-{\sim} & \sqcup^*_{\varphi}(\cal Y\times\cal Y)_{\tn{disj}} \ar[u]^{\cong}
}
}\right).
$$
These $2$-isomorphisms are required to satisfy a coherence condition over $\tn{Ran}^{\times 4}_{\tn{disj}}$ which we shall not specify.

\smallskip

Thus $\mathbf F^{\tn{fact}}(\cal Y)$ naturally forms a strict Picard $2$-groupoid, and the forgetful map $\mathbf F^{\tn{fact}}(\cal Y) \rightarrow \mathbf F(\cal Y)$ is a morphism of such. In the particular case where $\mathbf G$ is a theory of gerbes, we call sections of $\mathbf G^{\tn{fact}}(\cal Y)$ \emph{factorization gerbes} on $\cal Y$.

\subsection{The affine Grassmannian}

\subsubsection{} We shall now introduce the main example of a factorization prestack: the \emph{affine Grassmannian} $\tn{Gr}_{H,\tn{Ran}}$ associated to $X$ and a linear algebraic group $H$.

\smallskip

It is defined as the (classical) prestack over $\tn{Ran}$ whose fiber at an $S$-point $x^{(i)} : S\rightarrow X$ is the set of pairs $(\cal P_H, \alpha)$ where $\cal P_H$ is an \'etale $H$-torsor over $S\times X$ and $\alpha$ is a trivialization of $\cal P_H$ on the complement of the graphs:
$$
\alpha : \cal P_H \xrightarrow{\sim} \cal P_H^0 \big|_{S\times X\backslash\bigcup_{i\in I}\Gamma_{x^{(i)}}}.
$$
The Beauville--Laszlo lemma shows that $\tn{Gr}_{H, \tn{Ran}}$ has the structure of a factorization prestack over $X$ (c.f.~\cite{zhu2016introduction}).

\smallskip

Furthermore, the projection:
\begin{equation}
\label{eq-gr-projection}
\pi : \tn{Gr}_{H, \tn{Ran}} \rightarrow \tn{Ran}
\end{equation}
is ind-schematic and ind-finite type, i.e., for every $S\in\tn{Ran}$ with $S\in\mathbf{Sch}^{\tn{ft}}_{/k}$, the fiber product $\tn{Gr}_{H, \tn{Ran}} \underset{\tn{Ran}}{\times} S$ is representable by an ind-scheme of ind-finite type. When $H$ is reductive, $\pi$ is furthermore ind-proper \cite[Theorem 3.1.3]{zhu2016introduction}. For a finite set $I$, we will denote by $\tn{Gr}_{H, X^I}$ the fiber product:
$$
\tn{Gr}_{H, X^I} := \tn{Gr}_{H, \tn{Ran}} \underset{\tn{Ran}}{\times} X^I.
$$
The morphism \eqref{eq-gr-projection} admits a \emph{unit} section, defined by sending $x^{(i)}$ to the trivial $H$-torsor $\cal P_H^0$ equipped with the tautological trivialization:
$$
e : \tn{Ran} \rightarrow \op{Gr}_{H, \tn{Ran}}.
$$

\subsubsection{} Fixing a $k$-point $x\in X$ and a uniformizer $t$ of the completed local ring $\widehat{\cal O}_{X,x}$, the fiber of $\tn{Gr}_{H, \tn{Ran}}$ at $x$ identifies with the \'etale quotient of the loop group by the arc group  $H\loo{t}/H\arc{t}$. This is the ``classical version'' of the affine Grassmannian.

\smallskip

For $G$ reductive, let $I\subset G\arc{t}$ denote the Iwahori subgroup associated to the Borel $B$. Then the quotient $G\loo{t}/I$ is the affine flag variety $\tn{Fl}_G$. The projection:
$$
\tn{Fl}_G \rightarrow \tn{Gr}_{G, x}
$$
is an \'etale locally trivial fiber bundle with typical fiber $G/B$.

\subsubsection{}
\label{sec-schubert}
For $G$ semisimple and simply connected, $\tn{Gr}_{G, X^I}$ admits a well-behaved Schubert stratification. It is a closed subscheme $\tn{Gr}_{G, X^I}^{\le \lambda^I}$ associated to any $I$-tuple $\lambda^I := (\lambda^{(i)})$ of dominant cocharacters $\lambda^{(i)}\in\Lambda_T^+$. The ind-scheme $\tn{Gr}_{G, X^I}$ identifies with the colimit of $\tn{Gr}^{\le\lambda^I}_{G, X^I}$ over $\lambda^I$.

\smallskip

The Schubert varieties $\tn{Gr}_{G, X^I}^{\le\lambda^I}$ are flat over $X^I$. Furthermore, for every $\varphi : I\twoheadrightarrow J$, the restriction of $\tn{Gr}_{G, X^I}^{\le\lambda^I}$ to the diagonal $\Delta_{I\twoheadrightarrow J}$ identifies with $\tn{Gr}_{G, X^J}^{\le\lambda^J}$ where $\lambda^{(j)}:= \sum_{i\in\varphi^{-1}(j)}\lambda^{(i)}$ (see \cite[Proposition 1.2.4]{zhu2009affine} for the case $I=\{1,2\}$; the general case is similar).

\subsubsection{} Let $\mathbf{Pic}_{\tn{Gr}_{G, X^I}}^e$ denote the \'etale sheaf on $X^I$ which associates to $S\rightarrow X^I$ the abelian group of line bundles on $\tn{Gr}_{G, X^I}\underset{X^I}{\times}S$ trivialized over the unit section $e$. The following exact sequence is a mild generalization of \cite[Lemma 3.4.3]{zhu2016introduction} to $G$ semisimple, simply connected:
\begin{equation}
\label{eq-pic-exact-seq}
0\rightarrow \mathbf{Pic}^e_{\tn{Gr}_{G, X^I}} \rightarrow \boxtimes_{i\in I} \underline{B}_X \rightarrow \bigoplus_{\substack{I\twoheadrightarrow J\\ |J| = |I|-1}} (\Delta_{I\twoheadrightarrow J})_* \boxtimes_{j\in J}\underline{B}_X,
\end{equation}
where $B$ is the abelian group $\op{Maps}(\mathbf S, \mathbb Z)$ for $\mathbf S$ the set of simple factors of $G$.

\subsubsection{}
We will also mention the construction of determinant line bundles on $\tn{Gr}_{G, \tn{Ran}}$. Let $\mathbf S$ denote the set of simple factors of $\widetilde G_{\tn{der}}$. Then for each $s\in\mathbf S$, the corresponding Lie algebra $\fr g_s$ can be regarded as a $G$-representation. Consequently, we may define a line bundle $\det_{\fr g_s}$ over $\tn{Gr}_{G,\tn{Ran}}$ by specifying its fiber at an $S$-point $(x^{(i)}, \cal P_G, \alpha)$ to be the relative determinant of the vector bundles $(\fr g_s)_{\cal P_G}$ and $(\fr g_s)_{\cal P_G^0}$, identified outisde $\bigcup_{i\in I}\Gamma_{x^{(i)}}$. Then $\det_{\fr g_s}$ has the canonical structure of a factorization line bundle over $\tn{Gr}_{G, \tn{Ran}}$ (c.f.~\cite[\S5.2]{gaitsgory2018parameters}). Thus we have a map:
\begin{equation}
\label{eq-det-line-bundle}
\det : \bigoplus_{s\in\mathbf S} \mathbb Z \rightarrow \mathbf{Pic}^{\tn{fact}}(\tn{Gr}_{G, \tn{Ran}}),\quad (a_s) \leadsto \bigotimes_{s\in\mathbf S}\det{}_{\fr g_s}^{a_s}.
\end{equation}

\subsection{Enhanced $\Theta$-data}
\label{sec-enh-theta-data}

\subsubsection{}
\label{sec-setup-classification} Suppose we are given the following data:
\begin{enumerate}[(a)]
	\item a smooth, connected algebraic curve $X$;
	\item a reductive group $G$ over $k$ with maximal torus $T\subset G$;
	\item a theory of gerbes $\mathbf G$ such that $A(-1)$ is a \emph{divisible} abelian group (in particular, discrete).
\end{enumerate}
Then we shall attach a strict Picard $2$-groupoid $\Theta_G(\Lambda_T; \mathbf G)$ called \emph{enhanced $\Theta$-data}. It will consist of triples $(q, \cal G^{(\lambda)}, \varepsilon)$ to be specified below.

\subsubsection{Quadratic form}
\label{sec-quad-restr}
Let $W$ denote the Weyl group of $(G,T)$. It acts on the cocharacter lattice $\Lambda_T$. Let $\cal Q(\Lambda_T; A(-1))^W$ denote the abelian group of $W$-invariant $A(-1)$-valued quadratic forms on $\Lambda_T$. Any such quadratic form gives rise to a $W$-invariant bilinear form $\kappa$ defined by:
$$
\kappa(\lambda,\mu) := q(\lambda+\mu) - q(\lambda) - q(\mu).
$$
In particular, $\kappa(\lambda, \lambda) = 2q(\lambda)$.

\smallskip

Following Gaitsgory--Lysenko \cite{gaitsgory2018parameters}, we shall specify a subgroup$$
\cal Q(\Lambda_T; A(-1))_{\tn{restr}}^W \subset \cal Q(\Lambda_T; A(-1))^W,
$$
called \emph{restricted} quadratic forms, by the property that $q\in \cal Q(\Lambda_T; A(-1))_{\tn{restr}}^W$ if:
\begin{equation}
\label{eq-restr-defn}
\kappa(\alpha, \lambda) = \langle\check{\alpha}, \lambda\rangle q(\alpha),\quad\text{for all }\alpha\in\check{\Phi}, \lambda\in\Lambda_T,
\end{equation}
where $\check{\alpha}$ denotes the root associated to $\alpha$. We note that there always holds $2\kappa(\alpha, \lambda) = 2\langle\check{\alpha}, \lambda\rangle q(\alpha)$; indeed, this is because $\kappa(\alpha, \lambda) = \kappa(-\alpha, s_{\alpha}(\lambda))$ by $W$-invariance, where $s_{\alpha}(\lambda) = \lambda - \langle\check{\alpha}, \lambda\rangle\alpha$. Analogously, if each co-root is twice a co-character (e.g.~$G=\op{GL}_2$, $\op{PGL}_2$), then \eqref{eq-restr-defn} always holds. Let $\Lambda_T^r\subset\Lambda_T$ denote the co-root lattice and $\pi_1G:=\Lambda_T/\Lambda_T^r$ be the algebraic fundamental group of $G$.

\smallskip

We note an elementary fact.

\begin{lem}
\label{lem-quad-form}
Suppose $q\in \cal Q(\Lambda_T; A(-1))^W_{\tn{restr}}$. Then there is a (non-canonical) decomposition $q = q_1 + q_2$ where:
\begin{enumerate}[(a)]
	\item $q_1$ is an $A(-1)$-linear sum of Killing forms $q_{s,\tn{Kil}}$, attached to each irreducible component $\Phi_s$ ($s\in\mathbf S$) of the coroot system $\Phi_{(G,T)}$ by the formula:
	$$
	q_{s, \tn{Kil}}(\lambda) := \frac{1}{2}\sum_{\alpha\in\Phi_s}\langle\check{\alpha},\lambda\rangle^2.
	$$
	\item $q_2$ descends to a quadratic form on $\pi_1G$.
\end{enumerate}
\end{lem}
\begin{proof}
For each $s\in\mathbf S$, let $\alpha_s$ be a short coroot of $\Phi_s$. Since $A(-1)$ is divisible, there exists some $b_s\in A(-1)$ such that $q(\alpha_s) = b_sq_{s,\op{Kil}}(\alpha_s)$. We set $q_1 := \sum_{s\in\mathbf S} b_s q_{s, \op{Kil}}$ and $q_2 := q - q_1$. Thus $q_2$ still belongs to $\cal Q(\Lambda_T; A(-1))^W_{\tn{restr}}$. The identity \eqref{eq-restr-defn} implies that the $\Lambda_T^r$ lies in the kernel of the bilinear form attached to $q_2$, so it descends to a quadratic form on $\pi_1G$.
\end{proof}

Consider the injective map:
\begin{equation}
\label{eq-tensor-to-restr}
\cal Q(\Lambda; \mathbb Z)^W \underset{\mathbb Z}{\otimes} A(-1) \hookrightarrow \cal Q(\Lambda; A(-1))^W_{\tn{restr}}.
\end{equation}

\begin{lem}
\label{lem-sc-restr}
Suppose $G_{\tn{der}}$ is simply connected. Then \eqref{eq-tensor-to-restr} is bijective.
\end{lem}
\begin{proof}
The hypothesis shows that $\pi_1G$ is torsion-free. Hence every $A(-1)$-valued quadratic form on $\pi_1G$ lives in $\cal Q(\pi_1G; \mathbb Z)\underset{\mathbb Z}{\otimes}A(-1)$.
\end{proof}

\subsubsection{Integral $\Theta$-data}
Given a lattice $\Lambda$, we let $\Theta(\Lambda; \mathbf{Pic})$ denote the strict Picard $1$-groupoid consisting of an integral quadratic form $q\in\cal Q(\Lambda; \mathbb Z)$, and a $\Lambda$-indexed system of line bundles $\cal L^{(\lambda)}$ over $X$ equipped with multiplicative structures:
\begin{equation}
\label{eq-line-bundle-mult}
c_{\lambda,\mu} : \cal L^{(\lambda)} \otimes \cal L^{(\mu)} \xrightarrow{\sim} \cal L^{(\lambda + \mu)} \otimes \omega_X^{\kappa(\lambda,\mu)},
\end{equation}
satisfying associativity and the following $\kappa$-twisted commutativity condition:
$$
(-1)^{\kappa(\lambda,\mu)} c_{\lambda,\mu}(a \otimes b) =  c_{\mu,\lambda}(b\otimes a).
$$
Objects of $\Theta(\Lambda; \mathbf{Pic})$ are called \emph{integral $\Theta$-data}.

\subsubsection{Integral enhanced $\Theta$-data} For a semisimple, simply connected group $G$ with split maximal torus $T$, we have a morphism (c.f.~\cite[\S2.4.7]{tao2019extensions}) which attaches a $\Lambda_T$-indexed system of line bundles to a $W$-invariant form:
\begin{equation}
\label{eq-reversed-map-theta}
\cal Q(\Lambda_T; \mathbb Z)^W \rightarrow \Theta(\Lambda_T; \mathbf{Pic}),\quad q\leadsto (q, \cal L^{(\lambda)}).
\end{equation}
For a reductive group $G$, we shall use construction \eqref{eq-reversed-map-theta} for the simply connected cover of its derived subgroup $\widetilde G_{\tn{der}}$ (with maximal torus $\widetilde T_{\tn{der}}$). The \emph{integral enhanced $\Theta$-data} $\Theta_G(\Lambda_T; \mathbf{Pic})$ are defined to be the strict Picard $1$-groupoid of triples $(q, \cal L^{(\lambda)}, \varepsilon)$ where:
\begin{enumerate}[(a)]
	\item $q\in \cal Q(\Lambda_T; \mathbb Z)^W$, whose bilinear form is denoted $\kappa$;
	\smallskip
	\item $\cal L^{(\lambda)}$ is a $\Lambda_T$-indexed system of line bundles, equipped with multiplicative structure \eqref{eq-line-bundle-mult} which makes $(q, \cal L^{(\lambda)})$ an object of $\Theta(\Lambda_T; \mathbf{Pic})$;
	\smallskip
	\item $\varepsilon$ is an isomorphism between the restriction of $\cal L^{(\lambda)}$ to $\Lambda_{\widetilde T_{\tn{der}}}$ and the system of line bundles attached to the restriction of $q$ to $\Lambda_{\widetilde T_{\tn{der}}}$ via \eqref{eq-reversed-map-theta}.
\end{enumerate}

\subsubsection{$\Theta$-data for $\mathbf G$}
\label{sec-gerbe-theta}
We temporarily relax the condition: $A(-1)$ is only assumed discrete in this paragraph. Given a lattice $\Lambda$, we let $\Theta(\Lambda; \mathbf G)$ denote the strict Picard $2$-groupoid consisting of a quadratic form $q\in \cal Q(\Lambda; A(-1))$, and a $\Lambda$-indexed system of gerbes $\cal G^{(\lambda)} \in \mathbf G(X)$ equipped with multiplicative structures:
\begin{equation}
\label{eq-gerbe-mult}
c_{\lambda,\mu} : \cal G^{(\lambda)} \otimes \cal G^{(\mu)} \xrightarrow{\sim} \cal G^{(\lambda + \mu)} \otimes \omega_X^{\kappa(\lambda,\mu)},
\end{equation}
together with associativity \emph{constraint} and $\kappa$-twisted commutativity \emph{constraint}, i.e., a homotopy $h_{\lambda,\mu}$ witnessing the commutative diagram:
\begin{equation}
\label{eq-gerbe-twisted-commutativity}
\xysmall{
	\cal G^{(\lambda)} \otimes \cal G^{(\mu)} \ar[d] \ar[r]^-{c_{\lambda,\mu}} & \cal G^{(\lambda + \mu)} \otimes \omega_X^{\kappa(\lambda,\mu)} \ar[d]^{(-1)^{\kappa(\lambda,\mu)}}\ar@{=>}[dl]_{h_{\lambda,\mu}} \\
	\cal G^{(\mu)} \otimes \cal G^{(\lambda)} \ar[r]^-{c_{\mu,\lambda}} & \cal G^{(\mu + \lambda)} \otimes \omega_X^{\kappa(\mu,\lambda)}
}
\end{equation}
satisfying the usual coherence conditions for every triple $\cal G^{(\lambda)}$, $\cal G^{(\mu)}$, $\cal G^{(\nu)}$, as well as an additional condition expressing that \emph{strictness} ought to be respected. Namely, for $\lambda = \mu$, as the automorphism:
$$
(-1)^{\kappa(\lambda,\lambda)} \xrightarrow{\sim} (-1)^{2q(\lambda)} \xrightarrow{\sim} ((-1)^2)^{q(\lambda)}
$$
is canonically trivialized, we require that $h_{\lambda, \lambda}$ be the identity $2$-homotopy. The strict Picard $2$-groupoid $\Theta(\Lambda; \mathbf G)$ is called \emph{$\Theta$-data for $\mathbf G$}.

\begin{rem}
\label{rem-homotopy-to-square-root}
In fact, given a commutative diagram \eqref{eq-gerbe-twisted-commutativity} for $\lambda = \mu$, the $2$-homotopy $h_{\lambda,\lambda}$ determines conversely a square root of $\kappa(\lambda,\lambda)\in A(-1)$. Indeed, $h_{\lambda,\lambda}$ defines a trivialization of $(-1)^{\kappa(\lambda,\lambda)}$ whose square is the tautological trivialization of $(-1)^{2\kappa(\lambda,\lambda)}$.

\smallskip

On the other hand, for any $a\in A(-1)$, a trivialization of $(-1)^a$ which squares to the tautological trivialization of $(-1)^{2a}$ is equivalent to the choice of a square root of $a$, since both data are torsors for the $2$-torsion subgroup of $A(-1)$ and there is an obvious map from the latter to the former (c.f.~\cite[\S4.2]{gaitsgory2018parameters}).
\end{rem}

\subsubsection{Enhanced $\Theta$-data for $\mathbf G$}
For $G$ semisimple, simply connected, the morphism \eqref{eq-reversed-map-theta} coupled with the divisor class map for $\mathbf G$ gives rise to a morphism:
\begin{equation}
\label{eq-gerbe-reversed-map}
\cal Q(\Lambda_T; \mathbb Z)^W \underset{\mathbb Z}{\otimes} A(-1) \rightarrow \Theta(\Lambda_T; \mathbf G),\quad (q, a)\leadsto (q, (\cal L^{(\lambda)})^a).
\end{equation}
We reinstall the assumption that $A(-1)$ be divisible. For a reductive group $G$, define the \emph{enhanced $\Theta$-data $\Theta_G(\Lambda_T; \mathbf G)$ for $\mathbf G$} as the strict Picard $2$-groupoid of triples $(q, \cal G^{(\lambda)}, \varepsilon)$ where:
\begin{enumerate}[(a)]
	\item $q\in\cal Q(\Lambda_T; A(-1))^W_{\tn{restr}}$ is a \emph{restricted} quadratic form in the sense of \S\ref{sec-quad-restr}, whose bilinear form is denoted $\kappa$;
	\smallskip
	\item $\cal G^{(\lambda)}$ is a $\Lambda_T$-indexed system in $\mathbf G(X)$, equipped with multiplicative structure \eqref{eq-gerbe-mult}, associativity constraint, and $\kappa$-twisted commutativity constraint, making $(q, \cal G^{(\lambda)})$ an object of $\Theta(\Lambda_T; \mathbf G)$;
	\smallskip
	\item $\varepsilon$ is an isomorphism between the restriction of $\cal G^{(\lambda)}$ to $\Lambda_{\widetilde T_{\tn{der}}}$ and the system of gerbes $\cal G_q^{(\lambda)}$ attahced to the restriction of $q$ to $\Lambda_{\widetilde T_{\tn{der}}}$ via \eqref{eq-gerbe-reversed-map}, compatible with the associativity and $\kappa$-twisted commutativity constraints.
\end{enumerate}

\noindent
Therefore, we have a fiber sequence of strict Picard $2$-groupoids:
\begin{equation}
\label{eq-theta-fiber-seq}
\mathbf{Hom}(\pi_1G, \mathbf G(X)) \rightarrow \Theta_G(\Lambda_T; \mathbf G) \rightarrow \cal Q(\Lambda_T; A(-1))^W_{\tn{restr}},
\end{equation}
where $\mathbf{Hom}(\pi_1G, \mathbf G(X))$ denotes the groupoid of morphisms $\pi_1G\rightarrow\mathbf G(X)$ as \emph{strict} Picard $2$-groupoids.

\subsubsection{$\omega$-shift}
\label{sec-omega-shift}
We note a variant in the definition of enhanced $\Theta$-data where we incorporate shifts by a power of $\omega_X$. Define the \emph{shifted} enhanced $\Theta$-data $\Theta_G^+(\Lambda_T; \mathbf G)$ for $\mathbf G$ to be the strict Picard $2$-groupoid of triples $(q, \cal G^{(\lambda)}, \varepsilon)$ where:
\begin{enumerate}[(a)]
	\item $q\in \cal Q(\Lambda_T; A(-1))^W_{\tn{restr}}$ is as before;
	\smallskip
	\item $\cal G^{(\lambda)}$ is a $\Lambda_T$-indexed system in $\mathbf G(X)$, equipped with multiplicative structures:
	$$
	c_{\lambda,\mu}^+ : \cal G^{(\lambda)} \otimes \cal G^{(\mu)} \xrightarrow{\sim} \cal G^{(\lambda + \mu)},
	$$
	together with associativity constraint and $\kappa$-twisted commutativity constraint:
	$$
\xysmall{
	\cal G^{(\lambda)} \otimes \cal G^{(\mu)} \ar[d] \ar[r]^-{c^+_{\lambda,\mu}} & \cal G^{(\lambda + \mu)} \ar[d]^{(-1)^{\kappa(\lambda,\mu)}}\ar@{=>}[dl]_{h^+_{\lambda,\mu}} \\
	\cal G^{(\mu)} \otimes \cal G^{(\lambda)} \ar[r]^-{c^+_{\mu,\lambda}} & \cal G^{(\mu + \lambda)}
}
$$
	satisfying coherence conditions for every triple $\cal G^{(\lambda)}$, $\cal G^{(\mu)}$, $\cal G^{(\nu)}$ and respects strictness.
	\smallskip
	\item $\varepsilon$ is an isomorphism between the restriction of $\cal G^{(\lambda)}$ to $\Lambda_{\widetilde T_{\tn{der}}}$ and the system of gerbes $\cal G_q^{(\lambda)} \otimes \omega_X^{q(\lambda)}$, where $\cal G_q^{(\lambda)}$ is the system attahced to the restriction of $q$ to $\Lambda_{\widetilde T_{\tn{der}}}$ via \eqref{eq-gerbe-reversed-map}, compatible with the associativity and $\kappa$-twisted commutativity constraints.
\end{enumerate}

\noindent
Clearly, there is an equivalence between the two kinds of enhanced $\Theta$-data:
$$
\Theta_G(\Lambda_T; \mathbf G) \xrightarrow{\sim} \Theta_G^+(\Lambda_T; \mathbf G),\quad (q, \cal G^{(\lambda)}, \varepsilon) \leadsto (q, \cal G^{(\lambda)}\otimes\omega_X^{q(\lambda)}, \varepsilon).
$$

\subsection{Classification: statements}

\subsubsection{} We continue to fix $X$, $G$ as in \S\ref{sec-setup-classification}. The basis of our classification theorem is the equivalence between factorization line bundles over $\tn{Gr}_{G, \tn{Ran}}$ and integral enhanced $\Theta$-data, established in \cite{tao2019extensions}. We let $N_G\ge 1$ be the integer attached to $G$ as in \cite[\S0.1.8]{gaitsgory2018parameterization}.\footnote{The Lemma will only be used when $\tn{char}(k) = 0$, where the hypothesis $\tn{char}(k)\nmid N_G$ is trivially satisfied.}

\begin{lem}
\label{lem-pic-classification}
There is a canonical functor:
$$
\Psi_{\mathbf{Pic}} : \mathbf{Pic}^{\tn{fact}}(\tn{Gr}_{G, \tn{Ran}}) \rightarrow \Theta_G(\Lambda_T; \mathbf{Pic})
$$
with the following properties:
\begin{enumerate}[(a)]
	\item $\Psi_{\mathbf{Pic}}$ is an equivalence for $G$ a torus or a semisimple, simply connected group;
	\item For any reductive group $G$ with $\tn{char}(k)\nmid N_G$, the functor $\Psi_{\mathbf{Pic}}$ is an equivalence.
\end{enumerate}
\end{lem}
\noindent
We shall refer to $\Psi_{\mathbf{Pic}}$ as the classification functor for factorization line bundles on $\tn{Gr}_{G,\tn{Ran}}$. Sometimes we denote it by $\Psi_{\mathbf{Pic}, G}$ to emphasize the group $G$.
\begin{proof}
The functor $\Psi_{\mathbf{Pic}, T}$ for the torus $T$ is constructed and proved to be an equivalence in \cite[\S1]{tao2019extensions}. For $G_{\tn{sc}}$ semisimple and simply connected, with maximal torus $T_{\tn{sc}}$, $\Psi_{\mathbf{Pic}, G_{\tn{sc}}}$ is constructed and proved to be an equivalence in \cite[Proposition 2.5]{tao2019extensions}. Since the composition:
$$
\cal Q(\Lambda_{T_{\tn{sc}}}; \mathbb Z)^W \xrightarrow{\Psi_{\mathbf{Pic}, G_{\tn{sc}}}^{-1}} \mathbf{Pic}^{\tn{fact}}(\tn{Gr}_{G_{\tn{sc}}}) \rightarrow \mathbf{Pic}^{\tn{fact}}(\tn{Gr}_{T_{\tn{sc}}}) \xrightarrow{\Psi_{\mathbf{Pic}, {T_{\tn{sc}}}}} \Theta(\Lambda_{T_{\tn{sc}}}; \mathbf{Pic})
$$
identifies with \eqref{eq-reversed-map-theta}, one constructs the functor $\Psi_{\mathbf{Pic}}$ for any reductive group $G$.\footnote{The definition of $\Psi_{\mathbf{Pic}}$ will be explained below in \S\ref{sec-nature-of-classification-functor}.} Finally, statement (b) is \cite[Theorem 3.1]{tao2019extensions}.
\end{proof}

\subsubsection{} Let us now also fix a topology $t$ on $\mathbf{Sch}^{\tn{ft}}_{/k}$ as in \S\ref{sec-setup-topology}. The following classification statement will be proved in \S\ref{sec-classification-proof}.

\begin{thm}
\label{thm-gerbe-classification}
Let $\mathbf G$ be a motivic $t$-theory of gerbes whose coefficient $A(-1)$ is a divisible abelian group. Then there is a canonical equivalence between strict Picard $2$-groupoids:
$$
\Psi_{\mathbf G} : \mathbf G^{\tn{fact}}(\tn{Gr}_{G,\tn{Ran}}) \xrightarrow{\sim} \Theta_G(\Lambda_T; \mathbf G),
$$
which makes the following diagram commute:
$$
\xysmall{
	\mathbf{Pic}^{\tn{fact}}(\tn{Gr}_{G, \tn{Ran}}) \underset{\mathbb Z}{\otimes} A(-1) \ar[d]^{\Psi_{\mathbf {Pic}}}\ar[r] & \mathbf G^{\tn{fact}}(\tn{Gr}_{G, \tn{Ran}}) \ar[d]^{\Psi_{\mathbf G}} \\
	\Theta_G(\Lambda_T; \mathbf{Pic})\underset{\mathbb Z}{\otimes} A(-1) \ar[r] & \Theta_G(\Lambda_T; \mathbf G)
}
$$
\end{thm}

\noindent
We call $\Psi_{\mathbf G}$ the \emph{classification functor} for factorization gerbes on $\tn{Gr}_{G,\tn{Ran}}$. As before, we denote it by $\Psi_{\mathbf G, G}$ sometimes to emphasize the role of the reductive group $G$.

\subsubsection{}
\label{sec-nature-of-classification-functor}
Let us first clarify the nature of the functor $\Psi_{\mathbf G}$. In fact, we shall consider an arbitrary theory of gerbes $\mathbf G$ satisfying property (RP1), so it includes $\mathbf G=\mathbf{Pic}$ as a special case. The upshot will be that as long as $\Psi_{\mathbf G, G_{\tn{sc}}}$ is an equivalence for semisimple, simply connected $G_{\tn{sc}}$, the functor $\Psi_{\mathbf G, G}$ can be defined for general $G$.

\begin{rem}
Since $\mathbf{Pic}$ is not motivic in the sense of \S\ref{sec-motivic-gerbes}, Theorem \ref{thm-gerbe-classification} does not imply Lemma \ref{lem-pic-classification}. The proof of Lemma \ref{lem-pic-classification} in \cite{tao2019extensions} uses nontrivial input from K-theory.
\end{rem}

\subsubsection{} For a torus $T$, we introduce an auxiliary object $\op{Gr}_{T, \tn{comb}}$. As a prestack, it is defined as the colimit:
$$
\op{Gr}_{T,\op{comb}} := \underset{(I, \lambda^{(I)})}{\op{colim}} X^I,
$$
where the index category consists of pairs $(I,\lambda^{(I)})$ for $I$ a finite set, $\lambda^{(I)}=(\lambda^{(i)})_{i\in I}$ an $I$-family of elements in $\Lambda_T$, and a morphism $(I,\lambda^{(I)}) \rightarrow (J, \lambda^{(J)})$ in this category consists of a surjection $\varphi : I\rightarrow J$ such that $\lambda^{(j)} = \sum_{i\in\varphi^{-1}(j)} \lambda^{(i)}$. Then $\op{Gr}_{T, \tn{comb}}$ has the structure of a factorization prestack over $X$. It is equipped with a map:
\begin{equation}
\label{eq-grT-comb-to-ran}
\tn{Gr}_{T, \tn{comb}} \rightarrow \op{Gr}_{T, \tn{Ran}},\quad x^{(i)} \leadsto (x^{(i)}, \otimes_{i\in I}\cal O(\lambda^i\Gamma_{x^{(i)}}), \alpha),
\end{equation}
where $\alpha$ is the tautological trivialization. The closed subscheme $X^I\hookrightarrow\tn{Gr}_{T, \tn{comb}}$ corresponding to $(I, \lambda^{(I)})$ will be denoted by $X^{\lambda^I}$.

\begin{lem}
\label{lem-grT-comb-classification}
Suppose $\mathbf G$ is a theory of gerbes satisfying property $\tn{(RP1)}$, then we have an equivalence of strict Picard $2$-groupoids:
$$
\mathbf G^{\tn{fact}}(\tn{Gr}_{T, \tn{comb}}) \xrightarrow{\sim} \Theta(\Lambda_T; \mathbf G).
$$
\end{lem}
\begin{proof}
The $\Lambda_T$-indexed family of gerbes $\cal G^{(\lambda)}$ will be the restriction of $\cal G\in\mathbf G^{\tn{fact}}(\tn{Gr}_{T, \tn{comb}})$ to the closed subscheme $X^{(\lambda)}$ of $\tn{Gr}_{T, \tn{comb}}$.

\smallskip

We construct a bilinear form $\kappa : \Lambda_T \underset{\mathbb Z}{\otimes} \Lambda_T \rightarrow A(-1)$ as follows. Given $\lambda,\mu\in\Lambda_T$, we consider the subscheme $X^{(\lambda,\mu)}$ of $\tn{Gr}_{T, \tn{comb}}$, and denote by $\cal G^{(\lambda,\mu)}$ the restriction of $\cal G$. Then factorization isomorphism together with (RP1) shows that we have an isomorphism:
\begin{equation}
\label{eq-purity-isom-lambda-mu}
\cal G^{(\lambda)} \boxtimes \cal G^{(\mu)} \xrightarrow{\sim} \cal G^{(\lambda, \mu)} \otimes \cal O(-\Delta)^{\kappa(\lambda,\mu)}
\end{equation}
for a unique element $\kappa(\lambda,\mu)\in A(-1)$. The compatibility between factorization and the swapping map $X^{(\lambda,\mu)}\xrightarrow{\sim} X^{(\mu,\lambda)}$ shows that \eqref{eq-purity-isom-lambda-mu} is $\Sigma_2$-equivariant, in the sense that the following diagram commutes.
\begin{equation}
\label{eq-fact-compatible-with-swap}
\xysmall{
	\cal G^{(\lambda)} \boxtimes \cal G^{(\mu)} \ar[r]^-{\sim}\ar[d] & \cal G^{(\lambda, \mu)} \otimes \cal O(-\Delta)^{\kappa(\lambda, \mu)} \ar[d] \\
	\sigma^*(\cal G^{(\mu)} \boxtimes \cal G^{(\lambda)}) \ar[r]^-{\sim} & \sigma^*\cal G^{(\mu,\lambda)} \otimes \sigma^*\cal O(-\Delta)^{\kappa(\mu,\lambda)}
}
\end{equation}
This already implies that $\kappa(\lambda,\mu) = \kappa(\mu,\lambda)$. Considering the restriction of $\cal G$ to $X^{(\lambda,\mu,\nu)}$ for a triple $(\lambda,\mu,\nu)$ then establishes the bilinearity of $\kappa$.

\smallskip

Finally, restriction of \eqref{eq-fact-compatible-with-swap} to the diagonal produces a commutative diagram:
$$
\xysmall{
	\cal G^{(\lambda)} \otimes \cal G^{(\mu)} \ar[r]^-{\sim}\ar[d] & \cal G^{(\lambda + \mu)} \otimes \omega_X^{\kappa(\lambda, \mu)} \ar[d]^{(-1)^{\kappa(\lambda,\mu)}} \ar@{=>}[dl]_{h_{\lambda,\mu}} \\
	\cal G^{(\mu)} \otimes \cal G^{(\lambda)} \ar[r]^-{\sim} & \cal G^{(\mu + \lambda)} \otimes \omega_X^{\kappa(\mu,\lambda)}
}
$$
We note that for $\lambda = \mu$, the $2$-homotopy $h_{\lambda,\lambda}$ amounts to a trivialization of $(-1)^{\kappa(\lambda, \lambda)}$ whose square identifies with the tautological trivialization of $(-1)^{2\kappa(\lambda,\lambda)}$. Hence $h_{\lambda,\lambda}$ defines an element $q(\lambda)\in A(-1)$ with $2q(\lambda) = \kappa(\lambda, \lambda)$ (see Remark \ref{rem-homotopy-to-square-root}). With respect to the resulting trivialization of $(-1)^{\kappa(\lambda,\lambda)}$ afforded by $q(\lambda)$, we see that $h_{\lambda,\lambda}$ is the identity $2$-homotopy.

\smallskip

This completes the definition of a functor from $\mathbf G^{\tn{fact}}(\tn{Gr}_{T, \tn{comb}})$ to $\Theta(\Lambda_T; \mathbf G)$. Checking that it is an equivalence is straightforward, hence omitted.
\end{proof}

\subsubsection{}
\label{sec-classification-functor-strategy} Therefore, for $\mathbf G$ any theory of gerbes satisfying (RP1), pulling back along $\tn{Gr}_{T, \tn{Ran}} \rightarrow \tn{Gr}_{G, \tn{Ran}}$ and then along \eqref{eq-grT-comb-to-ran} defines a map from:
\begin{equation}
\label{eq-grG-to-theta}
\mathbf G^{\tn{fact}}(\tn{Gr}_G) \rightarrow \Theta(\Lambda_T; \mathbf G).
\end{equation}
Below, we summarize the definition of the classification functor $\Psi_{\mathbf G}$ of Theorem \ref{thm-gerbe-classification}, relying on results to be established in \S\ref{sec-classification-proof}. The purpose of doing so now is to make various compatibility statements apparent, so we may deduce corollaries from Theorem \ref{thm-gerbe-classification}.
\begin{enumerate}[(a)]
	\item For $G = T$ a torus, $\Psi_{\mathbf G}$ is precisely \eqref{eq-grG-to-theta};
	\smallskip
	\item For $G = G_{\tn{sc}}$ a semisimple, simply connected group, $\Psi_{\mathbf G}$ is the composition of \eqref{eq-grG-to-theta} with the forgetful functor to $\cal Q(\Lambda_T; A(-1))$;
	\smallskip
	\item For $G$ a reductive group, in order to construct $\Psi_{\mathbf G}$ we will need the \emph{conclusion} of Theorem \ref{thm-gerbe-classification} to hold for the case (b), i.e., for $\widetilde G_{\tn{der}}$ the functor above defines an equivalence:
	\begin{equation}
	\label{eq-step-2-to-3}
	\Psi_{\mathbf G, \widetilde G_{\tn{der}}} : \mathbf{G}^{\tn{fact}}(\tn{Gr}_{\widetilde G_{\tn{der}}, \tn{Ran}}) \xrightarrow{\sim} \cal Q(\Lambda_T; \mathbb Z)^W\underset{\mathbb Z}{\otimes}A(-1),
	\end{equation}
	which is furthermore compatible with $\Psi_{\mathbf{Pic}, \widetilde G_{\tn{der}}}$. Then the functor
	$$
	\Psi_{\mathbf G} : \cal G \leadsto (q, \cal G^{(\lambda)}, \varepsilon)
	$$
	is specified by the application of $\Psi_{\mathbf G}$ to the restriction of $\cal G$ to $\tn{Gr}_{T, \tn{Ran}}$, obtaining $(q, \cal G^{(\lambda)})\in \Theta(\Lambda_T; \mathbf G)$, and then using the inverse of \eqref{eq-step-2-to-3} to obtain the identification $\varepsilon$.
\end{enumerate}

\subsection{Application to quantum parameters}
\label{sec-application-quantum}

\subsubsection{} Suppose $k=\bar k$ with $\tn{char}(k) = 0$. Thus Theorem \ref{thm-gerbe-classification} classifies factorization tame gerbes on $\tn{Gr}_{G, \tn{Ran}}$ by its enhanced $\Theta$-data, as $\mathring{\mathbf{Ge}}$ is a motivic $\eh$-theory of gerbes (\S\ref{sec-tame-gerbe-motivic}). We shall use this result to obtain a classification of factorization tame twistings and explain their relations to quantum parameters.

\subsubsection{}
\label{sec-quantum-par}
Let $\mathring{\tn{Par}}_G$ denote the $k$-linear groupoid consisting of:
\begin{enumerate}[(a)]
	\item a $W$-invariant bilinear form $\kappa : \fr t\underset{k}{\otimes}\fr t \rightarrow k$;
	\item an extension $\mathring E$ of $\underline{\fr z}$ by $\mathring{\Omega}_X^1$ as Zariski sheaves of $k$-vector spaces on $X$.
\end{enumerate}
\noindent
Let $\tn{Par}_G$ denote the analogously defined $k$-linear groupoid where we replace $\mathring E$ by an extension $E$ of $\fr z\otimes\cal O_X$ by $\omega_X$ as coherent sheaves. The $k$-linear stack associated to $\tn{Par}_G$ is the non-compact space of quantum parameters studied in \cite{zhao2017quantum}.

\subsubsection{} The following Theorem summarizes the relationship between factorization tame twistings and quantum parameters.

\begin{thm}
\label{thm-twisting-classification}
There are three canonical equivalences between $k$-linear groupoids:
$$
\xysmall{
\mathring{\mathbf{Tw}}^{\tn{fact}}(\tn{Gr}_{G, \tn{Ran}}) \ar[r]^-{\Psi_{\mathring{\mathbf{Tw}}}}_-{\sim} & \Theta_G(\Lambda_T; \mathring{\mathbf{Tw}}) \ar[d]^{\otimes\omega_X^{q(\lambda)}}_{\cong} \\
& \Theta_G^+(\Lambda_T; \mathring{\mathbf{Tw}}) \ar[r]^-{\sim} & \mathring{\tn{Par}}_G \ar[r]^-{\fr j} & \tn{Par}_G,
}
$$
and the last functor $\fr j$ is an equivalence if and only if $X$ is proper.
\end{thm}

\noindent
The proof of Theorem \ref{thm-twisting-classification} occupies the remainder of this subsection. The functor $\Psi_{\mathring{\mathbf{Tw}}}$ will be built according to the paradigm of \S\ref{sec-classification-functor-strategy}. Both the construction and the proof that it is an equivalence require Lemma \ref{lem-pic-classification} and Theorem \ref{thm-gerbe-classification}, as well as the $k$-linear structure on $\mathring{\mathbf{Tw}}$.

\subsubsection{} First note that the equivalence between $\Theta_G(\Lambda_T; \mathring{\mathbf{Tw}})$ and $\Theta_G^+(\Lambda_T; \mathring{\mathbf{Tw}})$ is already noted in \S\ref{sec-omega-shift}, the functor being given by a $\omega_X$-shift.

\smallskip

To show the equivalence between the latter with $\mathring{\tn{Par}}_G$, we observe that $\mathring{\mathbf{Tw}}(X)$ is $k$-linear and $\fr z\cong \pi_1G \underset{\mathbb Z}{\otimes} k$, so we may rewrite the fiber sequence \eqref{eq-theta-fiber-seq} as follows:
\begin{equation}
\label{eq-theta-fiber-seq-twisting}
\mathbf{Hom}(\fr z, \mathring{\mathbf{Tw}}(X)) \rightarrow \Theta_G^+(\Lambda_T; \mathring{\mathbf{Tw}}) \rightarrow \cal Q(\Lambda_T; k)^W.
\end{equation}
On the other hand, the automorphism $(-1)^{\kappa(\lambda,\mu)}$ on $\mathring{\mathbf{Tw}}(X)$ is trivial since $d\log$ annihilates all constant sections. Thus an element of $\Theta_G^+(\Lambda_T; \mathring{\mathbf{Tw}})$ consists of the data of $q$ together with a \emph{commutative} multiplicative system $\cal T^{(\lambda)}$, i.e., a $k$-linear morphism $\fr t\rightarrow \mathring{\mathbf{Tw}}(X)$, whose restriction to $\fr t_{\tn{der}}$ is determined by $q$. In particular, \eqref{eq-theta-fiber-seq-twisting} canonically splits. It remains to observe:
\begin{enumerate}[(a)]
	\item $\cal Q(\Lambda_T; k)^W$ identifies with the space of $W$-invariant bilinear forms on $\fr t$;
	\smallskip
	\item $\mathbf{Hom}(\fr z, \mathring{\mathbf{Tw}}(X))$ is the space of $k$-linear maps $\fr z \rightarrow \op R\Gamma_{\et}(X; \mathring{\Omega}^1[1])$ (Lemma \ref{lem-twisting-coh-interpretation}). On the other hand, we have $\op R\Gamma_{\et}(X; \mathring{\Omega}^1[1]) \xrightarrow{\sim} \op R\Gamma_{\tn{Zar}}(X; \mathring{\Omega}^1[1])$ (Theorem \ref{thm-coh-properties}), and we have:
	$$
	\op R\Gamma_{\tn{Zar}}(X; \mathring{\Omega}^1[1]) \xrightarrow{\sim} \op R\tn{Hom}(\underline{\fr z}, \mathring{\Omega}_X[1]),
	$$
	by adjunction of Zariski sheaves.
\end{enumerate}

The functor $\fr j$ is induced from the tautological inclusion $\mathring{\Omega}_X^1\hookrightarrow\omega_X$ of Zariski sheaves. Our assertion follows directly from the comparison of cohomology groups:
$$
\op R\Gamma_{\tn{Zar}}(X; \mathring{\Omega}^1) \rightarrow \op R\Gamma(X; \omega),
$$
observed in \S\ref{sec-curve-coh} using the Gersten resolution of $\mathring{\Omega}_X^1$. The only remaining part of Theorem \ref{thm-twisting-classification} is to produce a canonical equivalence:
$$
\Psi_{\mathring{\mathbf{Tw}}} : \mathring{\mathbf{Tw}}^{\tn{fact}}(\tn{Gr}_{G, \tn{Ran}}) \xrightarrow{\sim} \Theta_G(\Lambda_T; \mathring{\mathbf{Tw}}).
$$

\subsubsection{Tori}
We first define $\Psi_{\mathring{\mathbf{Tw}}, T}$ for a torus $T$ by the functor \eqref{eq-grG-to-theta}. Thus we have a commutative diagram:
$$
\xysmall{
	\mathbf{Pic}^{\tn{fact}}(\op{Gr}_{T, \tn{Ran}}) \ar[r]\ar[d]_{\cong}^{\Psi_{\mathbf{Pic}, T}} & \mathring{\mathbf{Tw}}^{\tn{fact}}(\op{Gr}_{T,\tn{Ran}}) \ar[r]\ar[d]^{\Psi_{\mathring{\mathbf{Tw}}, T}} & \mathring{\mathbf{Ge}}(\tn{Gr}_{T, \tn{Ran}}) \ar[d]_{\cong}^{\Psi_{\mathring{\mathbf{Ge}}, T}} \\
	\Theta(\Lambda_T; \mathbf{Pic}) \ar[r] & \Theta(\Lambda_T; \mathring{\mathbf{Tw}}) \ar[r] & \Theta(\Lambda_T; \mathring{\mathbf Ge})
}
$$
where the rows are fiber sequences of strict Picard $2$-groupoids and $\Psi_{\mathbf{Pic}, T}$ and $\Psi_{\mathring{\mathbf{Ge}}, T}$ are both equivalences (Lemma \ref{lem-pic-classification} and Theorem \ref{thm-gerbe-classification}). In order to show that $\Psi_{\mathring{\mathbf{Tw}}, T}$ is also an equivalence, it remains to prove that it is essentially surjective. By the calculation of $\mathring{\mathbf{Tw}}(X)$ for $X$ a smooth curve, we see that the divisor class map:
$$
\mathbf{Pic}(X) \underset{\mathbb Z}{\otimes}k \rightarrow \mathring{\mathbf{Tw}}(X),\quad (\cal L, a) \leadsto \cal L^a
$$
is essentially surjective; indeed, this is clear for $X$ affine, and for $X$ proper, $\mathring{\mathbf{Tw}}(X)$ is $1$-dimensional and is spanned by the image of any line bundle of nonzero degree.

\smallskip

Let $\mathbf{Pic}_{q=0}^{\tn{fact}}(\op{Gr}_{T, \tn{Ran}})$ denote the subgroupoid of $\mathbf{Pic}^{\tn{fact}}(\tn{Gr}_{T, \tn{Ran}})$ consisting of factorization line bundles whose associated quadratic form vanishes. From the commutative diagram:
$$
\xysmall{
	\mathbf{Pic}^{\tn{fact}}_{q=0}(\op{Gr}_{T, \tn{Ran}}) \underset{\mathbb Z}{\otimes} k \ar[r]\ar[d]^{\cong} & \mathring{\mathbf{Tw}}^{\tn{fact}}(\op{Gr}_{T,\tn{Ran}}) \ar[d] \\
	\mathbf{Hom}(\Lambda_T, \mathbf{Pic}(X)) \underset{\mathbb Z}{\otimes} k \ar@{->>}[r] & \mathbf{Hom}(\Lambda_T, \mathring{\mathbf{Tw}}(X)),
}
$$
we see that objects of the full subgroupoid $\mathbf{Hom}(\Lambda_T, \mathring{\mathbf{Tw}}(X))$ inside $\Theta(\Lambda_T; \mathring{\mathbf{Tw}})$ admit lifts. It thus remains to show that the composition of $\Psi_{\mathring{\mathbf{Tw}}, T}$ with the forgetful functor to $\cal Q(\Lambda_T; k)$ is surjective.

\smallskip

Now, every $q\in \cal Q(\Lambda_T; k)$ is a $k$-linear combination of integral forms $q_i \in \cal Q(\Lambda_T; \mathbb Z)$. Scaling allows us to assume that each $q_i$ is valued in $2\mathbb Z$. Since $\Theta(\Lambda_T; \mathbf{Pic}) \rightarrow \cal Q(\Lambda_T; \mathbb Z)$ surjects onto even-valued forms, we find that the bottom arrow in the following diagram is surjective:
$$
\xysmall{
	\mathbf{Pic}^{\tn{fact}}(\op{Gr}_{T, \tn{Ran}}) \underset{\mathbb Z}{\otimes} k \ar[r]\ar[d]^{\cong} & \mathring{\mathbf{Tw}}^{\tn{fact}}(\op{Gr}_{T,\tn{Ran}}) \ar[d] \\
	\Theta(\Lambda_T; \mathbf{Pic}) \underset{\mathbb Z}{\otimes} k \ar@{->>}[r] & \cal Q(\Lambda_T; k).
}
$$
This concludes the proof that $\Psi_{\mathring{\mathbf{Tw}}, T}$ is essentially surjective, hence an equivalence.

\subsubsection{Simply connected groups} We now turn to the case of a semisimple, simply connected group $G_{\tn{sc}}$. We note that the image of $\Psi_{\mathring{\mathbf{Tw}}, G_{\tn{sc}}}$ in $\cal Q(\Lambda_{T_{\tn{sc}}}; k)$ is $W$-invariant. Indeed, by the compatibility between $\Psi_{\mathring{\mathbf{Tw}}, G_{\tn{sc}}}$ and $\Psi_{\mathring{\mathbf{Ge}}, G_{\tn{sc}}}$, we see that any form $q$ in the image is $W$-invariant modulo $\mathbb Z$. On the other hand, if $q$ belongs to the image, so does $c\cdot q$ for all $c\in k^{\times}$, so $q$ must itself be $W$-invariant.

\smallskip

Therefore, we again have a commutative diagram of fiber sequences:
$$
\xysmall{
	\mathbf{Pic}^{\tn{fact}}(\op{Gr}_{G_{\tn{sc}}, \tn{Ran}}) \ar[r]\ar[d]_{\cong}^{\Psi_{\mathbf{Pic}, G_{\tn{sc}}}} & \mathring{\mathbf{Tw}}{}^{\tn{fact}}(\op{Gr}_{G_{\tn{sc}},\tn{Ran}}) \ar[r]\ar[d]^{\Psi_{\mathring{\mathbf{Tw}}, G_{\tn{sc}}}} & \mathring{\mathbf{Ge}}{}^{\tn{fact}}(\tn{Gr}_{G_{\tn{sc}}, \tn{Ran}}) \ar[d]^{\Psi_{\mathring{\mathbf{Ge}}, G_{\tn{sc}}}}_{\cong} \\
	\cal Q(\Lambda_{T_{\tn{sc}}}; \mathbb Z)^W \ar[r] & \cal Q(\Lambda_{T_{\tn{sc}}}; k)^W \ar[r] & \cal Q(\Lambda_{T_{\tn{sc}}}; k/\mathbb Z)^W_{\tn{restr}}
}
$$
We are done because $\Psi_{\mathbf{Pic}, G_{\tn{sc}}}$ and $\Psi_{\mathring{\mathbf{Ge}}, G_{\tn{sc}}}$ are equivalences and $\cal Q(\Lambda_{T_{\tn{sc}}}; \mathbb Z)^W \underset{\mathbb Z}{\otimes} k$ surjects (in fact, isomorphes) onto $\cal Q(\Lambda_{T_{\tn{sc}}}; k)^W$.

\subsubsection{General case} The paradigm of \S\ref{sec-classification-functor-strategy} now implies that a functor $\Psi_{\mathring{\mathbf{Tw}}}$ exists for any reductive group $G$. An analogous argument reduces the problem to showing that $\Psi_{\mathring{\mathbf{Tw}}}$ is essentially surjective. Recall that every $q\in \cal Q(\Lambda_T; k)^W$ splits into the sum of $q_1 = \sum_{s\in\mathbf S} b_sq_{s,\tn{Kil}}$ and a form $q_2$ induced from $\pi_1G$ (Lemma \ref{lem-quad-form}). We first claim that the composition:
$$
\mathring{\mathbf{Tw}}^{\tn{fact}}(\op{Gr}_{G, \tn{Ran}}) \xrightarrow{\Psi_{\mathring{\mathbf{Tw}}}} \Theta_G(\Lambda_T; \mathring{\mathbf{Tw}}) \rightarrow \cal Q(\Lambda_T; k)^W
$$
surjects onto the span of Killing forms. Indeed, this is because the determinant line bundles construction \eqref{eq-det-line-bundle} gives a section:
$$
\xysmall{
	\mathbf{Pic}^{\tn{fact}}(\op{Gr}_{G, \tn{Ran}})\underset{\mathbb Z}{\otimes} k \ar[r] & \mathring{\mathbf{Tw}}^{\tn{fact}}(\op{Gr}_{G,\tn{Ran}}) \ar[d] \\
	\bigoplus_{s\in\mathbf S} k \ar[r]\ar[u]^{\det} & \cal Q(\Lambda_T; k)^W
}
$$

\smallskip

Therefore, it remains to show that $\Psi_{\mathring{\mathbf{Tw}}}$ surjects onto the full subgroupoid of $\Theta_G(\Lambda_T; \mathring{\mathbf{Tw}})$ where the associated quadratic form descends to $\pi_1G$. This is in turn the space of quadratic forms on the lattice of $Z_G^{\circ}$. Thus the problem reduces to showing that:
\begin{equation}
\label{eq-twisting-pullback-to-center}
\mathring{\mathbf{Tw}}^{\tn{fact}}(\op{Gr}_{G, \tn{Ran}}) \rightarrow \mathring{\mathbf{Tw}}^{\tn{fact}}(\op{Gr}_{Z_G^{\circ}, \tn{Ran}}) \xrightarrow{\sim} \Theta(\Lambda_{Z_G}; \mathring{\mathbf{Tw}})
\end{equation}
is essentially surjective. Let $T_1 := G/G_{\tn{der}}$. Then $Z_G^{\circ} \rightarrow T_1$ is an isogeny of tori, so we have the following equivalence by the $k$-linear structure on tame twistings:
$$
\xysmall{
	\Theta(\Lambda_{Z_G}; \mathring{\mathbf{Tw}}) &
	\op{Gr}_{Z_G, \tn{Ran}} \ar[dr]\ar[dd] & \\
	& & \op{Gr}_{G, \tn{Ran}} \ar[dl] \\
	\Theta(\Lambda_{T_1}; \mathring{\mathbf{Tw}})\ar[uu]^{\cong} & \op{Gr}_{T_1, \tn{Ran}}
}
$$
This provides a splitting of \eqref{eq-twisting-pullback-to-center}. Hence $\Psi_{\mathring{\mathbf{Tw}}}$ is essentially surjective. \qed(Theorem \ref{thm-twisting-classification})

\subsection{Relation to Brylinski--Deligne data}

\subsubsection{} We explain how quantum parameters are related to central extensions by $\mathbf K_2$ considered by Brylinski--Deligne \cite{brylinski2001central}. Let $\mathbf K_2$ denote the Zariski sheafification of the second algebraic K-group, regarded as a sheaf on $\mathbf{Sch}_{/X}$. On the other hand, the reductive group $G$ also defines a Zariski sheaf on $\mathbf{Sch}_{/X}$. By a \emph{Brylinsk--Deligne datum}, we shall mean a central extension:
\begin{equation}
\label{eq-bd-data}
1 \rightarrow \mathbf K_2 \rightarrow \mathbf E \rightarrow G \rightarrow 1.
\end{equation}
Brylinski--Deligne data form a strict Picard groupoid, to be denoted by $\mathbf{CExt}(G, \mathbf K_2)$.

\subsubsection{} We reinstall the assumption $k=\bar k$ and $\tn{char}(k) = 0$. For a Zariski sheaf $\mathbf F$ of groups on $\mathbf{Sch}_{/X}$, denote by $\mathbf F^{\varepsilon}$ the presheaf which sends $S$ to $\mathbf F(S[\varepsilon])$ where $S[\varepsilon] := S\times \Spec(k[\varepsilon]/\varepsilon^2)$. Then $\mathbf F^{\varepsilon}$ is again a Zariski sheaf and is equipped with a tautological map to $\mathbf F$. The \emph{derivative} of $D\mathbf F$ is the kernel of $\mathbf F^{\varepsilon} \rightarrow \mathbf F$, restricted to the \emph{small} Zariski site of $X$. It is clear that $D\mathbf F$ is a sheaf of $\cal O_X$-modules. Over the small site of $X$, the morphism:
$$
\mathbb G_m^{\varepsilon} \otimes \mathbb G_m^{\varepsilon} \rightarrow \Omega_{X[\varepsilon]/k}^2 \cong \omega_{X}\wedge d\varepsilon,\quad f \otimes g\leadsto d\log f\wedge d\log g
$$
induces an isomorphism $D\mathbf K_2 \xrightarrow{\sim} \omega_X$ (see \cite{van1971k2}). Given any short exact sequence \eqref{eq-bd-data} (i.e., $\mathbf K_2$ is not necessarily central), we obtain an extension of $\cal O_X$-modules:
$$
0 \rightarrow \omega_X \rightarrow D\mathbf E \rightarrow \fr g\underset{k}{\otimes}\cal O_X \rightarrow 0.
$$

\subsubsection{}
For $G = T$ a torus, we shall give an alternative description of $D\mathbf E$, which is in line with the Brylinski--Deligne classification of $\mathbf{CExt}(T, \mathbf K_2)$. Let $p :  X\times\mathbb G_m\rightarrow X$ be the projection map. There holds $\op R^1p_*\mathbf K_2 = 0$ and $p_*\mathbf K_2\xrightarrow{\sim} \mathbf K_2\oplus\mathbf K_1$ (c.f.~\cite[\S3.1]{brylinski2001central}), so \eqref{eq-bd-data} gives rise to an extension together with a morphism:
$$
\xysmall{
1 \ar[r] & p_*\mathbf K_2 \ar[r] \ar[d] & p_*\mathbf E \ar[r] & p_*T \ar[r] & 1 \\
& \mathbf K_1
}
$$
Further inducing along $d\log : \mathbf K_1 \rightarrow \omega_X$, we obtain an extension of $p_*T$ by $\omega_X$. Then the map $\underline{\Lambda}_T \rightarrow p_*T$ determines an extension of $\underline{\Lambda}_T$ by $\omega_X$, or equivalently of $\Lambda_T\underset{\mathbb Z}{\otimes} \cal O_X \xrightarrow{\sim} \fr t \underset{k}{\otimes} \cal O_X$ as $\cal O_X$-modules. We denote the resulting $\omega_X$-extension of $\fr t \underset{k}{\otimes} \cal O_X$ by $E$.

\begin{lem}
\label{lem-bd-diff}
For a torus $T$ and a short exact sequence of big Zariski sheaves of groups:
$$
1 \rightarrow\mathbf K_2 \rightarrow \mathbf E \rightarrow T \rightarrow 1,
$$
the extensions of $\cal O_X$-modules $E$ and $D\mathbf E$ are canonically identified.
\end{lem}
\begin{proof}
It suffices to assume $T=\mathbb G_m$ and compare the $\omega_X$-torsors associated to $E$, respectively $D\mathbf E$, over $1 \in \fr t\underset{k}{\otimes}\cal O_X$. Consider the (small) Zariski $\mathbf K_2$-torsor $\mathbf E_{1}$ over $X\times\mathbb G_m$ associated to $\mathbf E$. Denote its restriction to the first infinitesimal neighborhood of the identity section in $X\times\mathbb G_m$ by $\mathbf E_{1}^{\varepsilon}$. Namely, it is defined over a copy of $X[\varepsilon]$:
$$
\xysmall{
	X[\varepsilon] \ar[r]\ar[dr]_-{p^{\varepsilon}} & X\times\mathbb G_m \ar[d]^p \\
	& X
}
$$
The problem is to compare the following $\omega_X$-torsors:
\begin{enumerate}[(a)]
	\item The pushforward $p_*\mathbf E_{1}$, which is a $p_*\mathbf K_2$-torsor oweing to $\tn R^1p_*\mathbf K_2 = 0$, produces an $\omega_X$-torsor via inducing along the composition:
	$$
	p_*\mathbf K_2 \rightarrow \mathbf K_1 \xrightarrow{d\log}\omega_X.
	$$
	\item The pushforward $p^{\varepsilon}_*\mathbf E_{1}^{\varepsilon}$, which is a $p^{\varepsilon}_*\mathbf K_2$ (i.e., $\mathbf K_2^{\varepsilon}$)-torsor, produces an $\omega_X$-torsor via inducing along the composition:
	$$
	\mathbf K_2^{\varepsilon} \rightarrow D\mathbf K_2 \xrightarrow{\sim} \omega_X.
	$$
\end{enumerate}
We note that there is a commutative diagram of \emph{split} short exact sequences:
$$
\xymatrix@R=1.5em{
	0 \ar[r] & \mathbf K_1 \ar[r]^-{\{p^*(-),t\}}\ar[d] & p_*\mathbf K_2 \ar[r]\ar[d] & \mathbf K_2 \ar[d]^{\tn{id}} \ar[r] & 0 \\
	0 \ar[r] & D\mathbf K_2 \ar[r] & \mathbf K_2^{\varepsilon} \ar[r] & \mathbf K_2 \ar[r] & 0
}
$$
where $t$ is the coordinate on $X\times\mathbb G_m$, regarded as a section of $\mathbf K_1$. In particular, the morphism $\mathbf K_1 \rightarrow D\mathbf K_2$ is given by $\{-, 1 + \varepsilon\}$. Hence the composition $\mathbf K_1 \rightarrow D\mathbf K_2\xrightarrow{\sim} \omega_X$ identifies with $d\log$. This proves that the $\omega_X$-torsors of (a) and (b) are canonically identified.
\end{proof}

\subsubsection{}
Since $d\log : \mathbf K_1 \rightarrow \omega_X$ factors through $\mathring{\Omega}_X^1$, we can canonical factorize the derivative construction $D$ through:
$$
\mathring D : \mathbf{CExt}(T, \mathbf K_2) \rightarrow \mathbf{Ext}(\underline{\fr t}, \mathring{\Omega}_X^1).
$$
Here, $\mathbf{Ext}(\underline{\fr t}, \mathring{\Omega}_X^1)$ is the Picard groupoid of extensions of $\underline{\fr t}$ by $\mathring{\Omega}_X^1$ as Zariski sheaves of $k$-vector spaces. For any reductive group $G$, we obtain a functor from Brylinski--Deligne data to the space of tame quantum parameters $\mathring{\tn{Par}}_G$ (\S\ref{sec-quantum-par}):
$$
\mathring D : \mathbf{CExt}(G, \mathbf K_2) \rightarrow  \mathring{\tn{Par}}_G,\quad \mathbf E \leadsto (\kappa, \mathring E),
$$
where $\kappa$ is the bilinear form attached to $\mathbf E$ by the construction of \cite{brylinski2001central}, together with the derivative $\mathring E$ of the restriction of $\mathbf E$ to the torus $Z_G^{\circ}$.

\subsubsection{}
On the other hand, we have a functor of D.~Gaitsgory \cite{gaitsgory2018parameterization}:
$$
\Xi_{\mathbf{Pic}} : \mathbf{CExt}(G, \mathbf K_2) \rightarrow \mathbf{Pic}^{\tn{fact}}(\tn{Gr}_{G, \tn{Ran}}).
$$
It is proved to be an equivalence (in $\tn{char}(k) = 0$) in \cite{tao2019extensions}. The following commutative diagram summarizes the relationship between quantum parameters and Brylinski--Deligne data.

\begin{cor}
\label{cor-compatibility-bd-data}
The following composition canonically identifies with $\mathring D$.
\begin{align*}
	\mathbf{CExt}(G, \mathbf K_2) \xrightarrow{\Xi_{\mathbf{Pic}}} & \mathbf{Pic}^{\tn{fact}}(\tn{Gr}_{G, \Ran}) \\
	& \rightarrow \mathring{\mathbf{Tw}}{}^{\tn{fact}}(\tn{Gr}_{G, \Ran}) \xrightarrow{\Psi_{\mathring{\mathbf{Tw}}}} \mathring{\tn{Par}}_G.
\end{align*}
\end{cor}
\begin{proof}
By the classification of factorization tame twistings (Theorem \ref{thm-twisting-classification}) and its compability with the classification of factorization line bundles, it suffices to show that the following composition identifies with $\mathring D$:
\begin{align*}
	\mathbf{CExt}(G, \mathbf K_2) \xrightarrow{\Psi_{\mathbf{Pic}}\circ\Xi_{\mathbf{Pic}}} & \Theta_G(\Lambda_T; \mathbf{Pic}) \\
	& \rightarrow \Theta_G(\Lambda_T; \mathring{\mathbf{Tw}})
\xrightarrow{\sim} \mathring{\tn{Par}}_G.
\end{align*}
By definition of enhanced $\Theta$-data, it suffices to do this for $G = T$ a torus. There, the problem reduces to Lemma \ref{lem-bd-diff} and the fact that $\Psi_{\mathbf{Pic}} \circ \Xi_{\mathbf{Pic}}$ identifies with the Brylinski--Deligne classification functor $\mathbf{CExt}(T, \mathbf K_2) \xrightarrow{\sim} \Theta^+(\Lambda_T; \mathbf{Pic})$ after an $\omega$-shift (\cite[\S2]{tao2019extensions}).
\end{proof}

\subsection{Usual factorization twistings}

\subsubsection{} We now fix a semisimple, simply connected group $G_{\tn{sc}}$. The goal is to classify \emph{usual} factorization twistings on $\tn{Gr}_{G_{\tn{sc}}, \tn{Ran}}$. We will deduce the following Theorem from a combination of Theorem \ref{thm-gerbe-classification} and the affine analogue of the Borel--Weil--Bott theorem.

\begin{thm}
\label{thm-usual-twisting-classification}
There is a canonical equivalence of strict Picard $2$-groupoids:
$$
\Psi_{\mathbf{Tw}, G_{\tn{sc}}} : \mathbf{Tw}^{\tn{fact}}(\tn{Gr}_{G_{\tn{sc}}, \tn{Ran}}) \xrightarrow{\sim} \cal Q(\Lambda_{T_{\tn{sc}}}; k)^W.
$$
\end{thm}
\begin{proof}
We use the interpretation of twistings on $Y\in\mathbf{Sch}^{\tn{ft}}_{/k}$ as \'etale $\mathbb G_a$-gerbes on $Y_{\dR}$ equipped with a trivialization over $Y$ (c.f.~\cite[\S6]{gaitsgory2011crystals}). In other words, there is a fiber sequence:
$$
\mathbf{Tw}^{\tn{fact}}(\tn{Gr}_{G_{\tn{sc}}, \tn{Ran}}) \rightarrow \mathbf{Ge}^{+, \tn{fact}}_{\dR}(\tn{Gr}_{G_{\tn{sc}}, \tn{Ran}}) \rightarrow \mathbf{Ge}_{\mathbb G_a}^{\tn{fact}}(\tn{Gr}_{G_{\tn{sc}}, \tn{Ran}}),
$$
where $\mathbf{Ge}^{+, \tn{fact}}_{\dR}$ denotes the theory of additive de Rham gerbes of \S\ref{sec-add-dr-context}, and $\mathbf{Ge}_{\mathbb G_a}$ the sheaf of \'etale $\mathbb G_a$-gerbes. By Theorem \ref{thm-gerbe-classification}, we have an equivalence:
$$
\Psi_{\mathbf{Ge}^+_{\dR}, G_{\tn{sc}}} : \mathbf{Ge}^{+, \tn{fact}}_{\dR}(\tn{Gr}_{G_{\tn{sc}}, \tn{Ran}}) \xrightarrow{\sim} \cal Q(\Lambda_{T_{\tn{sc}}}; k)^W.
$$
Thus it remains to prove that $\mathbf{Ge}_{\mathbb G_a}^{\tn{fact}}(\tn{Gr}_{G_{\tn{sc}}, \tn{Ran}})$ is contractible.

\smallskip

We consider the projection $\pi : \tn{Gr}_{G_{\tn{sc}}, \tn{Ran}} \rightarrow \tn{Ran}$, and claim that each $\mathbb G_a$-gerbe canonically descends to $\tn{Ran}$. Indeed, over $X^I\rightarrow \tn{Ran}$, we consider the base change of $\pi$ as a colimit of the Schubert stratification (c.f.~\S\ref{sec-schubert}) $\pi^{\le\lambda^I} : \tn{Gr}^{\le\lambda^I}_{G_{\tn{sc}}, X^I} \rightarrow X^I$. By the affine Borel--Weil--Bott theorem, we have $\tn H^i(\tn{Gr}^{\le\lambda}_{G_{\tn{sc}}, x}, \cal O) = 0$ for $i\ge 1$ and $\op H^0(\tn{Gr}_{G_{\tn{sc}}, x}^{\le\lambda}, \cal O) \cong k$ at any $k$-point $x\in X$ (c.f.~\cite[Lemma 2.6]{tao2019extensions}). Thus the same holds for fibers of $\pi^{\le\lambda^I}$ at every $k$-point. Since $\pi^{\le\lambda^I}$ is proper, flat, and $X^I$ is reduced, the canonical map $\cal O_{X^I} \rightarrow \tn R\pi_*^{\le\lambda^I}\cal O_{\tn{Gr}}$ is an isomorphism by cohomology and base change. The same argument applies to products of $\pi$. Thus pullback defines an equivalence:
$$
\mathbf{Ge}^{\tn{fact}}_{\mathbb G_a}(\tn{Ran}) \xrightarrow{\sim} \mathbf{Ge}^{\tn{fact}}_{\mathbb G_a}(\tn{Gr}_{G_{\tn{sc}}, \tn{Ran}}).
$$

\smallskip

Finally, we argue that factorization $\mathbb G_a$-gerbes on $\tn{Ran}$ are canonically trivial. By Lemma \ref{lem-ran-contract} below, such a $\mathbb G_a$-gerbe $\cal G$ is pulled back from $\cal G_1$ along $p : \tn{Ran} \rightarrow \tn{pt}$. We choose distinct $k$-points $x, y \in X$. The pullbacks $x^*\cal G$, $y^*\cal G$, $(x,y)^*\cal G$ all identify with $\cal G_1$. However, factorization implies $(x, y)^*\cal G \xrightarrow{\sim} x^*\cal G \otimes y^*\cal G$ so we obtain a trivialization of $\cal G_1$ which one can see to be canonical.
\end{proof}

\subsubsection{} We supply a quick calculation of the cohomology of $\tn{Ran}$ with values in $\mathbb G_a$.

\begin{lem}
\label{lem-ran-contract}
Pullback along $\tn{Ran}\rightarrow\tn{pt}$ induces an isomorphism $k \xrightarrow{\sim} \op R\Gamma(\tn{Ran}; \cal O)$.
\end{lem}
\begin{proof}
We note that $\tn R\Gamma(\tn{Ran}; \cal O)$ is by definition $\lim_I \tn R\Gamma(X^I; \cal O)$. Suppose $X$ is proper. Then each $\tn R\Gamma(X^I; \cal O)$ is dualizable. Hence we have:
\begin{align*}
\tn R\Gamma(\tn{Ran}\times\tn{Ran}; \cal O) \xrightarrow{\sim} & \lim_{I, J}\tn R\Gamma(X^I \times X^J; \cal O) \\
& \xrightarrow{\sim} \lim_I \tn R\Gamma(X^I; \cal O) \otimes \lim_J \tn R\Gamma(X^J; \cal O) \xrightarrow{\sim} \tn R\Gamma(\tn{Ran}; \cal O) \otimes \tn R\Gamma(\tn{Ran}; \cal O).
\end{align*}
The argument of \cite[\S6]{gaitsgory2013contractibility} thus applies.

\smallskip

Suppose $X$ is affine. Then $\tn R\Gamma(X^I; \cal O) \xrightarrow{\sim} \Gamma(X^I; \cal O)$ and the problem reduces to the fact that global functions on $\tn{Ran}$ are constant (\cite[Proposition 4.3.10(1)]{zhu2016introduction}).
\end{proof}

\bigskip

\section{Proof of Theorem \ref{thm-gerbe-classification}}
\label{sec-classification-proof}

Throughout this section, we fix a topology $t$ on $\mathbf{Sch}^{\tn{ft}}_{/k}$ stronger than the \'etale topology and such that every object of $\mathbf{Sch}^{\tn{ft}}_{/k}$ is $t$-locally smooth. Let $X$ be a smooth curve, $G$ a reductive group, and $\mathbf G$ be a motivic $t$-theory of gerbes whose coefficient group $A(-1)$ is divisible.

\smallskip

The goal of this section is to prove Theorem \ref{thm-gerbe-classification}.

\subsection{Tori}
\label{sec-classification-tori}

\subsubsection{} To prove Theorem \ref{thm-gerbe-classification} for torus, we recall the definition of $\Psi_{\mathbf G, T}$ as the composition:
$$
\mathbf G^{\tn{fact}}(\tn{Gr}_{T, \tn{Ran}}) \rightarrow \mathbf G^{\tn{fact}}(\tn{Gr}_{T, \tn{comb}}) \xrightarrow{\sim} \Theta(\Lambda_T; \mathbf G),
$$
where the equivalence is already proved in Lemma \ref{lem-grT-comb-classification}.

\begin{lem}
The canonical map $\tn{Gr}_{T, \tn{comb}} \rightarrow \tn{Gr}_{T, \tn{Ran}}$ is an isomorphism after $t$-sheafification.
\end{lem}
\begin{proof}
The map is clearly a monomorphism of prestacks.\footnote{We are within classical algebraic geometry.} It suffices to check that it is surjective in the $t$-topology, and we reduce immediately to the case $T=\mathbb G_m$. Consider any $S$-point $(x^{(i)}, \cal L, \alpha)$ of $\tn{Gr}_{\mathbb G_m, \tn{Ran}}$. It belongs to $\op{Gr}_{\mathbb G_m, \tn{comb}}$ if and only if $L$ is isomorphic to $\cal O(\sum_i\lambda_i\Gamma_{x^{(i)}})$ for some $\lambda_i\in\mathbb Z$, and $\alpha$ identifies with its canonical trivialization. This is indeed the case after passing to any $\tau$-cover $\widetilde S\rightarrow S$ with $\widetilde S$ smooth.
\end{proof}

\subsubsection{} Since $\mathbf G$ satisfies $t$-descent, the Lemma implies that we have an isomorphism:
\begin{equation}
\label{eq-tori-without-fact}
\mathbf G(\tn{Gr}_{T, \tn{Ran}}) \xrightarrow{\sim} \mathbf G(\tn{Gr}_{T, \tn{comb}}).
\end{equation}
On the other hand, the map $\tn{Gr}_{T, \tn{Ran}}^{\times n} \rightarrow \tn{Gr}_{T, \tn{comb}}^{\times n}$ is an isomorphism after $t$-sheafification for all $n\ge 1$. Therefore the isomorphism \eqref{eq-tori-without-fact} lifts to one between factorization sections, so we have proved that $\Psi_{\mathbf G, T}$ is an equivalence.

\subsection{Semisimple, simply connected groups}

\subsubsection{}
\label{sec-fact-gerbe-to-form}
For any reductive group $G$ with a fixed maximal torus $T$, we consider the composition:
\begin{align*}
\cal Q_{\mathbf G, G} : \mathbf G^{\tn{fact}}(\tn{Gr}_G) \rightarrow & \mathbf G^{\tn{fact}}(\tn{Gr}_T) \\
& \xrightarrow{\sim} \Theta(\Lambda_T; \mathbf G) \rightarrow \cal Q(\Lambda_T; A(-1))
\end{align*}
Thus $\cal Q_{\mathbf G, G}$ associates a quadratic form to any factorization gerbe. This functor will be the basis of the classification of factorization gerbes for semisimple, simply connected groups.

\subsubsection{} Let $G_{\tn{sc}}$ be a semisimple, simply connected group with maximal torus $T_{\tn{sc}}$. We let $\mathbf S$ denote the set of its simple factors. The analogous procedure to \S\ref{sec-fact-gerbe-to-form} defines an equivalence of Picard groupoids:
\begin{equation}
\label{eq-sc-pic-classify}
\cal Q_{\mathbf{Pic}, G_{\tn{sc}}} : \mathbf{Pic}^{\tn{fact}}(\tn{Gr}_{G, \tn{sc}}) \xrightarrow{\sim} \cal Q(\Lambda_{T_{\tn{sc}}}; \mathbb Z)^W.
\end{equation}
In fact, $\cal Q(\Lambda_{T_{\tn{sc}}}; \mathbb Z)^W$ canonically identifies with $\tn{Maps}(\mathbf S, \mathbb Z)$. For each $s\in\mathbf S$, the mapping which sends $s$ to $1$ and all other elements to zero passes to the \emph{minimal} quadratic form $q_{\tn{min}, s}$ on $\Lambda_{T_{\tn{sc}}}$ which has $q_{\tn{min}, s}(\alpha_s) = 1$ for $\alpha_s$ a short coroot in $\Phi_s$ and vanishes on components associated to other simple factors. Under \eqref{eq-sc-pic-classify}, this passes to the minimal line bundle $\tn{min}_s$ (c.f.~\cite{faltings2003algebraic}) which has a factorization structure by \cite{tao2019extensions}.

\subsubsection{}
The inverse of \eqref{eq-sc-pic-classify} paired with the divisor class map defines a functor:
\begin{equation}
\label{eq-sc-construct}
\cal Q(\Lambda_{T_{\tn{sc}}}; \mathbb Z)^W \underset{\mathbb Z}{\otimes} A(-1) \rightarrow \mathbf G^{\tn{fact}}(\tn{Gr}_{G_{\tn{sc}}, \tn{Ran}}).
\end{equation}
By construction, the composition of \eqref{eq-sc-construct} with $\cal Q_{\mathbf G, G}$ is the forgetful map from $\cal Q(\Lambda_{T_{\tn{sc}}}; \mathbb Z)^W \underset{\mathbb Z}{\otimes} A(-1)$ to $\cal Q(\Lambda_{T_{\tn{sc}}}; A(-1))$.

\subsubsection{} Fix a point $x\in X$. By Lemma \ref{lem-grT-comb-classification} applied to the trivial group, we see that every factorization gerbe on $\tn{Gr}_{G_{\tn{sc}}, \tn{Ran}}$ is canonically trivialized when pulled back to the unit section. Thus the functor of restriction to $x$ factors through the category of gerbes rigidified at the unit point:
\begin{equation}
\label{eq-restriction-to-point}
\tn{Res}_x : \mathbf G^{\tn{fact}}(\tn{Gr}_{G_{\tn{sc}}, \tn{Ran}}) \rightarrow \mathbf G^e(\tn{Gr}_{G_{\tn{sc}}, x}).
\end{equation}

\begin{lem}
\label{lem-gr-point-classification}
The composition of $\eqref{eq-sc-construct}$ with $\tn{Res}_x$ is an equivalence:
$$
\cal Q(\Lambda_{T_{\tn{sc}}}; \mathbb Z)^W \underset{\mathbb Z}{\otimes} A(-1) \xrightarrow{\sim} \mathbf G^e(\tn{Gr}_{G_{\tn{sc}}, x}).
$$
\end{lem}
\noindent
In particular, $\mathbf G^e(\tn{Gr}_{G_{\tn{sc}}, x})$ is discrete.
\begin{proof}
By the product decomposition (Lemma \ref{lem-product-decomposition}), we reduce to the case $\mathbf S=\{1\}$, i.e., $G_{\tn{sc}}$ is simple and simply connected. We shall denote it simply by $G$. Choose a uniformizer $t$ of $\widehat{\cal O}_{X, x}$ and identify $\tn{Gr}_{G, x}$ with the \'etale quotient $G\loo{t}/G\arc{t}$. Recall that the morphism $\fr p : \tn{Fl}_G \rightarrow \tn{Gr}_{G,x}$ is an \'etale-locally trivial fiber bundle with typical fiber $G/B$.

\smallskip

We first observe that $\mathbf G^e(G/B)$ is canonically isomorphic to $\tn{Maps}(\Delta, A(-1))$ for $\Delta$ the set of simple roots. Indeed, a gerbe rigidified at the unit point $e$ of the big Bruhat cell $N^-e$ must be trivialized over $N^-e$ (Property (A)). The complement of $N^-e$ is an effective Cartier divisor whose irreducible components are labeled by $\Delta$. An appliction of Properties (RP1) and (RP2) shows that $\mathbf G^e(G/B) \xrightarrow{\sim} \tn{Maps}(\Delta, A(-1))$.

\smallskip

Therefore, the product decomposition and \'etale descent shows that
$$
\fr p^* : \mathbf G^e(\tn{Gr}_{G, x}) \rightarrow \mathbf G^e(\tn{Fl}_G)
$$
is fully faithful, and its image identifies with \emph{fiberwise} trivial objects. Since $\tn{Gr}_{G, x}$ is connected, the condition on fiberwise triviality is equivalent to triviality along the unit fiber $G/B\hookrightarrow\tn{Fl}_G$.

\smallskip

Next, we classify gerbes on $\tn{Fl}_G$ using a geometric description given in Faltings \cite[Theorem 7]{faltings2003algebraic}. To recall, let $e\in\tn{Fl}_G$ denote the unit $k$-point. Write $\mathring{I}^-$ for the subgroup of $G[t^{-1}]$ which is the preimage of $N^-$ under the quotient map $G[t^{-1}]\twoheadrightarrow G$. For each $n\ge 1$, write $\mathring I^-(n)$ for the subgroup of $\mathring I^-$ whose projection mod $t^{-n}$ is contained in $T[t^{-1}]$. Then the $\mathring I^-$-orbits on $\op{Fl}_G$ are parametrized by the affine Weyl group $W^{\tn{aff}}$ and:
$$
\mathring I^-w e \subset \overline{\mathring I^-w' e} \iff w'\preceq w \text{ in the Bruhat ordering}.
$$
Consider a subset $A\subset W^{\tn{aff}}$ with the property that $w\in A$ implies $w'\in A$ for all $w'\preceq w$. Then $\Omega_A := \bigcup_{w\in A} w \mathring I^- e$ is an open, $\mathring I^-$-invariant subset of $\op{Fl}_G$. For sufficiently large integer $n$, the quotient (as \'etale sheaves) $\Omega_A/ \mathring I^-(n)$ is represented by a \emph{smooth} scheme (Lemma 6 of \emph{loc.cit.}), and furthermore, the $\mathring I^-$-orbits in $\Omega_A$ are preimages of affine spaces:
$$
\xysmall{
	\mathring I^-we \ar@{^{(}->}[r]\ar[d] & \Omega_A \ar[d] \\
	\mathbb A^d \ar@{^{(}->}[r] & \Omega_A/\mathring I^-(n)
}
$$
and $\mathring I^-we$ is of codimension $l(w)$.

\smallskip

We now make the observation that $\mathring I^-(n)$ has a contracting $\mathbb G_m$-action by scaling $t$ (c.f.~\S\ref{lem-contraction}). Indeed, $G[t^{-1}]$ already admits a contracting $\mathbb G_m$-action which preserves $\mathring I^-(n)$. The fixed-point locus in $G[t^{-1}]$ is the subgroup $G$ and we have $\mathring I^-(n) \cap G = \{1\}$. By Lemma \ref{lem-contraction}, $\mathbf G(\Omega_A)$ identifies with $\mathbf G(\Omega_A \times\mathring I^-(n)^{\bullet})$, so \'etale descent implies an equivalence:
$$
\mathbf G(\Omega_A/ \mathring I^-(n)) \xrightarrow{\sim} \mathbf G(\Omega_A).
$$

\smallskip

On the other hand, for $A$ sufficiently large, the complement of the big cell $\mathring I^-e/\mathring I^-(n)$ in $\Omega_A/\mathring I^-(n)$ is the union of effective Cartier divisors corresponding to the set of simple \emph{affine} roots $\Delta^{\tn{aff}} = \Delta\sqcup\{\theta\}$. Thus an argument as for the usual flag variety implies that $\mathbf G^e(\Omega_A/\mathring I^-(n)) \xrightarrow{\sim} \tn{Maps}(\Delta^{\tn{aff}}, A(-1))$. Summarizing, we have:
\begin{align*}
\mathbf G^e(\tn{Gr}_{G, x})\hookrightarrow \mathbf G^e(\tn{Fl}_G) & \xrightarrow{\sim} \mathbf G^e(\Omega_A) \\
& \xrightarrow{\sim} \mathbf G^e(\Omega_A/\mathring I^-(n)) \xrightarrow{\sim} \op{Maps}(\Delta^{\tn{aff}}, A(-1)).
\end{align*}
It remains to observe that the restriction $\mathbf G^e(\tn{Fl}_G) \rightarrow \mathbf G^e(G/B)$ to the unit fiber passes to the restriction of functions $\op{Maps}(\Delta^{\tn{aff}}, A(-1)) \rightarrow \op{Maps}(\Delta, A(-1))$, and furthermore, the gerbe $\cal G\in\mathbf G^e(\tn{Gr}_G)$ corresponding to the function with value $a\in A(-1)$ at $\theta$ is precisely the $a$th power of the minimal line bundle on $\tn{Gr}_{G, x}$.
\end{proof}

\subsubsection{}
\label{sec-pic-to-gerbe-ss} We now analyze the process of restriction to $x\in X$. Let $A'$ denote the abelian group $\cal Q(\Lambda_{T_{\tn{sc}}}, \mathbb Z)^W \underset{\mathbb Z}{\otimes} A(-1) \cong \op{Maps}(\mathbf S, A(-1))$. Write $\mathbf G^e_{\tn{Gr}_{G_{\tn{sc}}}/X^n}$ for the (small) \'etale sheaf on $X^n$ whose value at $S\rightarrow X^n$ is the strict Picard $2$-groupoid of gerbes on $\tn{Gr}_{G_{\tn{sc}}, \tn{Ran}} \underset{\tn{Ran}}{\times} S$ trivialized at the unit section. For $n=1$, the functor \eqref{eq-sc-construct} defines a morphism of \'etale sheaves on $X$:
\begin{equation}
\label{eq-sc-construct-over-x}
\underline{A}'_X \rightarrow \mathbf G^e_{\tn{Gr}_{G_{\tn{sc}}}/X}.
\end{equation}
By Lemma \ref{lem-gr-point-classification} and Property (B), the stalks of \eqref{eq-sc-construct-over-x} at any $k$-point $x\in X$ are mutual retracts. Hence \eqref{eq-sc-construct-over-x} is an isomorphism. Now, the divisor class map and \eqref{eq-pic-exact-seq} induces a morphism:
\begin{equation}
\label{eq-sc-construct-over-ran}
\xysmall{
0 \ar[r] & \mathbf{Pic}^e_{\tn{Gr}_{G_{\tn{sc}}}/X^I}\underset{\mathbb Z}{\otimes} A(-1) \ar[r]\ar[d]^{\tn{div}} & \boxtimes_{i\in I} \underline{A}'_X \ar[r]^-{\delta}\ar[d]^{\cong} & \bigoplus_{\substack{I\twoheadrightarrow J\\ |J| = |I|-1}} (\Delta_{I\twoheadrightarrow J})_* \boxtimes_{j\in J}\underline{A}'_X \ar[d]^{\cong} \\
 & \mathbf G^e_{\tn{Gr}_{G_{\tn{sc}}}/X^I} \ar[r] & \boxtimes_{i\in I} \underline{A}'_X \ar[r]^-{\delta} & \bigoplus_{\substack{I\twoheadrightarrow J\\ |J| = |I|-1}} (\Delta_{I\twoheadrightarrow J})_* \boxtimes_{j\in J}\underline{A}'_X
}
\end{equation}
Here, the morphism $\mathbf G^e_{\tn{Gr}_{G_{\tn{sc}}}/X^I} \rightarrow \boxtimes_{i\in I} \underline{A}'_X$ is defined by restriction away from all diagonals using \eqref{eq-sc-construct-over-x}. By checking on stalks using Property (B), we see that $\tn{div}$ is also an equivalence. This implies that $\mathbf G^e_{\tn{Gr}_{G_{\tn{sc}}}/X^I}$ identifies with kernel of the map $\delta$.

\smallskip

Since $\delta$ is defined by taking difference along each diagonal, we see that restriction to $x\in X$ defines an equivalence:
$$
\mathbf G^e(\tn{Gr}_{G_{\tn{sc}}, \tn{Ran}}) \xrightarrow{\sim} \mathbf G^e(\tn{Gr}_{G_{\tn{sc}}, x}).
$$
Tautologically, the functor $\tn{Res}_x$ \eqref{eq-restriction-to-point} factors through the above equivalence.

\begin{lem}
\label{lem-restriction-fully-faithful}
The functor $\tn{Res}_x$ is fully faithful.
\end{lem}
\begin{proof}
It remains to prove that the forgetful functor:
$$
\mathbf G^{\tn{fact}}(\tn{Gr}_{G_{\tn{sc}}, \tn{Ran}}) \rightarrow \mathbf G^e(\tn{Gr}_{G_{\tn{sc}}, \tn{Ran}})
$$
is fully faithful. Since rigidified gerbes on $(\tn{Gr}_{G_{\tn{sc}}, \tn{Ran}})_{\tn{disj}}^{\times 2}$ are classified by the discrete abelian group $A'\times A'$, a factorization structure is unique if it exists.
\end{proof}

\subsubsection{} We now finish the classification for $G_{\tn{sc}}$.

\begin{lem}
\label{lem-classification-sc}
Let $G_{\tn{sc}}$ be a semisimple, simply connected group. Then $\cal Q_{\mathbf G, G}$ has image in $\cal Q(\Lambda_{T_{\tn{sc}}}, A(-1))^W_{\tn{restr}}$ and defines an equivalence:
$$
\Psi_{\mathbf G, G_{\tn{sc}}} : \mathbf G^{\tn{fact}}(\tn{Gr}_{G_{\tn{sc}}}) \xrightarrow{\sim} \cal Q(\Lambda_{T_{\tn{sc}}}; A(-1))^W_{\tn{restr}}.
$$
\end{lem}
\begin{proof}
Recall that $\cal Q(\Lambda_{T_{\tn{sc}}}; A(-1))^W_{\tn{restr}}$ identifies with $\cal Q(\Lambda_{T_{\tn{sc}}}; \mathbb Z)^W \underset{\mathbb Z}{\otimes} A(-1)$ (Lemma \ref{lem-sc-restr}). We have seen that there is a factoring of its embedding inside $\cal Q(\Lambda_{T_{\tn{sc}}}; A(-1))$ as follows.
$$
\cal Q(\Lambda_{T_{\tn{sc}}}; A(-1))^W_{\tn{restr}} \rightarrow \mathbf G^{\tn{fact}}(\tn{Gr}_{G_{\tn{sc}}}) \xrightarrow{\cal Q_{\mathbf G, G}} \cal Q(\Lambda_{T_{\tn{sc}}}; A(-1)).
$$
The first functor is an equivalence by combining Lemma \ref{lem-gr-point-classification} and Lemma \ref{lem-restriction-fully-faithful}.
\end{proof}

\subsection{Construction of $\Psi_{\mathbf G}$}

\subsubsection{} We start with a mild generalization of the classification result for semisimple, simply connected groups. Let $G$ be a reductive group whose derived subgroup $G_{\tn{der}}$ is simply connected. Denote by $T_1$ the quotient torus $G/G_{\tn{der}}$. We know by \cite[Lemma 3.4]{tao2019extensions} that the projection $\tn{Gr}_{G, \tn{Ran}} \rightarrow \tn{Gr}_{T_1, \tn{Ran}}$ is an \'etale fiber bundle with typical fiber $\tn{Gr}_{G_{\tn{der}}, \tn{Ran}}$. In other words, to every $S$-point of $\tn{Gr}_{T_1, \tn{Ran}}$ one can associate an \'etale cover $\widetilde S\rightarrow S$ and an isomorphism:
\begin{equation}
\label{eq-etale-der-bundle}
\widetilde S \underset{\tn{Ran}}{\times} \op{Gr}_{G_{\tn{der}}} \xrightarrow{\sim} \widetilde S\underset{\tn{Gr}_{T_1, \tn{Ran}}}{\times} \op{Gr}_{G, \tn{Ran}}.
\end{equation}

\subsubsection{} We will now identify the fiber of $\cal Q_{\mathbf G, G}$ (see \S\ref{sec-fact-gerbe-to-form}) when $G_{\tn{der}}$ is simply connected.

\begin{lem}
\label{lem-fiber-quad-der}
Suppose $G_{\tn{der}}$ is simply connected. Then pulling back along $\tn{Gr}_{G, \tn{Ran}}\rightarrow \tn{Gr}_{T_1, \tn{Ran}}$ defines a fiber sequence of strict Picard $2$-groupoids:
$$
\mathbf G^{\tn{fact}}(\op{Gr}_{T_1, \tn{Ran}}) \rightarrow \mathbf G^{\tn{fact}}(\tn{Gr}_{G, \tn{Ran}}) \rightarrow \cal Q(\Lambda_{T_{\tn{der}}}; A(-1)).
$$
\end{lem}
\begin{proof}
Let $\mathbf G_{\tn{Gr}_G/\tn{Gr}_{T_1}}$ denote the \'etale sheafification of the presheaf on $\tn{Gr}_{T_1}$:
$$
S \leadsto \tn{Cofib}(\mathbf G(S) \rightarrow\mathbf G(S\underset{\tn{Gr}_{T_1, \tn{Ran}}}{\times}\tn{Gr}_{G, \tn{Ran}})).
$$
Let $\mathbf{Pic}_{\tn{Gr}_G/\tn{Gr}_{T_1}}$ be the analogously defined \'etale sheaf where we replace $\mathbf G$ by $\mathbf{Pic}$. We claim that the divisor class map $\mathbf{Pic}_{\tn{Gr}_G/\tn{Gr}_{T_1}}\underset{\mathbb Z}{\otimes} A(-1) \rightarrow \mathbf G_{\tn{Gr}_G/\tn{Gr}_{T_1}}$ is an isomorphism. Indeed, it suffices to show the map on presheaves is an \'etale local equivalence. Take any $S$-point of $\tn{Gr}_{T_1}$, an \'etale cover $\widetilde S$ together with an isomorphism \eqref{eq-etale-der-bundle} reduces the claim to identifying the cofibers of the horizontal maps:
$$
\xysmall{
	\mathbf{Pic}(\widetilde S)\underset{\mathbb Z}{\otimes}A(-1) \ar[r]\ar[d] & \mathbf{Pic}(\widetilde S\underset{\tn{Ran}}{\times} \op{Gr}_{G_{\tn{der}}})\underset{\mathbb Z}{\otimes}A(-1) \ar[d] \\
	\mathbf G(\widetilde S) \ar[r] & \mathbf G(\widetilde S\underset{\tn{Ran}}{\times} \op{Gr}_{G_{\tn{der}}})
}
$$
This in turn follows from the identification $\mathbf{Pic}^e_{\tn{Gr}_{G_{\tn{der}}}/X^I} \underset{\mathbb Z}{\otimes}A(-1) \xrightarrow{\sim} \mathbf G^e_{\tn{Gr}_{G_{\tn{der}}}/X^I}$ of \eqref{eq-sc-construct-over-ran}.

\smallskip

In particular, $\mathbf G_{\tn{Gr}_G/\tn{Gr}_{T_1}}$ is \'etale locally isomorphic to a subsheaf of $\boxtimes_{i\in I}\underline A'_X$ (see \S\ref{sec-pic-to-gerbe-ss}). Then the argument of \cite[\S3.4.3]{tao2019extensions} applies. Namely, starting with a section $g$ of $\mathbf G_{\tn{Gr}_G/\tn{Gr}_{T_1}}$ over $\tn{Gr}_{T_1, \tn{Ran}}$, the hypothesis shows that $g$ vanishes over the unit section. To obtain the vanishing of the restriction $g^{(\lambda)}$ to the connected component $\tn{Gr}_{T_1}^{\lambda}$, we consider the section $g^{(\lambda, -\lambda)}$ over $\tn{Gr}_{T_1}^{(\lambda, -\lambda)}$. The fact that $g^{(\lambda, -\lambda)}$ vanishes over the diagonal in $X^2$ implies that $g^{(\lambda, -\lambda)}$, hence $g^{(\lambda)}$, vanishes. The vanishing of the sections $g^{(\lambda^I)}$ with $|I|\ge 2$ then follows by restriction away from the diagonals (see \cite[\S3.4.3]{tao2019extensions} for details).
\end{proof}

\subsubsection{} Let us control the type of quadratic forms that can arise from factorization gerbes. We remove the assumption on $G_{\tn{der}}$ and instead consider any reductive group $G$.

\begin{lem}
\label{lem-quad-image}
The image of $\cal Q_{\mathbf G, G}$ is contained in $\cal Q(\Lambda_T; A(-1))^W_{\tn{restr}}$.
\end{lem}
\begin{proof}
Let $\cal G\in\mathbf G^{\tn{fact}}(\tn{Gr}_{G, \tn{Ran}})$ and $q:=\cal Q_{\mathbf G, G}(\cal G)$. We need to establish the following identities for each simple co-root $\alpha_i$ and co-character $\lambda\in\Lambda_T$.
\begin{enumerate}[(a)]
	\item $q(s_{\alpha_i}(\lambda)) = q(\lambda)$;
	\smallskip
	\item $\kappa(\alpha_i, \lambda) = \langle\check{\alpha}_i, \lambda\rangle q(\alpha_i)$.
\end{enumerate}
Consider the parabolic subgroup $P\subset G$ generated by $T$ and $\alpha_i$. The quotient of $P$ by its nilradical $N_P$ is a reductive group $M$ of semisimple rank $1$. We have the following maps:
$$
\xysmall{
	& \op{Gr}_{P, \tn{Ran}} \ar[dl]_{\fr p}\ar[dr]^{\fr q } & \\
	\op{Gr}_{G, \tn{Ran}} & & \op{Gr}_{M, \tn{Ran}}.
}
$$

\smallskip

We observe that $\fr q$ is an \'etale fiber bundle with typical fiber $\tn{Gr}_{N_P, \tn{Ran}}$. On the other hand, there is a contracting $\mathbb G_m$-action on $\tn{Gr}_{N_P, \tn{Ran}}$ given by the co-root $\alpha_i$ whose fixed point locus is the unit section. By Lemma \ref{lem-contraction} and \'etale descent, we see that $\fr p^*\cal G$ canonically identifies with $\fr q^*\cal G_M$ for some $\cal G_M \in \mathbf G^{\tn{fact}}(\tn{Gr}_{M, \tn{Ran}})$. Regarding $\alpha_i$ as a co-root of $M$, we reduce the problem to reductive groups of semisimple rank $1$, with unique simple co-root $\alpha$. Such a group $G$ must be the direct product of a torus $T_1$ with $G_1 = \tn{SL}_2$, $\tn{GL}_2$, or $\tn{PGL}_2$.

\smallskip

To verify (a), we exhibit two paths $\gamma_1,\gamma_2 : \mathbb A^1 \rightarrow G$ such that:
$$
\gamma_1(0) = e,\quad \gamma_1(1) = \gamma_2(1), \quad \gamma_2(0) = \tilde s_{\alpha}.
$$
where $\tilde s_{\alpha}$ a lift of $s_{\alpha}\in W$ to $G$. For instance, we may set $\gamma_1,\gamma_2$ to be identity on the factor $T_1$ and be given by the following matrices for the $G_1$ factor:
$$
\gamma_1(t) = 
\begin{pmatrix}
1 & 2t \\
0 & 1
\end{pmatrix}
,\quad
\gamma_2(t) =
\begin{pmatrix}
t & t+1 \\
t-1 & t
\end{pmatrix}.
$$
As $G$ acts on itself by inner automorphisms, we have action morphisms $\mathbb A^1\times \tn{Gr}_{G,\tn{Ran}} \rightarrow \tn{Gr}_{G, \tn{Ran}}$ defined by $\gamma_1$ and $\gamma_2$. Pulling back $\cal G$ produces two factorization gerbes $\cal G_{\gamma_1}$, $\cal G_{\gamma_2}$ on $\mathbb A^1 \times \tn{Gr}_{G, \tn{Ran}}$. Thus $\mathbb A^1$-invariance (Lemma \ref{lem-contraction}) gives isomorphisms:
$$
\cal G \xrightarrow{\sim} \gamma_1(1)^*\cal G\xrightarrow{\sim} \gamma_2(1)^*\cal G \xrightarrow{\sim} \tilde s_{\alpha}^*\cal G.
$$
This proves identity (a).

\smallskip

For identity (b), we only need to consider the case $G = T_1 \times\tn{SL}_2$ as the other two cases are vacuous (c.f.~\S\ref{sec-quad-restr}). We claim that external product defines an equivalence:
$$
\mathbf G^{\tn{fact}}(\tn{Gr}_{T_1, \tn{Ran}}) \times \mathbf G^{\tn{fact}}(\tn{Gr}_{\tn{SL}_2, \tn{Ran}}) \xrightarrow{\sim} \mathbf G^{\tn{fact}}(\tn{Gr}_{G, \tn{Ran}}).
$$
Indeed, given $\cal G\in\mathbf G^{\tn{fact}}(\tn{Gr}_G)$, pulling back along $\tn{Gr}_{G,\tn{Ran}}\rightarrow \tn{Gr}_{\tn{SL}_2, \tn{Ran}}\rightarrow \tn{Gr}_{G,\tn{Ran}}$ and taking the quotient, we obtain a gerbe $\cal G_1 \in\mathbf G^{\tn{fact}}(\tn{Gr}_G)$ whose associated quadratic form vanishes on $\Lambda_{T_{\tn{der}}}$. Since $\tn{SL}_2$ is simply connected, Lemma \ref{lem-fiber-quad-der} applies and we see that $\cal G_1$ is pulled back from $\tn{Gr}_{T_1, \tn{Ran}}$. Having the product decomposition, the desired identity follows from the classification for semisimple, simply connected groups (Lemma \ref{lem-classification-sc}).
\end{proof}

\subsubsection{} We now combine the above ingredients to build the classification functor:
$$
\Psi_{\mathbf G} : \mathbf G^{\tn{fact}}(\tn{Gr}_{G, \tn{Ran}}) \rightarrow \Theta_G(\Lambda_T; \mathbf G).
$$
Indeed, given $\cal G\in\mathbf G^{\tn{fact}}(\tn{Gr}_{G, \tn{Ran}})$, the procedure of \S\ref{sec-fact-gerbe-to-form} produces a $\Theta$-datum $(q, \cal G^{(\lambda)}) \in \Theta(\Lambda_T; \mathbf G)$. Lemma \ref{lem-quad-image} shows that $q$ indeed lies in $\cal Q(\Lambda_T; A(-1))^W_{\tn{restr}}$.

\smallskip

It remains to produce the isomorphism $\varepsilon$. Indeed, the restriction of $\cal G$ to $\tn{Gr}_{\widetilde G_{\tn{der}}, \tn{Ran}}$ is the factorization gerbe classified by $q\big|_{\Lambda_{\widetilde T_{\tn{der}}}}$ via Lemma \ref{lem-classification-sc}. Thus we obtain an isomorphism $\varepsilon$ of $\Theta$-data for the lattice $\Lambda_{\widetilde T_{\tn{der}}}$ by functoriality of pullback along the following diagram.
$$
\xysmall{
	\tn{Gr}_{\widetilde T_{\tn{der}}, \tn{Ran}} \ar[r]\ar[d] & \tn{Gr}_{\widetilde G_{\tn{der}}, \tn{Ran}} \ar[d] \\
	\tn{Gr}_{T, \tn{Ran}} \ar[r] & \tn{Gr}_{G, \tn{Ran}}
}
$$

\subsection{$\Psi_{\mathbf G}$ is an equivalence}

\subsubsection{} Our final goal is to prove that the classification functor $\Psi_{\mathbf G}$, constructed in the previous subsection, is an equivalence of categories. In order to do so, we will first perform a reduction using the following geometric input.

\begin{lem}
\label{lem-tau-surj}
Suppose $G'\rightarrow G$ is a map of reductive groups whose kernel is a torus. Then the morphism $\op{Gr}_{G', \tn{Ran}} \rightarrow \op{Gr}_{G, \tn{Ran}}$ is surjective in the $t$-topology.
\end{lem}
\begin{proof}
One takes an $S$-point of $\op{Gr}_G$ represented by $(x^{(i)}, \cal P_G, \alpha)$. By the Drinfeld--Simpson theorem, we may assume that $\cal P_G$ is Zariski-locally trivial after an \'etale cover of $S$. A reduction of the datum $(\cal P_G, \alpha)$ to the structure group $G'$ is thus equivalent to the trivialization of a section of $i^!T[2]$ in the Zariski topology of $S\times X$, where $i$ denotes the closed immersion:
$$
\bigcup_{i\in I} \Gamma_{x^{(i)}} \xrightarrow{i} S\times X \xleftarrow{j} U_{\{x^{(i)}\}}.
$$
We shall show that over a $t$-cover $\widetilde S\rightarrow S$ with $\widetilde S$ smooth, every section of $i^!T[2]$ admits a trivialization. To prove this statement, one reduces to $T=\mathbb G_m$. The canonical triangle $i^!\mathbb G_m\rightarrow\mathbb G_m\rightarrow\op Rj_*\mathbb G_m$ induces a long exact sequence:
$$
\op{Pic}(\widetilde S\times X) \rightarrow \op{Pic}(U_{\{x^{(i)}\}}) \rightarrow \op H^2(\widetilde S\times X; i^!\mathbb G_m) \rightarrow 0.
$$
The map on Picard groups is surjective by smoothness of $\widetilde S$. Thus $\op H^2(\widetilde S\times X; i^!\mathbb G_m) = 0$.
\end{proof}

\subsubsection{} Recall that a \emph{$z$-extension} of $G$ is a short exact sequence of reductive groups:
$$
1 \rightarrow T_2 \rightarrow G' \rightarrow G\rightarrow 1.
$$
where the derived subgroup $G'_{\tn{der}}\subset G'$ is simply connected. Its existence is assured by the combinatorics of root data (c.f.~\cite[Proposition 3.1]{milne1982conjugates}). We fix a $z$-extension of $G$ and let $T_1$ be the quotient torus $G'/G'_{\tn{der}}$. Then the quotient of lattices $\Lambda_{T_1}/\Lambda_{T_2}$ identifies with $\pi_1G$.

\subsubsection{} One sees directly that $T_2$ is central in $G'$. Thus the \v Cech nerve of $G'\rightarrow G$ is in fact a co-simplicial system of group schemes $G'\times T_2^{\bullet}$. Since the formation of the affine Grassmannian commutes with product of groups, we see that the \v Cech nerve of $\tn{Gr}_{G', \tn{Ran}} \rightarrow \tn{Gr}_{G, \tn{Ran}}$ is co-simplicial system of prestacks $\tn{Gr}_{G' \times T_2^{\bullet}, \tn{Ran}}$.

\smallskip
We have a commutative diagram of strict Picard $2$-groupoids:
$$
\xymatrix@R=1.5em@C=4em{
	\mathbf G^{\tn{fact}}(\tn{Gr}_{G, \tn{Ran}}) \ar[r]^-{\Psi_{\mathbf G, G}}\ar[d] & \Theta_G(\Lambda_T; \mathbf G) \ar[d] \\
	\lim_{\Delta^{\tn{op}}} \mathbf G^{\tn{fact}}(\tn{Gr}_{G'\times T_2^{\bullet}, \tn{Ran}}) \ar[r]^-{\Psi_{\mathbf G, G'\times T_2^{\bullet}}} & \lim_{\Delta^{\tn{op}}} \Theta_G(\Lambda_{T'\times T_2^{\bullet}}; \mathbf G)
}
$$
Lemma \ref{lem-tau-surj} shows that the left vertical arrow is an equivalence. A direct argument shows that the right vertical arrow is an equivalence as well. Therefore, in proving that $\Psi_{\mathbf G, G}$ is an equivalence, we may assume:

\smallskip

\emph{---the derived subgroup $G_{\tn{der}}$ is simply connected}.

\subsubsection{} Under this assumption, we can write $T_1 = G/G_{\tn{der}}$ and $\Lambda_{T_1}$ is isomorphic to $\pi_1G$.

\begin{lem}
Suppose $G_{\tn{der}}$ is simply connected. Then $\Psi_{\mathbf G, G}$ is an equivalence.
\end{lem}

\begin{proof}[Fully faithfulness]
Since $\Psi_{\mathbf G, G}$ is a morphism of strict Picard $2$-groupoids, it suffices to show that $\Psi_{\mathbf G}$ has contractible fiber at $\mathbf 0 \in \Theta_G(\Lambda_T; \mathbf G)$. Let $(\cal G; \alpha)$ be an object of the fiber, so $\cal G\in\mathbf G^{\tn{fact}}(\op{Gr}_{G, \tn{Ran}})$ and $\alpha$ is a trivialization of its image $(q, \cal G^{(\lambda)}, \varepsilon) \in \Theta_G(\Lambda_T; \mathbf G)$. Since $q = 0$, Lemma \ref{lem-fiber-quad-der} implies that $\cal G$ descends to a factorization gerbe $\cal G_1$ over $\tn{Gr}_{T_1, \tn{Ran}}$.

\smallskip

By the classification for tori (\S\ref{sec-classification-tori}), we see that $\cal G_1$ corresponds to an object in $\Theta(\Lambda_T; \mathbf G)$ with vanishing quadratic form, i.e., an object of $\mathbf{Hom}(\Lambda_{T_1}, \mathbf G(X))$. In particular, the datum of the trivialization $\alpha$ is equivalent to a trivialization of $\cal G_1$.
\end{proof}

\begin{proof}[Essential surjectivity]
We have a morphism between fiber sequences of strict Picard $2$-groupoids, where the top fiber sequence comes from Lemma \ref{lem-fiber-quad-der} and the classification for tori.
$$
\xysmall{
	\mathbf{Hom}(\Lambda_{T_1}, \mathbf G(X))\ar[r]\ar[d]^{\cong} & \mathbf G^{\tn{fact}}(\op{Gr}_{G, \tn{Ran}})\ar[d]^{\Psi_{\mathbf G}} \ar[r]^-{\alpha} & \cal Q(\Lambda_T; A(-1))_{\tn{restr}}^W \ar[d]^{\cong} \\
	\mathbf{Hom}(\Lambda_{T_1}, \mathbf G(X))\ar[r] & \Theta_G(\Lambda_T; \mathbf G) \ar[r] & \cal Q(\Lambda_T; A(-1))_{\tn{restr}}^W
}
$$
By the $4$-lemma, it is enough to show that $\alpha$ is surjective. We note that the determinant line bundle construction \eqref{eq-det-line-bundle} gives a section:
$$
\xysmall{
	& \bigoplus_{s\in\mathbf S} A(-1)\ar[dl]_{\det}\ar[d]^{\tn{Kil}} \\
	\mathbf G^{\tn{fact}}(\op{Gr}_{G, \tn{Ran}}) \ar[r]^-{\alpha} & \cal Q(\Lambda_T; A(-1))^W_{\tn{restr}}
}
$$
Thus, by Lemma \ref{lem-quad-form}, it remains to consider quadratic forms pulled back from $\cal Q(\Lambda_{T_1}; A(-1))$. However, each such form $q$ lifts to some $\Theta$-datum $(q, \cal G^{(\lambda)})\in\Theta(\Lambda_{T_1}; \mathbf G)$ after choosing a square root $\frac{1}{2}q$. Indeed, such choice is possible because $\Lambda_{T_1}$ is free and $A(-1)$ is divisible. We are thus done by the section $\nu$:
$$
\xysmall{
	& \Theta(\Lambda_{T_1}; \mathbf G)\ar[dl]_{\nu}\ar[d] \\
	\mathbf G^{\tn{fact}}(\op{Gr}_{G, \tn{Ran}}) \ar[r]^-{\alpha} & \cal Q(\Lambda_T; A(-1))^W_{\tn{restr}}
}
$$
constructed by composing the equivalence $\Psi_{\mathbf G, T_1}^{-1} : \Theta(\Lambda_{T_1}; \mathbf G) \xrightarrow{\sim} \mathbf G^{\tn{fact}}(\op{Gr}_{T_1, \tn{Ran}})$ with the pullback along $\op{Gr}_{G, \tn{Ran}} \rightarrow \op{Gr}_{T_1, \tn{Ran}}$.
\end{proof}
\qed(Theorem \ref{thm-gerbe-classification})

\bigskip

\bibliographystyle{amsplain}
\bibliography{../biblio}

\end{document}